\documentclass[12pt]{amsart}
\usepackage[utf8]{inputenc}
\usepackage{amsmath}
\usepackage{amsfonts}
\usepackage{amssymb}
\usepackage{amsthm}
\usepackage{IEEEtrantools}
 \usepackage{bbold}
\usepackage{mathrsfs}

\newcommand{\Cc}{\mathbb{C}}
\newcommand{\Nn}{\mathbb{N}}
\newcommand{\Zz}{\mathbb{Z}}
\newcommand{\Rr}{\mathbb{R}}

\newcommand{\Med}{\mathrm{Med}}
\newcommand{\Aut}{\mathrm{Aut}}

\newcommand{\Induce}{\mathrm{Ind}}

\newcommand{\tp}{\mathrm{tp}}
\newcommand{\acl}{\mathrm{acl}}
\newcommand{\bdd}{\mathrm{bdd}}
\newcommand{\dcl}{\mathrm{dcl}}
\newcommand{\e}{\epsilon}

\newcommand{\I}{\mathbb{1}}

\def\Ind#1#2{#1\setbox0=\hbox{$#1x$}\kern\wd0\hbox to 0pt{\hss$#1\mid$\hss}
\lower.9\ht0\hbox to 0pt{\hss$#1\smile$\hss}\kern\wd0}

\def\ind{\mathop{\mathpalette\Ind{}}}

\def\notind#1#2{#1\setbox0=\hbox{$#1x$}\kern\wd0
\hbox to 0pt{\mathchardef\nn=12854\hss$#1\nn$\kern1.4\wd0\hss}
\hbox to 0pt{\hss$#1\mid$\hss}\lower.9\ht0 \hbox to 0pt{\hss$#1\smile$\hss}\kern\wd0}

\begin{document}

\newtheorem{lemma}{Lemma}[section]
\newtheorem{theorem}[lemma]{Theorem}
\newtheorem{definition}[lemma]{Definition}
\newtheorem{example}[lemma]{Example}
\newtheorem{corollary}[lemma]{Corollary}
\newtheorem{proposition}[lemma]{Proposition}
\newtheorem{claim}{Claim}[lemma]

\title{Piecewise Interpretable Hilbert Spaces}
\author{Alexis Chevalier, Ehud Hrushovski}
\date{}

\begin{abstract} 
We study Hilbert spaces $H$ interpreted, in an appropriate sense, in a first-order theory.   Under a new finiteness hypothesis that we call {\em scatteredness}   we  prove  that $H$  is a direct sum of {\em asymptotically free} components, where short-range interactions are controlled by algebraic closure and long-range interactions vanish.
Examples include $L^2$-spaces relative to Macpherson-Steinhorn definable measures; $L^2$ spaces relative to the Haar measure of the absolute Galois groups;    irreducible unitary representations   of $p$-adic Lie groups;  and unitary representations of the automorphism group of an $\omega$-categorical theory.   In the last case, our main result specialises to a theorem of Tsankov.    
New methods are required, making essential use of local stability theory in continuous logic.
\end{abstract}

\maketitle

\tableofcontents

In this paper, we define piecewise interpretable Hilbert spaces and study some of their properties. These are Hilbert spaces which arise as   direct limits of imaginary sorts of a first-order theory $T$.  While we employ continuous logic,
the case where $T$ itself is a discrete first order logic theory is already of  interest, and the introduction may be read with such theories in mind.    We will use tools of stability theory to study the structure of such Hilbert spaces, obtaining information about the underlying theory.  The  stability emanates from the Hilbert space inner product formulas themselves, so no  stability assumptions on the theory $T$ are required.

Definable measures, which themselves play a central role in recent model-theoretic literature, provide one rich source of
examples.    If $\mu$ is such a measure, the Hilbert space $L^2(\mu)$ is piecewise interpretable in $T$.  The stability-theoretic viewpoint already gives useful information here; notably it was used previously  to prove an independence, or 3-amalgamation theorem for definable measures (see \cite{HrushovskiApproxEq}; a more basic version was the main engine of  \cite{Hrushovski2012}, Theorem 2.22).

 A different class of piecewise interpretable Hilbert spaces arises via the absolute Galois group $G$ of a field $K$, with $T=Th(K, K^{alg})$. To put it roughly, a classical result of \cite{Cherlin1980} shows that  $G$ can be obtained in $T$ as an inverse limit of a  $\bigvee$-definable family of finite groups. Here no   definable measure need be present,  but the  Hilbert space $L^2(G,{\rm Haar})$ is piecewise interpretable in $T$. The conjugation invariant part $L^2(G^n, \mathrm{Haar})^G$ of $L^2(G^n, \mathrm{Haar}$) is also piecewise interpretable in $Th(K)$. This opens the way to a common treatment of some  analogies between    fields and measurable structures, such as the independence theorem; the striking similarity was previously unexplained.  
 
 Any  piecewise interpretable Hilbert space gives rise to a   functor from the category of models of a theory $T$ to the category of Hilbert spaces.  In particular, we obtain a homomorphism from the automorphism group $G$ of any model to the unitary group of a Hilbert space, i.e. a unitary reprsentation of $G$.  Conversely, a basic lemma of   \cite{Tsankov2012} implies in the $\omega$-categorical case that   all unitary representations of the automorphism group of the countable model arise in this way, giving a   third and very interesting connection to a deep field.

Our initial  inspiration was  the classification theorem for unitary representations of automorphism groups of
$\omega$-categorical structures from \cite{Tsankov2012}.    Changing the viewpoint to that
of piecewise interpretable Hilbert spaces, it becomes natural to  ask whether they   admit a structure theorem under more general hypotheses than $\omega$-categoricity; 
and if so, what form the statement would take.   The answers we found were, respectively, the notions of {\em scatteredness} and  {\em asymptotic freedom}.

 Both scatteredness and asymptotic freedom concern a complete type $p$ within a piecewise  interpretable Hilbert space $H$.
{\em Scatteredness} is defined in \ref{definition: scattered}; we mention for now a special case (referred to as `strict interpretability' in \ref{definition: strictly interpretable Hilbert space}):    whenever
 there are only finitely many values achieved by the  inner product  between elements of $p$, $p$ is scattered.    This subclass already includes a rich class of examples: these include the $\omega$-categorical case of \cite{Tsankov2012}, the Hilbert spaces arising from
definable measures over pseudo-finite fields (\cite{Chatzidakis1992}) or over measurable classes
(\cite{Macpherson2008}), and the Galois-theoretic examples mentioned above.

$p$ is {\em asymptotically free} if for any $a \in p$, and for any $b \in p$ that is not algebraic over $a$, the elements $a,b$ are orthogonal as  vectors in the Hilbert space.     
   In general, asymptotic freedom means that the Hilbert space structure is defined using only information within the algebraic closure (bounded closure in the continuous logic case) of an element of $p$.  
In fact, the interpretation of the subspace generated by $p$ factors through a   disintegrated, strongly minimal reduct of $T$ with $p$ as its universe (see \cite{spinoff} for a discussion).

 We prove a number of structure theorems that analyse scattered representations in terms of asymptotically free ones. Our core result is  Theorem \ref{theorem: summary of general scattered  interpretable hilbert space}: 
 
 \medskip
 
\begin{changemargin}{0.7cm}{0.7cm}
\emph{Let $\mathcal{H}$ be a piecewise interpretable Hilbert space in $T$. Let $\mathcal{H}_p$ be a piecewise $\bigwedge$-interpretable subspace of $\mathcal{H}$ generated by a scattered type-definable set $p$. Then $\mathcal{H}_p$ is the orthogonal sum of  piecewise $\bigwedge$-interpretable subspaces $(\mathcal{H}_\alpha)_{\alpha < \kappa}$ such that for all $\alpha< \kappa$, $\mathcal{H}_\alpha$ is generated by an asymptotically free complete  type.}
\end{changemargin}

%

 \medskip
 
  \noindent In particular, Theorem \ref{theorem: summary of general scattered  interpretable hilbert space} fully recovers Tsankov's structure theorem in the classical logic $\omega$-categorical case (see Section \ref{subsection: omega-categorical structures}). 
 
Even with the above mentioned  notions at hand however, the proof is not a direct generalisation from the $\omega$-categorical case.  It proceeds instead via a local stability analysis.   The interaction of the stable but highly non-discrete Hilbert space formulas, with the type provided by the underlying theory $T$ turns out to imply not only  rank properties but also a certain local modularity of forking that forms the key to the later analysis. A general discussion of these rank properties can be found in \cite{spinoff}.

Using a theorem of Howe and Moore, we  show that for any algebraic group $G$ over $\mathbb{Q}$, any irreducible representation of $G(\mathbb{Q}_p)$ or $G(\Rr)$ can essentially be obtained as a piecewise interpretable Hilbert space generated by an asymptotically free type 
 (see Section \ref{subsection: application to scattered representations} for details).    As the connection is made directly to  the conclusion of our theorem,  we do not obtain  any new  implications for the irreducible representations of these classical  groups;   but this does show that asymptotic freedom includes both settings for the unitary representation theory of  oligomorphic groups as well as algebraic groups.

A significant chapter of abstract model theory concerns interpretable groups,  usually referred to as definable groups.   First studied for their own sake, especially in the stable context, the study of groups interpreted in a theory $T$ was found to return significant structural information on $T$.  Our treatment of Hilbert spaces as objects of the definable world
is partly inspired by this analogy. We hope they may prove to illuminate other aspects of the theory,   and we view our results as an indication of the possibility of such a development.

 At the final stages of writing up this paper we became aware of Ibarluc\'ia's beautiful paper \cite{Ibarlucia2021}.   
 He employs the same philosophy  of using stability theoretic ideas to study interpretable representations.      
   Ibarluc\'ia  heads towards  a proof of Property (T) for the automorphism group of a continuous logic $\omega$-categorical structure
   by a route intentionally avoiding structure theorems, which are for us the main goal.  In terms of methods, he is able to use stability theory over a model, whereas for us treating canonical bases that are far from models is of the essence.  

The setup and main notions encountered in this paper were arrived at  by joint work, while the first author was a DPhil student under the supervision of the second.  The proof  of the main structure theorem relating scatteredness to asymptotic freedom, including the idea of invoking Von Neumann's lemma to prove one-basedness, are due to the first author in their entirety.

The first  author wishes to thank his PhD mate Arturo Rodriguez Fanlo for numerous helpful discussions.

%
%

\medskip

\begin{center}
\emph{Structure of the Paper}
\end{center}

\medskip

In Section \ref{section: exposition of interpretable Hilbert spaces}, we define and prove some general model theoretic facts about  piecewise interpretable Hilbert spaces. We give examples of piecewise interpretable Hilbert spaces in Section \ref{subsection: examples of interpretable Hilbert spaces}.

In Section \ref{section on structure} we begin a study of the fine structure of piecewise interpretable Hilbert spaces under the assumption of scatterdness, defined in Definition \ref{definition: scattered}. Asymptotic freedom is defined in Definition \ref{definition: asymptotically free}.
 Our main structure theorem is Theorem \ref{theorem: summary of general scattered  interpretable hilbert space}. 
 In Section \ref{subsection: strictly interpretable}, we study the case where $T$ is a classical logic theory and $\mathcal{H}$ is determined by classical logic formulas, in which case we say $\mathcal{H}$ is strictly interpretable. We will   extend Theorem \ref{theorem: summary of general scattered  interpretable hilbert space}  to this context in Corollary \ref{corollary: structure theorem for NFCP theories}.
In Section \ref{subsection: examples of decompositions} we give some concrete examples of the decomposition promised by Theorem \ref{theorem: summary of general scattered  interpretable hilbert space}. 

In Section \ref{section: L^2 spaces}, we consider two rich sources of piecewise interpretable Hilbert spaces: absolute Galois groups and measurable theories. In Section \ref{subsection: galois groups}, we recall a classical result of \cite{Cherlin1980} and give explicit decomposition theorems for the associated $L^2$-spaces, thus strengthening the results of Section \ref{section on structure}.

In Section \ref{subsection: definable measures}, we study some properties of the $L^2$-spaces associated to definable measures. We prove a strong germ property for pseudofinite fields and $\omega$-categorical measurable structures and we show a strong analogy between the probability algebras associated to the Galois groups of Section \ref{subsection: galois groups} and $\omega$-categorical measurable structures: both of these are  inverse limits of $2$-regular inverse systems of finite probability spaces, as defined in Definition \ref{definition: inverse system of finite probability spaces}.



In Section \ref{section: unitary representations} we study  piecewise interpretable Hilbert spaces from the point of view of  the theory of unitary group representations and we establish connections to classical theorems of Howe and Moore. In Section \ref{subsection: representations of automorphism groups} we study the connection between piecewise interpretable Hilbert spaces and irreducibility of the associated unitary representations. In Section \ref{subsection: application to scattered representations}, we study the representations associated to asymptotically free complete types, we prove a Mackey-style irreducibility criterion for these representations and  we show thanks to the theorem of Howe-Moore that these interpretable Hilbert spaces  capture the unitary representations of $p$-adic Lie groups. In Section \ref{subsection: omega-categorical structures} we discuss representations of automorphism groups of $\omega$-categorical structures.


\medskip

This paper uses continuous logic for metric spaces but we will not be reliant on any advanced results of continuous logic. In the first Appendix, we have included an exposition of continuous logic in the style of \cite{HensonIovino2002} or \cite{Gomez2021} and we list there the basic notions which we will use in this paper. The readers who are familiar with \cite{ChangKeisler1966}, \cite{HensonIovino2002} or \cite{BenYaacov2008} can safely ignore this appendix. Outside of the basic notions presented in the appendix, we will develop all the theory we need in this paper  as we go along.

 

%

In the second appendix, we include some additional discussion of the direct limit construction which underlies piecewise interpretable Hilbert spaces. This discussion is very technical in nature but not especially challenging, so it is relegated to an appendix.

\medskip


\section{Introduction to Piecewise Interpretable Hilbert Spaces}\label{section: exposition of interpretable Hilbert spaces}

\emph{In this section, we give a general exposition of piecewise interpretable Hilbert spaces. We use the GNS construction to show that piecewise interpretable Hilbert spaces are defined at the level of the theory $T$. In Section \ref{subsection: examples of interpretable Hilbert spaces} we give some basic examples of piecewise interpretable Hilbert spaces and in Section \ref{subsection: prolonged interpretation} we give an alternative treatment of piecewise interpretable Hilbert spaces which allows us to deduce an important proposition about forking independence in Section \ref{subsection: forking independence}. }

\subsection{Basic definitions}\label{subsection: definition of interpretable Hilbert spaces}

In this section, $M$ is a continuous logic structure in an arbitrary language. In this paper, \emph{definable} always means \emph{$\emptyset$-definable.}
\begin{definition}\label{definition: piecewise interpretable hilbert space}
A piecewise interpretable Hilbert space  $H(M)$ in $M$ is a  direct limit of imaginary sorts $(M_j)_{j\in J}$ of $M$ such that $H(M)$ is a Hilbert space  and the inner product  is definable between all the pieces $M_j$. More explicitly, for all $i,j\in J$, writing $h_i, h_j$ for the direct limit maps from $M_i$ and $M_j$ to $H(M)$, the map $ M_i \times M_j \to \Rr, (x, y) \mapsto \langle h_i x, h_j y\rangle$ is definable. 
\end{definition}

In the above definition, $J$ is a directed partial order. When discussing direct limits of sorts $(M_j)_{j \in J}$ of $M$, we always assume that for any $i \leq j$ in $J$, the transition maps $M_i \to M_j$   are definable. We do not require them to be injective. We say that each sort $M_i$ is a \emph{piece} of $H(M)$ and for any $i, j \in J$, we say that the definable map $f_{ij} : M_i \times M_j \to \Rr$ given by $f_{ij}(x, y) = \langle h_i x, h_j y\rangle$ is an \emph{inner product map.}

Definition \ref{definition: piecewise interpretable hilbert space} does not explicitly  require that the sum and scalar multiplication operations on $H(M)$ be definable between the pieces $(M_j)_J$. However, we show in Lemma \ref{lemma: all operations on interpretable Hilbert space are definable} that the Hilbert space operations are always definable   in an appropriate sense.

\medskip

Suppose $H(M)$ is piecewise interpretable in $M$ and suppose that $H(M)$ is the direct limit of the sorts $(M_j)_{j \in J}$. For $i, j \in J$ write  $f_{ij} : M_i \times M_j \to \Rr$ for the inner product maps. Let $T$ be a continuous logic theory, not necessarily complete, such that $M \models T $, the transition maps $M_i \to M_j$ and the inner product maps $f_{ij}$ are  definable in $T$. Suppose also that $T$ proves that the maps $f_{ij}$ factor through the transition maps of the partial order $(M_j)_J$ and  that for all $j  \in J$, $m\geq 1$ and  $\lambda_1, \ldots, \lambda_m\in \Rr$, 
\[
T \vdash \forall x_1 \ldots \forall x_m \sum_{n, p\leq m}\lambda_n \lambda_p f_{jj}(x_n, x_p) \geq 0
\]
where the variables $x_j$  range in the sort of the piece $M_i$. For every $N \models T$, write $N_j$ for the sort in $N$ corresponding to $M_j$. Since all data is definable, the direct limit $(N_j)_{j \in J}$ with the same inner product maps is a piecewise interpretable Hilbert space $H(N)$ in $N$.

Therefore, $H(M)$ gives rise to a \emph{piecewise interpretable Hilbert space $\mathcal{H}$ in $T$}. We will be careful to distinguish a piecewise interpretable $\mathcal{H}$ in $T$ and its interpretation $H(M)$ in a model $M$ of $T$. Many properties of $H(M)$ transfer to $\mathcal{H}$ and vice versa. 
\medskip
 
A useful way of defining piecewise interpretable Hilbert spaces is given by the GNS theorem (named after Gelfand, Naimark and Segal, see Appendix C in \cite{Bekka2008}), which we recall below.

\begin{definition}
Let $X$ be a set. A function $f: X \times X \to \Rr$ is said to be positive-semidefinite if $f$ is symmetric and for all $n \geq 1$, for all $x_1, \ldots, x_n\in X$ and for all $\lambda_1, \ldots, \lambda_n \in \Rr$, we have $\sum_{i,j} \lambda_i \lambda_j f(x_i, x_j) \geq 0$.
 \end{definition}

\begin{theorem}[GNS Theorem]\label{theorem: classical GNS theorem}
Let $X$ be a set and let $f : X \times X \to \Rr$ be positive-semidefinite. Then  there is a Hilbert space $H$ and a map $F : X \to H$ such that $F(X)$ has dense span in $H$ and for all $x, y\in X$, $\langle F(x), F(y) \rangle = f(x, y)$.

$F$ and $H$ are unique in the sense that if $F'$ and $H'$ satisfy the same statement, then there is a surjective unitary map $G : H \to H'$ such that $F' = G \circ F$. 

\end{theorem}

We will adapt the GNS theorem to our context:

\begin{proposition}\label{proposition: equivalent definition of piecewise interpretable hilbert space}
Suppose  we are given a Hilbert space $H$,  a collection of imaginary sorts $(M_i)_{i \in I}$ of $M$ and functions $F_i : M_i \to H$ for all $i\in I$ such that 
 for all $i,j \in I$, the map $M_i \times M_j \to \Rr, (x, y)\mapsto \langle F_i x, F_j y\rangle$ is definable.  

Then there is a piecewise interpretable Hilbert  space $H(M)$ in $M$ such that 
\begin{enumerate}
\item the  sorts $(M_i)_{I}$ are pieces of $H(M)$
\item writing $h_i : M_i \to H(M)$ for the direct limit maps,    $\bigcup h_i(M_i)$ has dense span in $H(M)$
\item  for any $i, j \in I$   and $x \in M_i$, $y\in M_j$, we have  $\langle h_i x, h_j y\rangle_{H(M)} = \langle F_ix , F_j  y\rangle_H$.
\end{enumerate}
\end{proposition}

\begin{proof}

 To make notation lighter, we write  $F$ for  all functions $F_i$, $i\in I$. It will always be clear from the variable which map $F_i$ we are using. By passing to a closed subspace of $H$, we can assume without loss of generality that $\bigcup F(M_i)$ has dense span in $H$. Let $x\in H$. There is an increasing function $\eta : \Nn \to \Nn$ and a uniformly Cauchy sequence  $(\sum_{i = \eta(n)}^{\eta(n+1)-1} \lambda^n_i Fx^n_i)$ which converges to $x$ such that $x^n_i\in \bigcup_j M_j$. We can assume that  $\eta(n)$ is large enough so that $|\lambda_i^n |\leq \eta(n)$ for all $i \leq n$. We will decompose $H$ according to the growth rate of $\eta$.

Fix $\eta : \Nn \to \Nn$ strictly increasing and fix an arbitrary countable sequence $(i_n)$ in $I$. We define an imaginary sort $M^\eta_{(i_n)}$ as follows.  $M_{(i_n)}^\eta$ will be the metric completion of the  countable Cartesian product   $\prod_{n \geq 0} \big(\prod_{k = \eta(n)}^{\eta(n+1) -1} [-\eta(n), \eta(n)] \times M_{i_k} \big)$ under the metric $\delta$  defined below. Take $(\overline{x}_n) \in M_{(i_n)}^\eta$ and write $\overline{x}_n = (\lambda_{\eta(n)}, x_{\eta(n)}, \ldots, \lambda_{\eta(n+1)-1}, x_{\eta(n+1) - 1})$.  
For all $n$, define 
\[
F(\overline{x}_n) = \sum_{i = \eta(n)}^{\eta(n+1) - 1} \lambda_i Fx_i.
\]
 We define inductively maps $g_k : M_{(i_n)}^\eta \to H$. $g_0((\overline{x}_n))$ is just $F(\overline{x}_0)$. Given $g_k$, define $g_{k+1}((\overline{x}_n)) = F(\overline{x}_{k+1})$ if $\|g_k((\overline{x}_n)) - F(\overline{x}_{k+1})\| \leq 2^{-n}$. Otherwise, define\footnote{\label{footnote: compare with BY definition} This is analogous to Definition 3.6 in \cite{BenYaacov2010}} 
\[
g_{k+1}((\overline{x}_n)) = g_k((\overline{x}_n)) + \frac{2^{-n}}{\|F(\overline{x}_{k+1}) - g_k((\overline{x}_n))\|}(F(\overline{x}_{k+1}) - g_k((\overline{x}_n)))
\]
Then $(g_k((\overline{x}_n)))_k$ is a uniformly Cauchy sequence in $H$. It is straightforward to check that for all $k$, the map $((\overline{x}_n), ( \overline{y}_n)) \mapsto \langle g_k((\overline{x}_n)), g_k((\overline{y}_n))\rangle$ is definable. We obtain a definable map 
\[
\beta: ((\overline{x}_n), ( \overline{y}_n) )\mapsto  \lim_k  \langle g_k((\overline{x}_n)), g_k((\overline{y}_n))\rangle.
\]
 $\beta$ is positive-semidefinite and hence $\beta$ induces a pseudo-metric $\delta$. We quotient $M^\eta_{(i_n)}$ by this pseudo-metric and take the metric completion, so that $M_{(i_n)}^\eta$ is identified with a subset of $H$.

We now define the direct limit structure. Fix an arbitrary  ordering $\leq_I$ of $I$. Let $J$ be the set of pairs $(\eta, (i_n))$ such that $\eta : \Nn \to \Nn$ is strictly increasing and $(i_n)$ is a sequence in $I$ such that for all $n$, $(i_{\eta(n)}, \ldots, i_{\eta(n+1)-1})$ is increasing with respect to the order $\leq_I$. We could have chosen such an $(i_n)$ when we constructed $M_{(i_n)}^\eta$ above.   We define a partial ordering on $J$ as follows: we say that $(\eta, (i_n)) \leq (\mu, (j_n))$ if and only if for all $n$, $\eta(n+1) - \eta(n) \leq \mu(n+1) - \mu(n)$ and $(i_{\eta(n)}, \ldots, i_{\eta(n+1)-1})$ is a subtuple of $(j_{\mu(n)}, \ldots, j_{\mu(n+1)-1})$. We have definable maps $M^\eta_{(i_n)} \to M^\mu_{(j_n)}$ for $(\eta, (i_n)) \leq (\mu, (j_n))$ by taking the obvious inclusions and by using the scalar $0$ to pad the image of $M_{(i_n)}^\eta$ in $M_{(j_n)}^\mu$. $H(M)$ is defined as the direct limit of the sorts $M_{(i_n)}^\eta$.  
\end{proof}

In Appendix \ref{appendix: more theory}, we define the notion of embedding and isomorphism of piecewise interpretable Hilbert spaces and we show that the space $H(M)$ constructed in Proposition \ref{proposition: equivalent definition of piecewise interpretable hilbert space} is unique up to isomorphism.

In this paper, we are only interested in discussing piecewise interpretable Hilbert spaces in $M$ up to isomorphism. Therefore, in order to fix a piecewise interpretable Hilbert space $H(M)$, it will be enough to specify a pair $(H, (h_i)_{i\in I})$ where $H$ is a  Hilbert space  and the maps $h_i : M_i \to H$ are  as in Proposition \ref{proposition: equivalent definition of piecewise interpretable hilbert space} and the span of $\bigcup h_i(M_i)$ is dense in $H$.

 Equivalently, a piecewise interpretable Hilbert space $H(M)$ can   be described up to isomorphism by fixing   sorts $(M_i)_I$ of $M$ and taking for every pair $i,j \in I$ a definable map $f_{ij} : M_i \times M_{j} \to \Rr$ such that the concatenation of all maps $(f_{ij})_I$ is positive-semidefinite on $(\bigcup_I M_i)^2$. These various presentations   correspond to the equivalence of categories discussed in Lemma \ref{lemma: category theoretic reformulation of GNS}.

\medskip

If $H(M)$ is piecewise interpretable in $M$ and is the direct limit of the imaginary sorts $(M_j)_{j \in J}$, we do not require that the direct limit maps $h_j : M_j \to H(M)$ be injective.  This is  because we show in Lemma \ref{lemma: transition maps can be assumed isometric} that $H(M)$ is always isomorphic to a piecewise interpretable $H'(M)$ with isometric  direct limit maps. In this paper we will often move to imaginary sorts so that the direct limit maps can be often be assumed to be isometries. We will always indicate when we move to imaginary sorts.

When the direct limit maps are isometries, we identify the pieces $M_j$ with subsets of $H(M)$.  By a \emph{type-definable subset $p$ of $H(M)$}  we mean a type-definable subset of some piece $M_j$ of $H(M)$ with an isometric direct limit map $M_j \to H(M)$. We stress that a type-definable subset $p$ of $H(M)$ is contained in a single piece of $H(M)$. This will allow us to quantify over the piece containing $p$ and to use compactness arguments.

A type-definable subset of $H(M)$ is not to be confused with a piecewise $\bigwedge$-interpretable   subspace of $H(M)$, which we define as follows:

\begin{definition}\label{definition: generalised piecewise interpretable}
A   piecewise $\bigwedge$-interpretable subspace $V(M)$ of $H(M)$  is a subspace   $V(M)$ such that if $H(M)$ is the direct limit of the sorts $(M_j)_{j \in J}$ with direct limit maps $h_j: M_j \to H(M)$, then for all $j \in J$ the set $h_j^{-1}(V(M))$ is type-definable in $M_j$.

If $p$ is a type-definable subset of some piece $M_j$, we write $H_p(M)$ for the 
  piecewise $\bigwedge$-interpretable subspace of $H(M)$ consisting of the closed span of the set   $h_j(p)$ in $H(M)$. 

If $\mathcal{H}$ is the piecewise interpretable Hilbert space in $T$ corresponding to $H(M)$, we write $\mathcal{H}_p$ for the piecewise $\bigwedge$-interpretable Hilbert space in $T$ corresponding to $H_p(M)$.
\end{definition}
 
In this paper, we always assume that pieces of a piecewise interpretable Hilbert space $H(M)$ are imaginary sorts of $M$. As remarked above, this allows us to quantify over the pieces of $H(M)$. However, as we will see in Section \ref{subsection: omega-categorical structures}, it is often natural to consider piecewise interpretable Hilbert spaces with pieces which are distance-definable sets. We recall   some  notions from \cite{BenYaacov2008} and we introduce a general construction which shows that there is no loss of generality in only considering piecewise interpretable Hilbert spaces whose pieces are sorts of $M$.

\begin{definition}[\cite{BenYaacov2008}   9.16] \label{definition: distance-definable}
Let $T$ be a complete continuous logic theory. Let $M$ be an $\omega_1$-saturated model of $T$.

Let $r$ be a type-definable set. We say that $r$ is distance-definable if the function $d(x, r) = \inf\{d(x, y) \mid y \models r\}$ is definable in $M$, where $x$ is in the sort of $r$.
\end{definition}

%
Distance-definability is the continuous logic equivalent of `definability' in classical logic (as distinguished from `type-definability'). By 9.18 in \cite{BenYaacov2008}, distance-definability of $p$ is not model-dependent and if $r \neq \emptyset$ in $M$, then $r \neq \emptyset$ in any model of $T$.

 \begin{definition}\label{definition: expansion by a distance-definable set}
 Let $T$ be an arbitrary complete continuous logic theory. Let $r$ be a non-empty distance-definable set contained in a sort $S$ of $T$. 
 
 Define an expansion $T^{r}$ of $T$ as follows: we add a sort $X_r$ and a map $f_r : X_r \to S$. $T^r$ extends $T$ and says that $f_r$ is an isometry with dense image in $r$, i.e. for every $0 < \e $, if $d(x, r) < \e$ then there is $z \in X_r$ with $d(x, f_r(z)) \leq \e$. 
\end{definition}

For every $M \models T$, we write $M^{r}$ for the extension of $M$ to a model of $T^{r}$. Since non-empty distance-definable sets are realised in every model, this extension exists. This extension is clearly unique. By Lemma \ref{lemma: stable embedded conditions}, $T$ is stably embedded in $T^r$. If one wishes to work with a piecewise interpretable Hilbert space $H(M)$ such that $r$ is a piece of $H(M)$, we can move to the theory $T^r$ and work with the sort $X_r$ instead. This shows that there is no loss of generality in assuming that piecewise interpretable Hilbert spaces always have entire sorts as pieces.

\medskip

\noindent\emph{\textbf{Convention:} In this paper,  we will only consider `piecewise interpretable Hilbert spaces' and `piecewise $\bigwedge$-interpretable subspaces', so we will now refer to them simply as `interpretable'  or `$\bigwedge$-interpretable'.
}

\subsection{Examples}\label{subsection: examples of interpretable Hilbert spaces}

 \emph{We give some elementary examples which illustrate a variety of interpretable Hilbert spaces.   See Sections \ref{section: L^2 spaces} and \ref{section: unitary representations} for rich sources of examples which are of wider relevance to model theory and representation theory.} 

\medskip

 1. In classical logic, let $T^\infty$ be the  theory of an infinite set. Writing $S$ for the main sort of $T$, we define the inner product map $f : S^2 \to \Rr$ by $f(x, x) = 1$  and $f(x,y) = 0$ if $x\neq y$. This gives an interpretable Hilbert space   $\mathcal{H}$ such that for any $M \models T^{\infty}$, $M$ is an orthonormal set  in $H(M)$ with dense span. 

Define also the inner product maps $g(x, x) = 2$ and $g(x, y) = 1$ if $x \neq y$. Define also $h(x, x) = 4$ and $h(x, y) = 3$ if $x\neq y$. These also give interpretable Hilbert spaces $\mathcal{H}'$ and $\mathcal{H}''$ respectively. Observe that $\mathcal{H}'$ and $\mathcal{H}''$ are isomorphic, but they are not isomorphic to $\mathcal{H}$. One way of proving this is to note that for any $M \models T$, $H'(M)$ and $H''(M)$ have an invariant vector under the action of $\Aut(M)$, but this is not true of $H(M)$. This will be discussed further in Section \ref{section: unitary representations}.

\medskip

2. Let $T = Th(\Zz, \leq)$. Define $f(x, x) = 2$, $f(x, y) = 1$ if $x,y$ are consecutive, and $f(x, y) = 0$ otherwise. $f$ is positive definite and defines an interpretable Hilbert space.

For a more complicated example with the same flavour, let $V = \ell^2(\Zz)$ and for $n \in \Zz$ let $v_n = (2^{-|k+n|})_{k \in \Zz}$. The sequence $(v_n)_{\Zz}$ generates $\ell^2(\Zz)$.  Define the map $h : \Zz \to V, n \mapsto v_n$.  Then for $x,y \in \Zz$, $\langle hx, hy\rangle$ depends only on the distance between $x$ and $y$, so the inner product is definable on $\Zz$ in $Th(\Zz, \leq)$. Then  we are in the situation of Proposition \ref{proposition: equivalent definition of piecewise interpretable hilbert space} so $h$ induces an interpretation of $\ell^2(\Zz)$ in $(\Zz, \leq)$.

For yet another example, let $S$ be an  arc of the circle $S^1$. $\Zz$ acts on $L^2_\Cc(S)$ via $f \mapsto z^nf$. Let $V$ be the subspace of $L^2_\Cc(S)$ generated by the orbit of $1$ under $\Zz$. Then Proposition \ref{proposition: equivalent definition of piecewise interpretable hilbert space} shows that $V$ is interpretable in $\Zz$.

\medskip


\medskip

3. Given two interpretable Hilbert spaces $\mathcal{H}$ and $\mathcal{H'}$, we can form their sum $\mathcal{H + H'}$ as follows. Say $\mathcal{H}$ and $\mathcal{H}'$ are the direct limits of $(S_i)_{i \in I}$, $(S_{j})_{j \in J}$  with direct limit maps $h_i : S_i \to \mathcal{H}$ and $h_j : S_j \to \mathcal{H}'$ respectively. For any $M \models T$, let $H$ be the orthogonal sum of $H(M)$ and $H'(M)$. Then the system of maps $h_i : M_i \to H$ and $h_j : M_j \to H$ is as in Proposition \ref{proposition: equivalent definition of piecewise interpretable hilbert space} and hence they define an interpretable Hilbert space $\mathcal{H}+ \mathcal{H}'$ in $T$. It is clear that for all $N \models T$, $(H + H')(N)$ is the orthogonal sum of $H(N)$ and $H'(N)$.

We can also define the orthogonal sum of infinitely many interpretable Hilbert spaces in the same way. In Section \ref{section on structure}, we will see that it is sometimes possible to recognise an interpretable Hilbert space in $T$ as the orthogonal sum of a family of interpretable Hilbert spaces with interesting properties.

\medskip

4. Suppose that $\mathcal{H}$ and $\mathcal{H}'$ are interpretable in $T$. Then their tensor $\mathcal{H} \otimes \mathcal{H}'$ is interpretable in $T$. For any $M \models T$, write $H(M)$ and $H'(M)$ as the direct limits of $(M_j)_{j \in J}$ and $(M_{j'})_{j' \in J'}$ respectively. For any $j \in J$ and $j' \in J'$, define the map $F_{j, j'} : M_j \times M_j' \to H(M) \otimes H'(M)$ by $F_{j, j'}(x, y) = h_jx \otimes h'_{j'}y$. The image of the maps $(F_{j, j'})$ have dense span in $H(M) \otimes H'(M)$ and for any $(j_1, j_1')$ and $(j_2, j_2')$, 
\[
\langle F_{j_1, j_1'} (x_1, y_1), F_{j_2, j_2'}(x_2, y_2)\rangle = \langle h_{j_1}x_1, h_{j_2}x_2\rangle_{H(M)} \langle h_{j_1'} x_1', h_{j_2'}x_2'\rangle_{H'(M)}.
\]
 This is a definable map $M_{j_1} \times M_{j_1'} \times M_{j_2} \times M_{j_2'} \to \Rr$ so Proposition \ref{proposition: equivalent definition of piecewise interpretable hilbert space} applies.


\medskip

\subsection{Prolonging piecewise interpretable Hilbert spaces}\label{subsection: prolonged interpretation}

 Let $\mathcal{H}$ be an interpretable Hilbert space in $T$. In this section, we fix a family of pieces $(S_i)_{ i \in I}$ of $\mathcal{H}$ with direct limit maps $h_i$ and inner product maps $f_{ij}$ ($i, j \in I$) such that $\bigcup h_i(S_i)$ has dense span in $\mathcal{H}$. $I$ is not necessarily a directed partial order. 
 
 We have seen that in practice $\mathcal{H}$ is often characterised by such a collection of pieces.
  In contrast, the full direct limit structure of $\mathcal{H}$ can be complicated to describe. Therefore, we present here a construction   which allows us to work with interpretable Hilbert spaces in a setting which is closer to the classical presentation of Hilbert spaces in continuous logic. 
 
  Under this construction, the balls $B(0, n)$ in $H(M)$ become subsets of Hilbert space balls which we add to the theory $T$ as new sorts. These subsets are not definable, but this construction helps simplify the discussion of interpretable Hilbert spaces and easily yields the important results about model-theoretic independence in interpretable Hilbert spaces in Section \ref{subsection: forking independence}.

Suppose that $T$ is a theory  in the language $\mathcal{L}$.  We define an extension $T^\mathcal{H}$ of the theory $T$ in a language $\mathcal{L}^\mathcal{H}$ as follows.
We add to $\mathcal{L}$ all the sorts and functions which are used in the   presentation of  Hilbert spaces in continuous logic, as in Appendix \ref{appendix: Hilbert spaces in continuous logic}, and $T^\mathcal{H}$ says that these new sorts form an infinite dimensional Hilbert space.  For each $i \in I$, we also add to our language a function symbol $h_i$ from $S_i$ to one of the Hilbert space balls with radius greater than $\sup_x \sqrt{f_{ii}(x, x)}$. For $i,j \in I$, $T^{\mathcal{H}}$ contains the additional axioms
\[
\forall x \in S_i, \forall y \in S_j, f_{ij}(x, y) = \langle h_i(x), h_j (y)\rangle
\]
where $\langle ., . \rangle$ is the inner product on the ball which $h_i$ maps into.
We also  add axioms saying that the orthogonal complement of $\bigcup h_i(M_i)$ is infinite dimensional. We will refer to the maps $(h_i)$ as the interpretation maps in $T^{\mathcal{H}}$. 


Any  model of $T^{\mathcal{H}}$ has the  form $(M, H )$ where $M\models T$ and where $H$ is an infinite dimensional Hilbert space containing $H(M)$ as a subspace in the obvious way. It follows that any model $(M, H)$ of $T^{\mathcal{H}}$ is uniquely determined by its $T$-part and the dimension of the orthogonal complement of $H(M)$ in $H$. One deduces that $T^{\mathcal{H}}$ is a complete theory by applying standard saturation arguments.
 We now show that  moving from $M \models T$ to the $\mathcal{L}^{\mathcal{H}}$-structure $(M, H)$ does not add any extra structure to $M$. 
 \begin{proposition}\label{proposition: T is stably embedded}
$T$ is stably embedded in $T^{\mathcal{H}}$. 

If $\mathcal{H}$ is the direct limit of the   sorts $(S_j)_{j \in J}$, if the sorts $S_j$ are real sorts of $T$ and if the direct limit map on each $S_j$ is injective, then 
 for any $(M, H) \models T^{\mathcal{H}}$ and $C \subseteq (M, H)$, if $f$ is a $C$-definable function between sorts of $T$ in $\mathcal{L}^{\mathcal{H}}$, then $f$ is $\dcl_{\mathcal{L}^{\mathcal{H}}}(C) \cap M$-definable in $\mathcal{L}$.

Hence, without assuming that $\mathcal{H}$ is the direct limit of real sorts of $T$,  $T$ is fully embedded in $T^{\mathcal{H}}$ in the sense that for any $(M, H) \models T^{\mathcal{H}}$ and  any function $f$ between sorts of $T$ definable in $\mathcal{L}^{\mathcal{H}}$ over $(M,H)$ is definable in $\mathcal{L}$ over $M$.

\end{proposition}

\begin{proof}
First let $(M, H) \models T^{\mathcal{H}}$ be $\kappa$-saturated with $|(M, H)| = \kappa$ where $\kappa \geq \omega_1$. We show that $(4)$ from Lemma \ref{lemma: stable embedded conditions} holds. Let $\alpha$ be an automorphism of $M$. By the GNS-theorem, $\alpha$ induces an automorphism of $H(M)$. Define $\beta: (M, H) \to (M, H)$  extending $\alpha$ as follows:   on $H(M)$, $\beta$ is the Hilbert space isomorphism induced by $\alpha$, and on $H(M)^\perp$, $\beta$ is the identity. Then $\beta$ respects all the basic relations in $\mathcal{L}^{\mathcal{H}}$ so $\beta$ is an isomorphism and $T$ is stably embedded in $T^{\mathcal{H}}$. 

Now assume that the pieces of $\mathcal{H}$ belong to $T$ and that the direct limit maps are injective.
  Let $(M, H)$ be any model of $T^{\mathcal{H}}$.  Let $C \subseteq (M, H)$ and let $f$ be a $C$-definable function into $\Rr$ on a sort of $T$. We can assume that $f (x) = g(x, C)$ where $g$ is $0$-definable and $C$ is a finite tuple of $(M, H)$. We can also write the tuple $C$ as a pair $ab$ where $a \subseteq M$ and $b \subseteq H$. Finally, we can assume that $b = b_0b_1$ where $b_0 \subseteq H(M)$ and $b_1 \subseteq H(M)^\perp$, by expressing elements of $b$ are sums of elements of $b_0$ and $b_1$ and by noting that $b_0b_1$ and $b$ are inter-definable.


Let $(N, H')$ be any elementary extension of $(M, H)$. As above, we can construct an automorphism of $(N, H')$ by taking any automorphism of $N$ and extending it by any unitary automorphism of $H(N)^\perp$. It follows that $\tp(b_1/M)$ is completely determined by the values of the inner product between elements of $b_1$ and by the partial type which says $b_1 \subseteq H(M)^\perp$. In particular $\tp(b_1/M)$ is the unique extension to $M$ of $\tp(b_1)$. It follows   that $f$ is definable over $ab_0$. 

Since we have assumed that the sorts $(S_j)$ are part of $T$  and that the direct limit maps are injective, each element of $b_0$ is inter-definable with an element of $M$. Hence we have proved that $f$ is definable over $\dcl_{\mathcal{L}^{\mathcal{H}}}(C) \cap M$ in $\mathcal{L}^{\mathcal{H}}$. 

In order to show that $f$ is definable over $\dcl_{\mathcal{L}^{\mathcal{H}}}(C) \cap M$ in $\mathcal{L}$, it is enough to note that $S^{\mathcal{L}}(C') \cong S^{\mathcal{L^{\mathcal{H}}}}(C')$ via the restriction map for any $C' \subseteq M$, where $S^{\mathcal{L}}(C')$ and $S^{\mathcal{L}^{\mathcal{H}}}(C')$ are  type-spaces in the sort of $f$ over $C'$ in  $\mathcal{L}$ and  $\mathcal{L}^{\mathcal{H}}$ respectively. This is   seen by noting that in a sufficiently saturated extension $N$ of $M$, if $a,b\in M$ are conjugate over $C'$, then they are conjugate in any $(N, H) \models T^{\mathcal{H}}$.

Now let $f$ be a $C$-definable function between any two sorts $X$ and $X'$ of $T$. We need to show that its graph is $\dcl_{\mathcal{L}^{\mathcal{H}}}(C) \cap M$-definable in $\mathcal{L}$. Our result for functions $X \times X' \to \Rr$ shows that  $S_{X \times X'}^{\mathcal{L}}(\dcl_{\mathcal{L}^\mathcal{H}}(C) \cap M) \cong S_{X \times X'}^{\mathcal{L}^{\mathcal{H}}}(C)$ so $f$ is indeed $\dcl_{\mathcal{L}^\mathcal{H}}(C) \cap M$-definable in $\mathcal{L}$. 

Finally, if we do not assume that $\mathcal{H}$ is the direct limit of real sorts of $T$, then the final part of the proposition follows by choosing parameters in $M$ containing all necessary imaginaries in their definable closure.
 \end{proof}

\medskip

\subsection{Forking independence in interpretable Hilbert spaces }\label{subsection: forking independence}
This section aims to prove Proposition \ref{proposition: transfer of independence} which will be essential in the remainder of this paper. Proposition \ref{proposition: transfer of independence} says that we can use stability of the inner product maps and stable embeddedness of $T$ in $T'$ to obtain full forking independence in the sense of  Lemma \ref{lemma: forking independence in Hilbert spaces}. In this section, $T$ is a complete continuous logic theory.

 We begin with a general lemma.
Throughout this paper, `forking' is always meant with respect to stable  definable functions. 

\begin{lemma}\label{nonforking preserved under reducts}
Let $\mathcal{L}_0$ be a language and $\mathcal{L}$ an extension of $\mathcal{L}_0$, possibly with new sorts. Let  $T$ be a complete $\mathcal{L}$-theory and $T_0$ the reduct of $T$ to $\mathcal{L}_0$. Suppose that $T_0$ has weak elimination of imaginaries.

 Let $(M, N) \models T$ where $M \models T_0$ and $N$  denotes the new sorts in $\mathcal{L}$ added to $\mathcal{L}_0$.  Take $A \subseteq B \subseteq (M, N)$   such that $A = \bdd_{\mathcal{L}}(A)$. Let $p$ be a type over $B$ in a fragment of $\mathcal{L}_0$-formulas which does not fork over $A$ in the sense of $T$. Then  $p$ does not fork over $A\cap M$ in the theory  $T_0$.
\end{lemma}

\begin{proof}

Let $q$ be a non-forking extension of $p$ to $M$ in $\mathcal{L}_0$. We know that $q$ is $\mathcal{L}$-definable over $A$. The usual proof of the theorem which says that stable partial types over models are definable tells us that $q$ is $\mathcal{L}_0$-definable over $M$. Let $f(x, y)$ be a stable $\mathcal{L}_0$-definable function.

Let $\alpha$ be a canonical parameter of $d_qf(y)$, viewed as an $\mathcal{L}_0$-$M$-definable function, so that $\alpha$ is an imaginary element of $M$. Working in $(M, N)^{eq}$, since $q$ is $\mathcal{L}$-definable over $A$, $\alpha$ is in the $\mathcal{L}$-definable closure of $A$. By weak elimination of imaginaries in $T_0$, $\alpha$ is in the $\mathcal{L}_0$-definable closure of some $C \subseteq \bdd_{\mathcal{L}_0}(\alpha)$, so $d_qf(y)$ is $C$-definable in $\mathcal{L}_0$. Since adding imaginaries does not affect the definable or bounded closure, in $(M, N)$ we have $C \subseteq \bdd_\mathcal{L}(A) \cap M = A \cap M$ and $q$ is $\mathcal{L}_0$-definable over $A \cap M$.
\end{proof}


%
%
%
%
%


 If $H(M)$ is intepretable in $M$ and is the direct limit of the sorts $(M_j)_{j \in J}$ where the direct limit maps are injective and $M_j$ is a real sort of $M$, and if $A \subseteq M$ is definably closed, we write   $P_A$ for the orthogonal projection onto the subspace of $H(M)$ given by $A \cap H(M)$.  The following proposition follows from full embeddedness of $T$ in $T^{\mathcal{H}}$ and the charactersiation of forking in Hilbert spaces:

\begin{proposition}\label{proposition: transfer of independence}
Take $M \models T$.  Suppose that $H(M)$ is the direct limit of the sorts $(M_j)_{j \in J}$ and that the direct limit maps are injective. Suppose that the sorts $M_j$ belong to $M$, so that elements of $H(M)$ are identified with elements of $M$.

Let $A$ be a subset of $H(M)$ and let $B \subseteq C \subseteq M$ be $\bdd$-closed. If $A \ind_B C$ in the sense of $T$ with respect to the inner product maps between the pieces $(M_j)$, then for all $a \in A$, $P_{C \cap H(M)} a = P_{B \cap H(M)} a$.  
\end{proposition}  

\begin{proof}
Suppose that $A \ind_B C$ with respect to the inner product maps. Take $(M, H) \models T^{\mathcal{H}}$ where the $T^{\mathcal{H}}$ construction is applied to all pieces $(M_j)$. Let $B'= \bdd_{\mathcal{L}^\mathcal{H}}(B)$ and $C' = \bdd_{\mathcal{L}^\mathcal{H}}(C)$. By full embeddedness of $M$ in $(M, H)$, $B' \cap M = B$. Suppose that $v \in H \cap B' $. Then $v \in H(M)$ and we can find $x \in M$ such that $v = h_jx$. Since the interpretation maps are assumed to be injective, $x \in \dcl_{\mathcal{L}^\mathcal{H}}(B)$. Hence $B'  = \dcl_{\mathcal{L}^\mathcal{H}}(B) $, and similarly for $C'$. 

Write $h(A)$ for the result of mapping $A$ in $H$ by the appropriate interpretation maps $h_j$.  By considering definable bijections, we find that    $h(A) \ind_{B'} C'$ with respect to the maps $\langle x, y\rangle$ where $x$ and $y$ range in the balls of $H$.   Hilbert spaces have weak elimination of imaginaries so Lemma \ref{nonforking preserved under reducts} applies and we have $h(A) \ind_{B' \cap H} C' \cap H$ in the theory of Hilbert spaces.  

Note that $B' \cap H = B' \cap H(M)$ and similarly for $C'$.  Now the lemma follows by the characterisation of forking independence in Hilbert spaces (see Lemma \ref{lemma: forking independence in Hilbert spaces}).
\end{proof}

\section{Structure Theorems for Scattered Interpretable Hilbert Spaces}\label{section on structure}

\emph{In this section, we give our structure theorems for interpretable Hilbert spaces. These reduce the notion of scatteredness to the stronger notion of asymptotic freedom; see the definitions below. The main structure theorem is \ref{theorem: summary of general scattered  interpretable hilbert space}. In section \ref{subsection: strictly interpretable} we show how the decomposition of   Theorem \ref{theorem: summary of general scattered  interpretable hilbert space} can be almost recovered when working in classical logic with the weak NFCP. This is Corollary \ref{corollary: structure theorem for NFCP theories}. In Section \ref{subsection: examples of decompositions} we give concrete examples of the decomposition promised by \ref{theorem: summary of general scattered  interpretable hilbert space} and \ref{corollary: structure theorem for NFCP theories}.}

\medskip 

\subsection{Weak closure, scatteredness and asymptotic freedom}\label{subsection: key notions for structure theorems}

Fix a continuous logic theory $T$  and   take $\mathcal{H}$  an interpretable Hilbert space in $T$.   In this section, we will freely move to imaginary sorts of $T$, so we assume that the   direct limit maps on pieces of $\mathcal{H}$ are isometries, and for any $M \models T$, we identify  pieces of $H(M)$ with subsets of $H(M)$.

When $M \models T$ and $A \subseteq H(M)$, we write $P_A$ for the orthogonal projection onto  the closed subspace of $H(M)$ generated by $A$.

\begin{definition}\label{definition: partial order P}
Let $M \models T$ and let  $p$ be a type-definable set in $H(M)$. 
We define $\mathcal{P}(p) \subseteq H(M)$  to be the closure of the set of realisations of $p$  in the weak topology. 
\end{definition}

The notation `$\mathcal{P}(p)$' stands for the partial order which we will define in Definition \ref{definition: partial ordering}.  

\begin{lemma}\label{lemma: weak closure is definable}
Let $S$ be a piece of $\mathcal{H}$. Then for any $M \models T$, $\mathcal{P}(S)$ is a type-definable set in a piece of $\mathcal{H}$.

Let $p$ be a type-definable set in $\mathcal{H}$. Then for any $\omega$-saturated $M \models T$,  $\mathcal{P}(p)$ is a type-definable set in a piece of $\mathcal{H}$. 
\end{lemma}

\begin{proof}
Let $M \models T$. By Lemma  \ref{lemma: weak closure is set of limit points} $\mathcal{P}(p)$ is the set of weak limit points of sequences in $S$. Moreover, Lemma \ref{lemma: weak closure is set of limit points} shows that for any $w \in \mathcal{P}(p)$ and for any $\e > 0$, there is a sequence $(a_n)$ in $S$ converging weakly to $w$. It is easy to check that the sequence $(\sum_{k = 0}^na_k/n)$ converges to $w$ in norm. Since the inner product is bounded on $S \times S$, we can control the rate of convergence of $(\sum_{k = 0}^na_k/n)$ uniformly for all weakly convergent $(a_n)$ by considering subsequences of the form $(\sum_{k = 1}^{\eta(n)} a_k/\eta(n))$, where $\eta$ is an increasing function $ \Nn \to \Nn$.  It is then straightforward to adapt the construction of Proposition  \ref{proposition: equivalent definition of piecewise interpretable hilbert space} to construct a piece of $H(M)$ containing $\mathcal{P}(S)$.\footnote{\label{footnote: new piece for partial order} this piece may not have been present in the original direct limit of $H(M)$, but in this section we freely add imaginary sorts to our structure.}

If $p$ is a type-definable set in $\mathcal{H}$ contained in a piece $S$, we use the previous construction to construct $\mathcal{P}(S)$ inside a piece of $\mathcal{H}$. For $\omega$-saturated $M \models T$, we can express the fact that $w \in \mathcal{P}(S)$ is the weak limit of a sequence in $p$. 
 \end{proof}
 
\begin{lemma}\label{lemma: base set for partial order}
Let $p$ be a  type-definable set in $\mathcal{H}$. 
When $M$ is $\omega_1$-saturated, $\mathcal{P}(p)$ is equal    to  the set $ \{P_{\bdd(A)}(b)\mid b\models p$,  $A \subseteq M\}$ and is closed under the maps $P_{\bdd(A)}$ for arbitrary $A \subseteq M$.
\end{lemma}
 
\begin{proof}
 
%

  Take $w \in \mathcal{P}(p)$ and $(a_n)$ a   sequence in $p$  converging weakly to $w$. Write $A = \bdd(w)$.  
By saturation, we can find  $b \models p$ realising the eventual type of $(a_n)$ over $A$ with respect to the inner product maps. Then $w = P_A b$ and $\mathcal{P}(p)$ is contained in $\{P_{\bdd(A)} b \mid b \models p, A \subseteq M\}$. For the converse inclusion, if $w = P_{\bdd(A)}b$, we can assume that $A$ is separable and we can take a Morley  sequence $(b_n)$ in $\tp(b/\bdd(A))$ with respect to the inner product maps, as defined in  Definition \ref{definition: Morley sequence}. Proposition \ref{proposition: transfer of independence} and Lemma \ref{lemma: convergence of morley sequences in H spaces} show that $ b_n \rightharpoonup w$ so that $w \in \mathcal{P}(p)$.
\end{proof}


  In \cite{spinoff}, it is shown that $\mathcal{P}(p)$ can also be constructed as a type-definable set in an imaginary sort of $T$ coding canonical bases for $\langle x, y\rangle$-types consistent with $p$.

%

\begin{definition}\label{definition: partial ordering}
Let $M \models T$ be $\omega_1$-saturated and let $p$ be a type-definable set in $H(M)$. We define the partial order $\leq$ on $\mathcal{P}(p)$ as follows: we say that $v \leq w$ in $\mathcal{P}(p)$ is there is a finite sequence of $\bdd$-closed subsets $A_1, \ldots, A_n$ of $M$ such that $v = P_{A_n} \ldots P_{A_1} w$.
\end{definition}

The partial order on $\mathcal{P}(p)$ is especially interesting under the assumption of scatteredness, which is one of the main notions of this paper.    

\begin{definition}\label{definition: scattered}
  Let $M \models T$ be $\omega$-saturated and let $p$ be a type-definable set in $H(M)$. We say that $p$ is scattered   if $\mathcal{P}(p)  $ is locally compact in the norm topology.
\end{definition}

  While scatterdness is the most general assumption we will work with, the following stronger condition is of special interest in many model-theoretic situations:

\begin{definition}\label{definition: strictly definable}

Let $p, q$ be type-definable sets in $\mathcal{H}$. We say that the inner product map   on $p \times q$ is strictly definable if it takes only finitely many values on $p \times q$.
\end{definition}

Strict definability of the inner product map is often easier to verify than scatteredness. For example, if the type-definable sets $p, q$ are pieces of $\mathcal{H}$, then strict definability of the inner product on $p \times q$ is not model dependent.

\begin{lemma}\label{lemma: strictly definable implies scattered}
Let $p$ be a partial type in $\mathcal{H}$.  If the inner product map $f$ is strictly definable on $p\times p$ then in any $M \models T$, $\mathcal{P}(p)$ is a discrete metric space in $H(M)$  and hence $p$ is scattered.
\end{lemma}

\begin{proof}
  Let $v, w \in \mathcal{P}(p)$. We showed in Lemma \ref{lemma: base set for partial order} that $\mathcal{P}(p)$ is the set of weak limit points of $p$ so there are   sequences $(a_n)$ and $(b_n)$ in $p$ such that $a_n \rightharpoonup v$ and $ b_n \rightharpoonup w$. Then $\langle v, w \rangle = \lim_n \lim_m \langle a_n, b_m \rangle$. By stability of the inner product and strict definability,  $ \langle a_n,  a_m \rangle$ must be eventually constant. Therefore $\langle v, w\rangle$ is one of the finitely many values   achieved by the inner product on $p \times p$.  It follows that $\mathcal{P}(p)$ is a discrete set and hence it is locally compact. 
\end{proof}

\noindent \textbf{Remarks:} (1) Saying that $p$ is scattered  is   stronger than saying that the set of realisations of $p$ is locally compact in $H(M)$, for $M \models T$ $\omega$-saturated. Consider the following example. Let $T$ be a two-sorted structure $(S_1, S_2)$ where the sort $S_1$ is an infinite set with the discrete metric and $S_2$ is the surface of the unit ball in an infinite dimensional Hilbert space. We add to $T$ the inner product map on $S_2$ and a function $g : S_1 \to S_2$. $T$ says that $f$ has dense image in $S_2$ and that every fibre of $f$ is infinite. Define the positive-definite map $f(x, y)$ on $S_1 \times S_1$ by   $f(x, x) = 2$ and $f(x, y) = \langle g(x), g(y)\rangle$ with the inner product is defined on the sort $S_2$. Let $\mathcal{H}$ be the interpretable Hilbert space generated by $S_1$  with inner product map $f$.

For any $M \models T$, $H(M)$ can be viewed as the sum of $H_2(M)$, a Hilbert space generated by the set $S_2$, and $H_1(M)$, a Hilbert space with orthonormal basis indexed by $S_1$. The set $S_1$ in $H(M)$ is the sum of this orthonormal set with certain vectors on the sphere $S_2$. Therefore $S_1$ is locally compact in $H(M)$ but $\mathcal{P}(S_1)$ contained $S_2$, which is not locally compact. Hence $S_1$ is not scattered.

(2) $\omega$-saturation is needed in the definition of scatteredness. Consider the following example: let $T$ be a theory of classical logic with equivalence relations $(E_n)_{n \geq 0}$ and constant symbols $(c_n)$ as follows. The constants $c_n$ are each in a different equivalence $E_0$-class and each $E_0$-class is infinite. We write $[c_n]_0$ for the $E_0$-class of $c_n$. The equivalence relations $(E_n)$ satisfy the following properties:

 \begin{enumerate}
\item for every $n \geq 0$ and for every $m \leq n$, $E_{n+1}$ equals $E_{n}$ on $[c_m]_0$. 
\item for every $n \geq 0$ and outside of the set $\bigcup_{m \leq n} [c_m]_0$, $E_{n+1}$ refines each $E_n$-class into infinitely many infinite equivalence classes.

\end{enumerate}
$T$ is a complete theory. We see that each definable set $[c_n]_0$ has Morley rank $n+1$. 
We define an inner product map $f$ on the main sort of $T$ as follows: 
\begin{enumerate}
\item for every $x , y$, $f(x, y) \leq 2$ 
\item for every $n \geq 0$, if $x,y $ are $E_m$-equivalent for every $m \leq n$, then $f(x, y) \geq \sum_{k = 0}^n 2^{-n}$
\item for every $n \geq 0$, if $x, y$ are $E_m$-equivalent for every $m \leq n$ and $x, y$ are not $E_{n+1}$-equivalent, then $f(x, y) = \sum_{k = 0}^n 2^{-k}$
\end{enumerate}
$f$ is clearly definable. To see that $f$ is positive-semidefinite, let $M \models T$ be saturated and such that for every $n$ we can enumerate the $E_n$-classes of $M$ by some cardinal $\kappa$. Let $H$ be a Hilbert space generated by orthonormal sequences $(v_n^\alpha)_{n \geq 0, \alpha < \kappa}$. For every $x \in M$, let $\alpha_n(x)$ be the ordinal such that $x$ is in the $\alpha_n(x)$-th $E_n$-class. For example, $\alpha_n(x)$ is constant on $[c_0]_0$ for all $n$. Consider the map $h : M \to H$ defined by $h(x) = \sum_{n = 0}^\infty v_n^{\alpha_n(x)}/2^n$. Then $\langle h(x), h(y)\rangle = f(x, y)$, so that $f$ defines an interpretable Hilbert space $\mathcal{H}$. 

Let $p$ be the type of an element which is not $E_0$-equivalent to any $c_n$. By considering weakly convergent sequences in $M$, we see that $\mathcal{P}(p) = \{ \sum_{k = 0}^\nu v_k^{\alpha_k(x)}/2^k \mid x \models p, \nu \in \Nn \cup \{\infty\} \} \cup \{0\}$, which is not locally compact. Hence the sort $S $ is not scattered in  $\mathcal{H}$.

Let $N \models T$ be the prime model, so that $N$ is the union of the $E_0$-classes of the elements $c_n$ and the type $p$ above is omitted. As above, $H(N)$ is generated by orthonormal sets $(v_n^\alpha)_{n \geq 0, \alpha < \omega}$ where $N$ is identified with the set $\{\sum_{k = 0}^\infty v_n^{\alpha(x)} \mid x \in N\}$. Then $\mathcal{P}(N) = N \cup \{\sum_{k = 0}^m v_k^{\alpha_k(x)}/2^k \mid m \leq n \in \Nn,   x \in [c_n]_0\}$. This set is locally compact.
\medskip

\begin{lemma}\label{lemma: sum of scattered sets is scattered}
Let $p$, $q$ be type-definable scattered sets in $\mathcal{H}$ such that in an $\omega$-saturated $M \models T$, for all $x \models p$ and $y \models q$, we have $\langle x, y\rangle = 0$. Let $r = \{x+y \mid x \models p, y \models q\}$. Then $r$ is scattered.
\end{lemma}

\begin{proof}
Fix $M \models T$ $\omega$-saturated. Since $\mathcal{P}(r)$ is the set of weak limit points of $r$ in $M$, $\mathcal{P}(r) = \{v + w \mid v \in \mathcal{P}(p), w \in \mathcal{P}(q)\}$. Fix $v + w \in \mathcal{P}(r)$ and $\e > 0$ such that the closed balls $B(v, \e)$ and $B(w, \e)$ are compact in $\mathcal{P}(p)$ and $\mathcal{P}(q)$ respectively. Then $B(v + w, \e) \cap \mathcal{P}(r)$ is contained in the sum of $B(v, \e)\cap \mathcal{P}(p)$ and $B(w, \e) \cap \mathcal{P}(q)$. It follows easily that $B(v+ w, \e)$ is compact in $\mathcal{P}(r)$. 
\end{proof}

\begin{definition}\label{definition: asymptotically free}
For $p$  a type-definable set in $\mathcal{H}$, we say that $p$ is asymptotically free  if for any $M \models T$, $x \models p$ in $M$ and $\e \geq 0$, the set   $\{y \models p\mid |\langle x, y \rangle| \geq \e\}$ is compact.

Equivalently, when $M \models T$ is $\omega$-saturated, for any $x, y \models p$, either $\langle x, y \rangle = 0$ or $x \in \bdd(y)$.
\end{definition}

The following lemma is trivial but important:
\begin{lemma}
If $p$ is an asymptotically free type-definable set in $\mathcal{H}$, then $p$ is scattered.
\end{lemma}

We will see in Theorem \ref{theorem: summary of general scattered  interpretable hilbert space} that asymptotically free sets arise naturally in scattered interpretable Hilbert spaces. 
While the notion of asymptotic freedom is natural and will be shown to have interesting representation theoretic consequences in Section \ref{section: unitary representations}, the notion of scatteredness may appear less natural at first sight. However, by Lemma \ref{lemma: sum of scattered sets is scattered}, a finite sum of asymptotically free sets is scattered, and there is no other obvious candidate notion to characterise a finite sum of asymptotically free sets. Moreover, we have not found any general conditions weaker than scatteredness under which Theorem \ref{theorem: summary of general scattered  interpretable hilbert space} holds, or any weaker version of it. Therefore, scatteredness currently appears to be the correct notion for a general decomposition theorem in the style of Theorem \ref{theorem: summary of general scattered  interpretable hilbert space}.

\subsection{Decomposition into $\bigwedge$-interpretable subspaces }\label{subsection: local analysis}
\emph{Until the end of Section \ref{subsection: local analysis}, we make the following assumptions and notational conventions. We fix an $\omega_1$-saturated $M \models T$ and we fix a type-definable set $p$  in $\mathcal{H}$. We assume that $p$ is scattered.  Recall that we  write $H_p(M)$ for the $\bigwedge$-interpretable subspace of $H(M)$ generated by the set $p$ (see Definition \ref{definition: generalised piecewise interpretable}). }

\medskip

 The next theorem is the basic fact which shows that it is interesting to look at $\mathcal{P}(p)$ as a partial order. It also shows that  types over $\bdd$-closed subsets of $M$ contained  in $p$ are one-based in a restricted sense (see \cite{spinoff} for a discussion).

\begin{theorem}\label{projections commute}
Let $A$, $B$ be small subsets of $M$ such that $A = \bdd(A)$ and $B = \bdd(B)$. Then $A \cap H_p(M)$ and $B \cap H_p(M)$ are orthogonal over $A \cap B \cap H_p(M)$. Equivalently,  for any $v \in A \cap H_p(M)$, we have $P_{B}v = P_{A \cap B}v$. Equivalently, for any $v \in H_p(M)$,
\[
 P_{B} P_{A}v = P_{A}P_{B}v = P_{A \cap B}v\]
\end{theorem}

\begin{proof}
 It is enough to check the statement for  arbitrary $v  \models p$.  Define $x_0 = P_A v$, $y_n = P_B x_n$ and $x_{n+1} = P_A y_n$. It is well-known that  the sequences $(x_n)$ and $(y_n)$ converge to $w = P_{A \cap B}v \in \mathcal{P}(p)$ in the norm topology. See  Theorem 13.7 in \cite{VonNeumann1950} for more details. 
 
 Suppose for a contradiction that  $x_n$ and $y_n$ are distinct from $w$ for all $n$. Then $y_n \notin A$ and $x_n \notin B$. By Lemma \ref{lemma: base set for partial order}, for every $n$ we can find infinite sequences $(x_n^k)_k$ and $(y_n^k)_k$ such that $(x_n^k)_k$ is a sequence in $\tp(x_n / B)$ converging weakly to $y_n$ and similarly for $(y_n^k)_k$. Then for any $\e > 0$ there is $n \geq 0$ such that the sequence $(x_n^k)_k$ is within distance $\e$ of $w$. Since we are assuming that $\mathcal{P}(p)$ is locally compact, this is a contradiction and $(x_n), (y_n)$ are eventually constant equal to $w$.

Take $n \geq 1$ such that $y_n \in A \cap B$. We now show that $x_n \in A \cap B$.  Write $x_n = y_{n} + \alpha$ where $\alpha \perp B  $ and $y_{n-1} = x_n + \beta$ where $\beta \perp A $. We have 
\begin{IEEEeqnarray*}{rClL}
\langle \alpha, y_{n-1}\rangle & = & 0\\
& = & \langle \alpha, y_{n} + \alpha + \beta\rangle \\
& = & \langle \alpha, \alpha \rangle + \langle \alpha, y_{n}\rangle +  \langle \alpha, \beta\rangle\\
& = & \langle \alpha, \alpha \rangle + \langle \alpha, \beta\rangle \\
& = & \langle \alpha, \alpha\rangle + \langle x_n - y_{n}, \beta\rangle\\
& = & \langle \alpha, \alpha \rangle \text{ since $x_n - y_n \in A$}.
\end{IEEEeqnarray*}
 So $\alpha = 0$,  $x_n = y_n$ and $x_n \in A \cap B$. We use a similar calculation to show that $y_{n-1}\in A \cap B$ when $x_n \in A\cap B$, with $n\geq 1$. This proves by induction that $y_0 \in A$. Hence $P_{B} P_{A}v = P_{A \cap B}v$ and the theorem is proved. 
\end{proof}


\begin{lemma}\label{lemma: partial order is definable}
  For every $v \in \mathcal{P}(p)$, $\{w \in \mathcal{P}(p) \mid w \leq v\}$ is  uniformly  type-definable  over $v$. Therefore, this set is compact in the norm topology. 
\end{lemma}

\begin{proof}
 Recall from Lemma \ref{lemma: weak closure is definable} that $\mathcal{P}(p)$ is a type-definable set in $H(M)$. Take $v \in \mathcal{P}(p)$. By Theorem \ref{projections commute}, $\{w \in \mathcal{P}(p) \mid w \leq v\} = \{P_{\bdd(A)} v \mid A \subseteq M\}$. As in Lemma \ref{lemma: base set for partial order}, $\{P_{\bdd(A)} v \mid A \subseteq M\}$ is the set of weak limit points of Hilbert space indiscernible sequences in $\tp(v)$ which begin at $v$. It is straightforward to check that this is  a uniformly type-definable set over $v$.


By Theorem \ref{projections commute}, $\{w \in \mathcal{P}(p) \mid w \leq v\}$ is contained in $\bdd(v)$ in the language $\mathcal{L}^{\mathcal{H}}$. By type-definability, this set is compact in the norm topology.
\end{proof}

\begin{lemma}\label{lemma: partial order is well-founded}
$\mathcal{P}(p)$ is well-founded with respect to the partial order from Definition \ref{definition: partial order P}.
\end{lemma}

\begin{proof}
Suppose $(v_n)$ is an infinite   decreasing sequence in $\mathcal{P}(p)$. By Theorem \ref{projections commute}, we can write $v_n = P_{V_n}v_0$ where $V_n$ is a  $\bdd$-closed subspace of $H(M)$ and $V_{n+1} \subseteq V_n$. Since $\{w \in \mathcal{P}(p) \mid w < v_0\}$ is metrically compact, the sequence $(v_n)$ is convergent and it follows that it converges to $z :=P_{V} v_0$ with $V = \bigcap_n V_n$. Then $z  < v_n$ for all $n$ and $v_n \notin \bdd(z)$. For every $\e > 0$ we can find $n$ such that $\|v_n - z\| < \e$ and we can take an infinite indiscernible sequence in $\tp(v_n/z)$ which must lie in $\mathcal{P}(p)$, by  Lemma \ref{lemma: weak closure is definable}. Hence $\mathcal{P}(p)$ is not locally compact around $z$ and this contradicts scatteredness of $p$. Therefore any decreasing sequence in $\mathcal{P}(p)$ is eventually constant.
\end{proof}


We use Lemma \ref{lemma: partial order is well-founded} to decompose $H_p(M)$. Fix an enumeration $(p_\alpha)_{\alpha < \kappa}$ of the complete types in $\mathcal{P}(p)$ with the property that for any $a,b \in \mathcal{P}(p)$, if $b < a$ in $\mathcal{P}(p)$ then $\tp(b)$ comes before $\tp(a)$ in the sequence $(p_\alpha)$. Such an enumeration exists by Lemmas  \ref{lemma: weak closure is definable} and \ref{lemma: partial order is well-founded}. For any $\alpha  < \kappa$, let $V_\alpha$ be the subspace of $H(M)$ generated by the realisations of $\bigcup_{\beta < \alpha} p_\beta$ (set $V_0 = \{0\})$.

For every $\alpha$, find a complete type $q_\alpha$ in  $\mathcal{H}$  such that for some (any) $v \models p_\alpha$, there is $w \models q_\alpha$ such that $w = P_{V_\alpha^\perp} v$. 

 \begin{lemma}\label{lemma: orthogonal maps are definable}
For every $\alpha$, the relation $   P_{V_\alpha^\perp}  v = w$ is type-definable on $p_\alpha \times q_\alpha$.
\end{lemma}
 
 \begin{proof}
   Fix $v \models p_\alpha$ and write $d \geq 0$ for the distance between $x$ and $V_\alpha$. Since $p_\alpha$ is a complete type and $V_\alpha$ is generated by a union of $\bigwedge$-interpretable Hilbert spaces, $d$ does not depend on $x$. Moreover, $P_{V_\alpha} v$ is the unique element $z \in V_\alpha$ such that $\| v - z\| = d$. Note also that $P_{V_\alpha^\perp}  v = v - P_{V_\alpha} v$.

For every $\e > 0$, there is  $n_\e \geq 0$ such that  for every $ i\leq n_\e$ we can  find a type $p_{\alpha_i}$, $v_i \in  q_{\alpha_i}$ and $\lambda_i \in [-n_\e, n_\e] $   satisfying $\|\sum_{i \leq n_\e} \lambda_i  v_i -  v\| \leq d + \e$.   $n_\e$, the types $p_{\alpha_i}$ and the scalars $\lambda_i$ do not depend on $v$. Write $\phi_\e(v, w)$ for the type-definable set
\[
\exists v_1, \ldots v_{n_\e} \Big(\bigwedge_{i \leq n_\e}  p_{\alpha_i}(v_i)  
\wedge \| \sum \lambda_i   v_i -  v\| \leq d + \e 
\wedge \| v - \sum \lambda_i  v_i  -  w\| \leq \e \Big)
\]
The relation  $ P_{V_\alpha^\perp}  v  = w$  is defined  on $p_\alpha \times q_\alpha$ by the intersection of all $\phi_\e(v, w)$. 
 \end{proof}


It is easy to prove by induction that for every $\alpha \leq \kappa$ the Hilbert space generated by $\bigcup_{\beta < \alpha} q_\beta$ is equal to $V_\alpha$. Therefore, $H_p(M)$ is the orthogonal sum of the spaces generated by each $q_\alpha$.  

\begin{lemma}\label{lemma: orthogonal maps give asymptotically free}
For every $\alpha < \kappa$, $ q_\alpha$ is asymptotically free.
\end{lemma}

\begin{proof}
Suppose $v, w \in q_\alpha$ and $v \notin \bdd(w)$. Let $u = P_{\bdd(w)}v$ so that $u < v$ and  $u \in V_\alpha$. Then $\langle v, w\rangle = \langle u, w\rangle = 0$.
\end{proof}

\begin{theorem}\label{theorem: summary of general scattered  interpretable hilbert space}
 
Let $\mathcal{H}$ be an interpretable Hilbert space in $T$. Let $\mathcal{H}_p$ be a $\bigwedge$-interpretable subspace of $\mathcal{H}$ generated by a scattered type-definable set $p$. Then $\mathcal{H}_p$ is the orthogonal sum of $\bigwedge$-interpretable Hilbert spaces $(\mathcal{H}_\alpha)_{\alpha < \kappa}$ such that for all $\alpha < \kappa$, $\mathcal{H}_\alpha$ is generated by an asymptotically free complete  type.

The decomposition of $\mathcal{H}_p$ is not model dependent, in the sense that for any $M \models T$, $H_p(M)$ is either empty or is the orthogonal sum of the spaces $H_\alpha(M)$. 
\end{theorem}

\begin{proof}
For $M \models T$ $\omega_1$-saturated, we constructed the asymptotically free types $(q_\alpha)_{\alpha < \kappa}$ such that $H_p(M)$ is the orthogonal sum of the spaces $H_\alpha(M)$ generated by    $q_\alpha$. We only need to show that this decomposition applies when the model is not $\omega_1$-saturated.

Let $N \models T$ be an arbitrary model where $p$ is realised. We can assume $N \prec M$ and we fix $v \models p$ in $N$.

Working in $M$, let $w$ be an element realising a type $q_{\alpha}$ such that $\langle v,  w\rangle \neq 0$.   If $w \notin \bdd(v)$, then we find an infinite indiscernible sequence $(w_n)$ in $\tp(w/\bdd(v))$. Since $q_\alpha$ is asymptotically free, $( w_n)$ is an infinite orthogonal sequence in $H(M)$ with $\langle v,   w_n\rangle \neq 0$, and this is a contradiction. Therefore $w \in \bdd(v)$ and hence $w \in N$. 

Working again in $M$, we can find a $\bdd$-independent family $(w_n)$ realising types $q_{\alpha_n}$ such that $\langle v,   w_n\rangle \neq 0$ and $v$ is in the closed span of $\bigcup  (\bdd(w_n) \cap q_{\alpha_n})$. Since each $w_n$ is in $N$, each set $\bdd(w_n) \cap q_{\alpha_n}$ is contained in $N$ and the spaces $(H_\alpha(N))_{\alpha < \kappa}$ generate  $H_p(N)$.
\end{proof}

 The following corollary shows that Theorem \ref{theorem: summary of general scattered  interpretable hilbert space} is not restricted to the $\bigwedge$-interpretable subspace $\mathcal{H}_p$ generated by a single type-definable set.

\begin{corollary}\label{corollary: orthogonal structure for general scattered Hilbert space}
  Suppose that $\mathcal{H}$ contains  $\bigwedge$-interpretable subspaces $(\mathcal{H}_i)_I$ such that each $\mathcal{H}_i$ is generated by a scattered type-definable set. Then the subspace of $\mathcal{H}$ generated by all $\mathcal{H}_i$  can be expressed as  the orthogonal sum of $\bigwedge$-interpretable subspaces $(\mathcal{H}_j)_{j \in J}$ such that for all $j \in J$, $\mathcal{H}_j$ is generated by an asymptotically free complete   type.
  
  This decomposition is not model dependent in the same sense as in Theorem \ref{theorem: summary of general scattered  interpretable hilbert space}.
 \end{corollary}
 
\begin{proof}
Take  $M \models T$ $\omega_1$-saturated as before.  Let $q$ be an  asymptotically free complete type of $\mathcal{H}$.  Let $V$ be a subspace of $H(M)$ generated by an arbitrary collection of $\bigwedge$-definable sets. As in Lemma \ref{lemma: orthogonal maps are definable}, the projection $P_{V^\perp}  $ is definable on $q$ and we can find a type $q'$ in $\mathcal{H}$ which is the image of $q$ under $P_{V^\perp}$. We check that $q'$ is asymptotically free.

Take $v, w \models q'$ with $v \notin \bdd(w)$ and suppose $v = P_{V^{\perp}}z$ where $z \models q$. Then $\langle v, w \rangle = \langle z, w\rangle$ and $z \notin \bdd(w)$. Take $(z_n)$ an infinite indiscernible sequence in $\tp(z/w)$. Then $(z_n)$ is an infinite orthogonal sequence in $H(M)$ and hence $\langle z,  w\rangle = 0$.

We apply Theorem \ref{theorem: summary of general scattered  interpretable hilbert space} to each $\mathcal{H}_i$ and obtain asymptotically free types $(q_\alpha^i)$ generating each $\mathcal{H}_i$. We can then ensure that these types are pairwise orthogonal by enumerating them as $(r_\alpha)$ and mapping each $r_\alpha$ to the orthogonal complement of the subspace generated by $\bigcup_{\beta < \alpha} r_\beta$.

  The same proof as in Theorem \ref{theorem: summary of general scattered  interpretable hilbert space} shows that this decomposition is not model dependent.
%
%
\end{proof}

\subsection{Strictly interpretable Hilbert spaces}\label{subsection: strictly interpretable} 

\emph{In Section \ref{subsection: strictly interpretable}, we fix an interpretable Hilbert space $\mathcal{H}$ in $T$. In the following results, there are no unstated assumptions on $T$ or $\mathcal{H}$.}

 \medskip
 
In this section, we investigate interpretable Hilbert spaces generated by strictly definable inner product maps. The next proposition shows that we can find an asymptotically free generating set  for such an interpretable Hilbert space  on which the inner product is still strictly definable.


\begin{proposition}\label{proposition: asymptotically free generating set for general strictly interpretable hilbert space}
Let  $p$ be a type-definable set in  $\mathcal{H}$. 
 If  the inner product map is strictly definable on $p \times p$, then there is an asymptotically free type-definable set $q$ generating $\mathcal{H}_p$ such that the inner product map   is strictly definable on $q \times q$.
 
Moreover, if $p$ is a finite union of complete types, then $q$ is a finite union of complete types. 
\end{proposition}

\begin{proof}

 Take   $M \models T$ $\omega_1$-saturated. We have already seen that $\mathcal{P}(p)$ is a discrete type-definable set. It follows from   Lemma  \ref{lemma: partial order is definable} that for any $v \in \mathcal{P}(p)$, the set $\{w \in \mathcal{P}(p) \mid w \leq v\}$ is finite and uniformly definable over $v$. Additionally if  $p$ is a finite union of complete types, then  $\mathcal{P}(p)$ is a finite union of complete types.

  For every $v \in \mathcal{P}(p)$   write  $
\pi(v)$ for the finite set $\{w \in \mathcal{P}(p) \mid w  < v \}$.  Write $V(v)$ for the subspace of $H(M)$ spanned by $\pi(v)$. Let $q$ be the type-definable set in $\mathcal{H}$ of elements of the form $P_{V(v)^\perp} v$ for $v \in \mathcal{P}(p)$.   Since $P_{V(v)^\perp}v$ is a linear combination of $ \{v\} \cup \pi(v)$ and the coefficients in this linear combination only depend on the values of the inner product between elements in this set,   the inner product on $q \times q$ is strictly definable. We only have to show that $q$ is asymptotically free.

Fix $x, y \in q$ and find $v \in \mathcal{P}(p) $ such that $ x = P_{V(v)^\perp} v$. To simplify notation,   write $V = V(v)$ and $Y = \bdd(y)$.
 We will prove that if  $x \notin Y$ then $P_Y x = 0$.   Since  $P_{Y} x  = P_{Y}v - P_Y P_V  v$, it is enough to prove that $P_Y P_V v = P_Y v$. 
 We show that $\|P_YP_Vv - P_Y v \|^2 = 0$. Expanding the left hand side gives $\langle P_Y P_V v , P_Y P_V v\rangle + \langle P_Y v, P_Y v\rangle -2\langle P_Y v, P_Y P_V v\rangle$. 
 Now we have:
\[
\langle P_Y v, P_Y P_V v\rangle = \langle  P_V P_Y v, v\rangle = \langle P_Y v, v\rangle =  \langle P_Y v, P_Yv\rangle  
\] 
and similarly
\[
\langle P_Y P_Vv , P_Y P_V v\rangle  = \langle P_V P_Y P_V v, v\rangle  = \langle P_Y P_V v, v\rangle = \langle  v, P_V P_Y v\rangle = \langle  v, P_Yv\rangle =  \langle P_Y v, P_Y v\rangle 
\]
The proposition follows for $M$. To see that $q$ generates $\mathcal{H}_p$ independently of a choice of model, we observe that the realisations of $q$ are contained in $\bdd(p)$.
\end{proof}


\medskip

We now focus on the case where  $T$ is a  classical discrete  logic theory  and   $\mathcal{H}$ is generated by classical sorts of  $T$ with strictly definable inner product maps. 
We investigate to what extent the decomposition of Theorem \ref{theorem: summary of general scattered  interpretable hilbert space} can be recovered inside the classical logic sorts of $T$.

\begin{definition}\label{definition: strictly interpretable Hilbert space}
Let $T$ be a classical logic theory. If $\mathcal{H}$ is an interpretable Hilbert space in $T$ generated by classical imaginary sorts of $T$ with strictly definable inner product maps, we say that $\mathcal{H}$ is a strictly interpretable Hilbert space in $T$. 
\end{definition}

In Proposition \ref{proposition: asymptotically free generating set for general strictly interpretable hilbert space}, even if $T$ is a classical logic theory and $p$ is a definable set in a classical   sort of $T$, the asymptotically free set $q$ with  its strictly definable inner product map is not always a definable set. We show that this can be obtained with additional assumptions on $T$.

\begin{definition}
In a classical logic theory $T$,  a formula $\phi(x,y)$ (possibly with parameters) has the finite cover property  (the FCP) if for all $n\geq 1$ there are $a_1, \ldots, a_n$ such that $\bigwedge_{i\leq n} \phi(x,a_i)$ is inconsistent but for every $l\leq n$, $\bigwedge_{i\neq l} \phi(x,a_i)$ is consistent. If $\phi$ does not have the  FCP, we say $\phi$ has the NFCP.

We say that $T$ has the weak NFCP if all stable formulas of $T$ have the NFCP. 
\end{definition}

We will use the following easy lemma about NFCP formulas, which is a weak version of Theorem II.4.6 in \cite{Shelah1978a}:

\begin{lemma}\label{lemma: easy nfcp lemma}
Let $T$ be a classical logic theory and let $M\models T$ be $\omega$-saturated. Let $\phi(x, y)$ be a formula with the NFCP over $A \subseteq M $. There is $n\in \Nn$ such that, for all $n\leq \alpha < \omega$, any sequence $(a_i)_{i < \alpha}$  such that  $\models \phi(a_i, a_j)$  for all $i\neq j$ can be extended to a sequence $(a_i)_{i < \omega}$  with  the same property.
\end{lemma}

\begin{proof}
Take $n$ as given by the definition of NFCP for $\phi(x, y)$. Given $(a_i)_{i < \alpha}$, the partial type $\{\phi(x, a_i) \mid i < \alpha\}$ is $n$-consistent, so it is consistent. Take $a_{\alpha}$ a  realisation of this partial type.
\end{proof}

The next lemma can be seen as a strengthening of Lemma \ref{lemma: weak closure is definable} to the present context.

\begin{lemma}\label{lemma: extend interpretation to definable sets in NFCP}
Suppose that $T$ is a classical logic theory with the weak NFCP and $\mathcal{H}$ is strictly interpretable in $T$. Let $S$ be a piece of $\mathcal{H}$ which is a  classical sort of $T$  with a strictly definable inner product map. 


For any $M \models T$, $\mathcal{P}(S)$ is a definable set in a piece of $\mathcal{H}$ which is a classical sort of $T$ with strictly definable inner product map.


\end{lemma}

\begin{proof}
Throughout this proof, we write $\overline{x}$ for tuples of variables and $x$ for single variables. We write $R$ for the finite set of  values achieved by the inner product on $S \times S$. Let $M$ be an arbitrary model of $T$.

Recall from elementary stability theory that there is a number $N$ such that for any   sequence $(x_n)$ in $S$ and $y \in S$, there is a unique $\lambda$ in $R$  such that   $|\{i \in \Nn \mid  \langle x_i, y\rangle \neq \lambda\}| < N/2$. Let $S'$ be the imaginary sort  $S^N/E$ where $E$ is the equivalence relation  defined by 
\[
\forall z \in S, \  \Med_{i\leq N} \langle x_i, z \rangle =  \Med_{i \leq N} \langle y_i, z \rangle
\]
and where $\Med_{i\leq N} \langle x_i, z \rangle $ is the median of the set of values $\{\langle x_i, z \rangle \mid i \leq N\}$.  
 Let $S^+$ be the type-definable subset of $S'$ consisting of elements $z$ such that there is a sequence $(x_n)$ in $S$, possibly constant, such that $(x_n)$ is Hilbert space indiscernible and for all $k_0 < \ldots < k_N$, the $E$-class of $(x_{k_0}, \ldots, x_{k_N})$ equals $z$.  
\begin{claim}
$S^+$ is a definable subset of $S'$
\end{claim}

\begin{proof}[Proof of claim]
  For any $\lambda \in R$, write   $F_\lambda (\overline{x}, \overline{y})$ for the formula on $S^N \times S^N$  which says
   \begin{enumerate}
   \item$\overline{x}$, $\overline{y}$ and  $E$-equivalent 
   \item For all $i, j \leq N$, $\langle x_i, y_j\rangle = \lambda$
   \item For all $i\neq j \leq N$, $\langle x_i, x_j\rangle = \langle y_i, y_j\rangle = \lambda$
   \item For all $i,j \leq N$, $\langle x_i , x_i\rangle = \langle y_j, y_j\rangle$.
   \end{enumerate}

    $F_\lambda $ is an equivalence relation so $F_\lambda$ is stable . By the weak NFCP  there is a number $n_\lambda$ such that for all $k\geq n_\lambda$ and $\overline{y}_1, \ldots, \overline{y}_k$, if $\{ F_\lambda(\overline{x}, \overline{y}_i)\mid i\leq k\}$ is $n_\lambda$-consistent then it is consistent. Take $m > n_\lambda$ for all $\lambda \in R$.
    
Let $S_0(\overline{x})$ be the definable set in $S^N $
 \[
\exists \overline{y}_1, \ldots, \overline{y}_{m} \bigvee_{\lambda \in R} \Big(  \bigwedge_{i \leq m } F_\lambda (\overline{x}, \overline{y}_i) \wedge \bigwedge_{i \neq j} F_\lambda (\overline{y}_i, \overline{y}_j) \Big)
\]
and let $S_1$ be $S_0/E$. We check that $S_1$ is in fact equal to $S^+$. 
$S_1$ contains $S^+$  because any indiscernible sequence $(x_n)$ witnessing the definition of $S^+$ can be broken down into $m $ $N$-tuples which witness the definition of $S_1$. Conversely, suppose $\overline{a} \models S_0$,  take  $ \overline{b}_1, \ldots, \overline{b}_{m}$ as given by the definition of $S_0$ and fix $\lambda$ such  that these satisfy 
\[
\bigwedge_{i \leq n_0 } F_\lambda (\overline{a}, \overline{b}_i) \wedge \bigwedge_{i \neq j} F_\lambda (\overline{b}_i, \overline{b}_j). 
  \]
 By Lemma \ref{lemma: easy nfcp lemma}, we can construct an   sequence of tuples $(\overline{b}_n)$ which satisfy $  F_\lambda (\overline{b}_i, \overline{b}_j) \wedge F_\lambda (\overline{a}, \overline{b}_i)$ for $i\neq j$. Concatenate the tuples $\overline{b}_i$ to form a   sequence $(c_k)$ in $S$. By construction, $(c_k)$ is Hilbert space indiscernible. It follows easily that $(c_k)$ witnesses the definition of $S^+$.
\end{proof}

Take $z \in S^+$ and let $(x_n)$ be a sequence in $S$ as in the definition of $S^+$ for $z$. Let $v$ be the weak limit of $(x_n)$. Note that $v \in \mathcal{P}(S)$. Then for all $y \in S$, $\langle v, y \rangle = \lim \langle x_n, y \rangle = \Med_{i \leq N} \langle x_i, y \rangle$ and this last value only depends on $z$. Therefore $v$ only depends on $z$ and we can define the definable  map $h : S^+ \to \mathcal{P}(S)$ which maps $z$ to $v$. Note that $h$ is injective.

An application of the weak NFCP similar to the one in the claim shows that in any $M \models T$, $\mathcal{P}(S)$ is the set of weak limit points of $S$. It follows that $h : S^+ \to \mathcal{P}(S)$ is bijective.
Since $S^+$ is definable, we can extend $h$ to $S'$ by mapping the complement of $S^+$ to $0$. The lemma follows.
\end{proof}

  Combining Lemma \ref{lemma: extend interpretation to definable sets in NFCP} and  Proposition \ref{proposition: asymptotically free generating set for general strictly interpretable hilbert space}, we obtain the following corollary:

\begin{corollary}\label{corollary: structure theorem for NFCP theories}
Let $T$ be a classical logic theory with the weak NFCP. If $\mathcal{H}$ is a strictly interpretable Hilbert space in $T$, 
then $\mathcal{H}$ is generated  by asymptotically free pieces which are classical imaginary  sorts of $T$ with  strictly definable inner product maps.
\end{corollary}

Corollary \ref{corollary: structure theorem for NFCP theories} does not give the same kind of information as Theorem \ref{theorem: summary of general scattered  interpretable hilbert space} because it does not decompose the strictly interpretable $\mathcal{H}$ into orthogonal subspaces generated by complete  types. We do not know if Corollary \ref{corollary: structure theorem for NFCP theories} can be strengthened in this direction.

   It is certainly possible to improve Corollary \ref{corollary: structure theorem for NFCP theories}
 when $T$  is an $\omega$-categorical classical logic theory. 
 In this case, inner product maps on classical sorts of $T$ are always strictly definable, since type spaces are finite. 
 Lemma \ref{lemma: extend interpretation to definable sets in NFCP} applies and we find that the sets $\mathcal{P}(p)$  are  definable sets in classical sorts of $T$. Moreover, the decomposition procedure of Theorem \ref{theorem: summary of general scattered  interpretable hilbert space} produces definable sets in classical imaginary sorts of $T$, on which the inner product maps are necessarily strictly definable. Therefore we have the corollary:  

\begin{corollary}\label{corollary: orthogonal structure for omega-categorical theories}
Let $T$ be an $\omega$-categorical classical logic theory and let $\mathcal{H}$ be a  strictly interpretable Hilbert space in $T$. Then $\mathcal{H}$ is isomorphic to an orthogonal sum of    interpretable Hilbert spaces each generated by asymptotically free complete types in classical imaginary sorts of $T$.
\end{corollary}

 We will see in Section \ref{subsection: omega-categorical structures} that Corollary \ref{corollary: orthogonal structure for omega-categorical theories} is equivalent to the classification theorem of Tsankov for unitary representations of oligomorphic groups.


\subsection{Interpretable Hilbert spaces in atomic $\omega$-near-homogeneous structures}

\emph{We observed in Section \ref{subsection: key notions for structure theorems} that the definition of scatteredness relies on $\omega$-saturation in an essential way. This  restricts the scope of Theorem \ref{theorem: summary of general scattered  interpretable hilbert space} somewhat, especially in relation with Section \ref{section: unitary representations}. }

\emph{In this section, we show that the proof of Theorem \ref{theorem: summary of general scattered  interpretable hilbert space} can be adapted to deduce a structure theorem for interpretable Hilbert spaces in atomic $\omega$-near-homogeneous structures generated by a type-definable set whose weak closure is locally compact. This is Proposition \ref{proposition: decomposition for atomic omega-near-homogeneous}. This will be applied in Proposition \ref{proposition: modified proof of theorem for locally compact orbit}. }

\medskip

We start by recalling some classical definitions from \cite{BenYaacov2008}.

\begin{definition}[\cite{BenYaacov2008} 8.7, 9.16,  12.2, 12.7, 12.11]\label{definition: atomic and omega-homogeneous}
A complete type $p$ in a theory $T$ is said to be principal if $p$ is distance-definable (see Definition \ref{definition: distance-definable}). 

A model $M \models T$ is said to be atomic if every complete type realised in $M$ is principal.

We define a metric on   type-spaces of $T$ as follows: let $M$ be a model of $T$ realising all types. If $p, q$ are complete types of $M$, we define $d(p, q) = \inf\{d(a, b) \mid a \models p$, $b \models q$, $a, b \in N\}$.

We say that $M \models T$ is $\omega$-near-homogeneous if for any two finite tuples $a$ and $b$ in $M$,   for every $\e > 0$ there exists $g \in \Aut(M)$ such that $d(g(a), b) < d(\tp(a), \tp(b)) + \e$.
\end{definition}

\noindent \textbf{Remark:} In this section, we will work with atomic $\omega$-near-homogeneous structures. Interesting examples of these will be given in Section \ref{section: unitary representations}. These assumptions come with certain technicalities. Recall that we work in continuous logic for metric structures so that every sort of $T$ comes equiped with a metric.   $\omega$-near-homogeneity is defined with respect to these metrics. As a result, $\omega$-near-homogeneity for an arbitrary model $M$ is \emph{not} preserved under adding imaginary sorts to $T$. A similar remark applies to atomic models. Therefore, we make the following assumption on $\mathcal{H}$:
\medskip
\begin{changemargin}{0.7cm}{0.7cm}
\emph{We assume that the pieces of $\mathcal{H}$ are real sorts of $T$ and that the direct limit maps on each piece of $\mathcal{H}$ are isometries.}
\end{changemargin}
\medskip
 With this assumption, if $M \models T$ is atomic and $\omega$-near-homogeneous, then the characterisation of  atomicity and $\omega$-near-homogeneity applies to types in $\mathcal{H}$.  

\medskip

\begin{lemma}\label{lemma: base set for partial order in atomic case}
Let $M \models T$ be atomic. Let $p$ be a type-definable set in $\mathcal{H}$. Then $\mathcal{P}(p)$ is closed under the projections $P_A$ where $A \subseteq M$ is $\bdd$-closed. If $v, w\in \mathcal{P}(p)$ and  $w = P_Av$ with $A$ $\bdd$-closed, then there is an infinite sequence $(v_n)$ in $\tp(v)$ converging weakly to $w$. 
\end{lemma}

\begin{proof}
Write $S$ for the piece of $\mathcal{H}$ containing $p$. Suppose $w = P_A v$ where $v \models p$.   Write $q = \tp(v)$ and let $d(x)$ be the definable function which gives the distance to $q$. Let $N$ be an $\omega_1$-saturated elementary extension of $M$. Since $w = P_A v$ in $N$, Lemma \ref{lemma: base set for partial order} applies and $w$ is in the weak closure of the set of realisations of $q$ in $N$.  Hence $w$ is the weak limit of a sequence of realisations of $q$ in $N$.

 Suppose that  we have found $v_1, \ldots, v_n$ in $M$ satisfying $q$ such that for all $m < n$, $|\langle v_n, v_m\rangle - \langle w, v_m\rangle | \leq 1/n$ and $|\langle v_n, w\rangle - \langle w, w\rangle| \leq 1/n$. Fix $\e > 0$ small enough, to be determined below. By considering   $w$ inside the structure $N$, we find that $w$  satisfies the formula 
\begin{multline*}
\forall x_1, \ldots, x_n \in S,\exists y \in S , \Big( d(y) \leq \e \wedge |\langle y, w\rangle - \langle w, w\rangle| \leq 1/2n \\
\wedge \bigwedge_{m \leq n} |\langle y, x_m \rangle - \langle w,  x_m\rangle | \leq 1/2n \Big)
\end{multline*}
Find a realisation   $y$ of this formula over the tuple $v_1, \ldots, v_n$ previously constructed. Then there is $v_{n+1} \models q$ in $M$ with $\|v_{n+1} - y\| \leq \e$. Choose $\e > 0$ small enough so that $|\langle v_{n+1}, w\rangle - \langle w, w\rangle| \leq 1/(n+1)$ and  $|\langle  v_{n+1}, v_m \rangle - \langle w,  v_m\rangle | \leq 1/(n+1)$ for all $m \leq n$. 
In this way, we construct a sequence $(v_n)$ of realisations of  $q$ in $M$ which can be shown to converge weakly to $w$ as in Lemma \ref{lemma: weak closure is set of limit points}. 

Hence $w \in \mathcal{P}(p)$ and the second part of the lemma follows immediately.
\end{proof}

The proof of the following proposition follows the general structure of the proof of Theorem \ref{theorem: summary of general scattered  interpretable hilbert space}. We stress that the assumption  of Proposition \ref{proposition: decomposition for atomic omega-near-homogeneous} is weaker than scatteredness and that its conclusion is weaker than asymptotic freedom.

\begin{proposition}\label{proposition: decomposition for atomic omega-near-homogeneous}
Let $M \models T$ be atomic and $\omega$-near-homogeneous. Let $p$ be a type-definable set in $\mathcal{H}$ such that the weak closure $\mathcal{P}(p)$ in $H(M)$ is locally compact. Then $\mathcal{H}_p $ is the orthogonal sum of $\bigwedge$-interpretable subspaces $(\mathcal{H}_i)_{i \in I}$ such that for every $i \in I$, $\mathcal{H}_i$ is generated by a principal type $q_i$ and   for every $v, w \in M$ realising $q_i$, either $\langle v,w\rangle = 0$ or $v \in \bdd(w)$. 
\end{proposition}

\begin{proof}
  Let $\mathcal{Q}(p) \subseteq \mathcal{P}(p)$ be the smallest metrically closed subset of $H(M)$ containing $p$ and  which is closed under the projections $P_{A}$, where $A \subseteq M$ is $\bdd$-closed. Then $\mathcal{Q}(p)$ is locally compact.  
  Note that $\mathcal{Q}(p)$ contains $p$ but may be much smaller than $\mathcal{P}(p)$. We view $\mathcal{Q}(p)$ as a partial order with the order inherited from $\mathcal{P}(p)$.

\begin{claim}\label{claim: decomposition in atomic homogeneous case, 3}
$\mathcal{Q}(p)$ is a union of complete types in $M$.
\end{claim}

\begin{proof}
Observe that $\mathcal{Q}(p)$ is invariant under $\Aut(M)$. Suppose $v \in \mathcal{Q}(p)$ and suppose $w$ has the same type as $v$. By $\omega$-near-homogeneity, $w$ is the limit of a sequence of $\Aut(M)$-conjugates of $v$, and hence $w \in \mathcal{Q}(p)$. 
\end{proof}

\begin{claim}\label{claim: decomposition in atomic homogeneous case,1}
For any  $\bdd$-closed $A_1, A_2 \subseteq M$, $P_{A_1}P_{A_2} = P_{A_1 \cap A_2} = P_{A_2}P_{A_1}$ on $H_p(M)$.
\end{claim}

\begin{proof}
By Lemma \ref{lemma: base set for partial order in atomic case}, it is straightforward to adapt the application of Von Neumann's lemma in Theorem \ref{projections commute} to prove the claim.
\end{proof}

\begin{claim}\label{claim: decomposition in atomic homogeneous case, 2}
$\mathcal{Q}(p)$ is a well-founded partial order.
\end{claim}

\begin{proof}
This is a straightforward modification of the proof of Lemma \ref{lemma: partial order is well-founded}, using Claim \ref{claim: decomposition in atomic homogeneous case,1}.
\end{proof}

Applying Claim \ref{claim: decomposition in atomic homogeneous case, 3}, let  $(p_\alpha)$ be an enumeration of the complete types in $\mathcal{Q}(p)$. By well-foundedness, we can assume that the enumeration $(p_\alpha)$ respects the partial order on $\mathcal{Q}(X)$, in the same sense as in the proof of Theorem \ref{theorem: summary of general scattered  interpretable hilbert space}.
Let $V_\alpha$ be the closed subspace generated by $ \bigcup_{\beta < \alpha} p_\alpha$. For every $\alpha$, let $q_\alpha$ be the complete type obtained by applying $P_{V_\alpha^\perp}$ to the realisations of $p_\alpha$. $H_p(M)$ is the orthogonal sum of the subspaces generated by each $ q_\alpha$.

\begin{claim}\label{claim: decomposition in atomic homogeneous case, 4}
The relation $P_{V_\alpha^\perp}v  = w$ is a type-definable set on $p_\alpha \times q_\alpha$ and for every $y$ in $q_\alpha$ and $\e > 0$, there is $v$ in $p_\alpha$ such that $\|P_{V_\alpha^\perp}v  - w\| < \e$.
\end{claim}

\begin{proof}
Let $v \models p_\alpha$ in $M$ and let $d$ be the distance between $v$ and the subspace $V_\alpha$. By Lemma \ref{lemma: orthogonal maps are definable}, it is enough to check that  $d$ does not depend on $M$, i.e. if $M \prec N$ then the distance between $v$ and the subspace $V_\alpha(N)$   in $N$ is equal to $d$. 

If the distance between $v$ and $V_\alpha(N)$ is $< d$, there is a finite set $E$ of vectors realising the types $p_\beta$ in $N$ where $\beta < \alpha$ such that $v$ is at distance $<  d$ from the subspace generated by $E$. Since the types $p_\alpha$ are principal, we can find a set $E'$  of vectors in $M$ such that the Hilbert space type of the set $\{v\} \cup E'$ is arbitrarily close to the Hilbert space type of $\{v\} \cup E$. Hence the distance between $v$ and the subspace generated by $E'$ is $<d$, which is a contradiction.

A similar argument shows the second part of the claim.
\end{proof}


  We check that the types $q_\alpha$ satisfy the proposition. Fix $\alpha$ and $v, w \models q_\alpha$ in $M$ such that $v \notin \bdd(w)$ and suppose  that $\langle v,w \rangle  \neq 0.$ By Claim \ref{claim: decomposition in atomic homogeneous case, 4}, we can find a sequence  $(z_n)$ of realisations of $p_\alpha$ in $M$   such that $ \|v - P_{V_\alpha^\perp} z_n\| < 1/n$. By Claim \ref{claim: decomposition in atomic homogeneous case, 4} again, for large enough $n$ we have $z_n \notin \bdd(w)$.  We fix a sufficiently large $n$ such that $z_n \notin \bdd(w)$ and $\langle z_n, w\rangle \neq 0$. 
 
 Let $u = P_{\bdd(w)}z_n$ so that $u \in V_\alpha$ and hence $\langle u, w \rangle = 0$. However, $\langle u, w\rangle = \langle z_n, w\rangle \neq 0$, a contradiction.
\end{proof}

\subsection{Some elementary examples and counterexamples}\label{subsection: examples of decompositions}

The asymptotically free  type-definable sets  of Theorem \ref{theorem: summary of general scattered  interpretable hilbert space} or Corollary  \ref{corollary: structure theorem for NFCP theories} live in imaginary sorts of $T$ which may be difficult to identify in practice.   Nevertheless, in many special cases it is   possible to   give   a presentation of $\mathcal{H}$ which satisfies our various decomposition theorems without having to go through the proofs of these results.

\medskip

1. Let $T$ be the theory of an infinite set   $X$ and let $\mathcal{H}_1$ be the interpretable Hilbert space  generated by $X$ with inner product map $f(x, y) = 0$ for $x\neq y$ and $f(x, x) = 1$. Then $\mathcal{H}_1$ already satisfies the conclusion of Corollary \ref{corollary: orthogonal structure for omega-categorical theories}.

 Let $\mathcal{H}_n$ be the interpretable Hilbert space generated by $X^n$ with inner product map $f(\overline{x}, \overline{y}) = k$ if $\overline{x}, 
\overline{y}$ share $k$ entries, ignoring order. Then we can take the integer $N$ from Lemma \ref{lemma: extend interpretation to definable sets in NFCP} to be  equal to $2n+1$ so the factors   from Corollary \ref{corollary: orthogonal structure for omega-categorical theories} are given by a quotients of $X^{N}$. Going through the decomposition procedure of Theorem \ref{theorem: summary of general scattered  interpretable hilbert space} leads to the easy observation that  $\mathcal{H}_n$ is the orthogonal sum of $n$ copies of $\mathcal{H}_1$.  

Let $\mathcal{H}_1'$ be the interpretable Hilbert space generated by $X$ with inner product map $f(x, y) = 1$ if $x \neq y$ and $f(x, x) = 2$. Going through the decomposition procedure of Theorem \ref{theorem: summary of general scattered  interpretable hilbert space} again leads to the easy observation that $\mathcal{H}_1'$ is the orthogonal sum of a copy of $\mathcal{H}_1$ and a one-dimensional interpretable Hilbert space.

%
%
%
%

\medskip

2. For $n\geq 2$, let $T_n$ be the theory of the set of unordered $n$-tuples over an infinite set $X$. For $k \leq n$, $T_n$ has predicates $P_k(x,y)$  to say that $x$ and $y$ have exactly $k$ elements in common. Let $\mathcal{H}$ be the interpretable Hilbert space generated by the main sort $S$ with inner product map $f$ defined by $f(x,y) = k$ if and only if $P_k(x, y)$. Then  the structure of $\mathcal{H}$ is similar to the structure of  $\mathcal{H}_n$ from the previous example, but imaginary sorts of $T$ are needed to give the decomposition into orthogonal subspaces generated by asymptotically free definable sets.

\medskip
 
 3. Let $T = Th(\Zz, \leq)$. All three interpretable Hilbert spaces considered in Section \ref{subsection: examples of interpretable Hilbert spaces}, Example 2 already satisfy Theorem \ref{theorem: summary of general scattered  interpretable hilbert space}, as they are generated by an asymptotically free complete type.




\medskip

4. 
  We show the failure of Corollary \ref{corollary: structure theorem for NFCP theories} when some stable formula of $T$ has the FCP. 
  Let $T$ be the classical logic theory of an equivalence relation $E$ on a sort $S$ such that, for each $n$, $E$ has exactly one equivalence class of cardinality $n$. Define the  positive definite function $f$ by $f(x,x) = 2$ for all $x$, $f(x,y) = 1$ if $x \neq y$ and $xEy$ and $f(x,y) = 0$ if $\neg xEy$.   Write $\mathcal{H}$ for the   interpretable Hilbert space in $T$ generated by $S$ with inner product map $f$. Let $M\models T$ be $\omega_1$-saturated.

The set $\mathcal{P}(S)$ consists of the set $S$ together with $0$ and the weak limit point of each infinite $E$-class. In the notation of Lemma \ref{lemma: extend interpretation to definable sets in NFCP}, $N = 2$ and $S' = S^2$ quotiented out by the equivalence relation $E'$ on pairs $(x, y)$ with classes $x = y$, $x \neq y \wedge xEy$, and $\neg x E y$.  $S^+$ is the type-definable subset of $S'$ containing all $E'$-classes except those represented by pairs $(x, y)$ such that  $x \neq y \wedge x Ey$ and such that the $E$-class of $x$ is finite. $S^+$ is not definable.

Observe that there is no way of extending $h$ to a definable set $D$ containing $S^+$ in such a way that the inner product map on $D$ is strictly definable. Since there are no other natural candidates for generating sets, this suggests that the conclusion of Corollary \ref{corollary: structure theorem for NFCP theories} is false in this case.
Furthermore, we have the following classical result:

\begin{lemma}[\cite{Shelah1978a}, II 4.4]\label{FCP theorem}
Let $T$ be an arbitrary classical logic theory. If $T$ does not have the weak NFCP, then there is a definable equivalence relation $E((x, z), (y, z'))$ such that for all $n\geq 1$ there is a tuple $c_n$ such that the formula $E((x, c_n),(y, c_n))$ is an equivalence relation with more than $n$ but only finitely many equivalence classes.
\end{lemma}

Therefore, any theory without the weak NFCP has an interpretable Hilbert space with the same properties as $\mathcal{H}$ constructed above. 


\section{Interpretable $L^2$-spaces} \label{section: L^2 spaces}

\emph{In this section, we study two particular sources of interpretable Hilbert spaces: $L^2$-spaces associated to absolute Galois groups and $L^2$-spaces associated to definable measures. We show that the decomposition theorems of Section \ref{section on structure} have natural interpretations in these cases, and we find strong structural similarities between these two sources of interpretable Hilbert spaces.}

\subsection{Absolute Galois Groups}\label{subsection: galois groups}

\emph{In this section, we rely on the results \cite{ChevalierHru2022} which show that absolute Galois groups $Gal(K)$ of definably-closed substructures $K \subseteq M$ are interpretable in a very general setting. These give rise to $L^2$-spaces, which are interpretable. }

\emph{In Definition \ref{definition: inverse system of finite probability spaces} we introduce the general notion of a $2$-regular inverse system of finite probability spaces. These inverse systems subsume both the interpretation of $L^2(Gal(K))$ and of the subspace of class functions on $Gal(K)$.  See Section \ref{subsection: definable measures} for another source of examples. In Proposition \ref{proposition: structure of profinite measure spaces}, we show that $L^2$-spaces associated to $2$-regular inverse systems of finite probability spaces decompose   into subspaces generated by asymptotically free sets in natural way. This strengthens the results of Section \ref{section on structure}}.

\medskip

Let $T$ be a classical lgoic theory in a language $\mathcal{L}$. Recall from \cite{ChevalierHru2022} that we say that $T$ admits elimination of finite imaginaries if $T$ eliminates imaginaries for equivalence relations whose classes are all finite. We will assume that $T$ admits elimination of finite imaginaries. 

 If $\mathcal{L}$ has sorts $(S_j)_{j \in J}$, we write $\mathcal{L}_P$ for the expansion of $\mathcal{L}$ by a collection $P$ of predicates $ (P_j)_{j \in J}$  where $P_j \subseteq S_j$. Let $T_P$ be the theory extending $T$ which says that the union $\bigcup_{j \in J}P_j$ is definably closed.

For any $M \models T$ with a choice of subset $K \subseteq M$ such that $\dcl(K) = K$, we can view $M$ as a model of $T_P$ where $K = P(M)$. Write $Gal(K)$ for the group of elementary maps $\acl(K) \to \acl(K)$ which fix $K$ pointwise. Recall that $Gal(K)$ is a profinite group.  In the following,  we will also consider the spaces    $Gal(K)^k/Ad(Gal(K)$ for any $k \geq 1$. This is the space of orbits of  $Gal(K)^k$ under diagonal conjugation by $Gal(K)$.  

Recall the classical result of \cite{Cherlin1980} which says that if $K$ is a perfect field, then $Gal(K)$ is interpretable in $T_P$ where $T = ACF$. The following theorem generalises this to the general setting of $T$ and $T_P$ as explained above.

\begin{theorem}[\cite{ChevalierHru2022}]\label{theorem: galois paper}
Let $T$ be a classical logic theory which admits elimination of finite imaginaries and elimination of quantifiers. Let $T_P$ be the $\mathcal{L}_P$-expansion of $T$ as defined above. Let $M \models T$ and let $K \subseteq M$ be definably closed. We view $M$ as an $\mathcal{L}_P$-structure and $K$ as an $\mathcal{L}$-structure. 

Let $I$ be the partial order consisting of pairs $(S, n)$ where $S$ is a finite Cartesian product of sorts of $T$ and $n \geq 0$. We define $(S, n) \leq (S', n')$ if $S'$ extends $S$ and $n \leq n'$.

Then $Gal(K)$ is $\mathcal{L}_P$-interpretable  in $M$ as an $I$-graded profinite group. For any $k \geq 1$,  $Gal(K)^k/Ad(Gal(K))$ is  $\mathcal{L}$-interpretable in $K$ as an $I$-graded profinite  space. 
\end{theorem}

Let $T_P$ be as above and let $(S_i)_{i \in I}$ be the $I$-graded inverse system of $Gal(K)$ in imaginary sorts of $T_P$. Recall that $Gal(K)$ is a probability space with the Haar measure, and hence gives rise to an $L^2$-space.  Then $T_P$ has a strictly interpretable Hilbert space $\mathcal{H}$ generated by the sorts $(S_i)_I$ such that for every $(M, K) \models T_P$, $H(M, K)$ is canonically isomorphic to $L^2(Gal(K))$ in the following sense. For every $i \in I$, write $h_i  :S_i \to \mathcal{H}$ for the interpretation map. Then for any $(M, K) \models T_P$,    there is a surjective unitary  map $F : H(M) \to L^2(Gal(K))$ such that  for any $i \in I$ and $x \in S_i$, $F \circ h_i(x) = \I_{\pi^{-1}(x)}$ where $\pi$ is the quotient map of $Gal(K)$ to the finite group containing $x$. We say that $\mathcal{H}$ is \emph{the interpretation of $L^2(Gal(K))$} in $T_P$. 

With $(M, K) \models T_P$ as above and $k \geq 1$, let $(S_i')_{i \in I}$ be the $I$-graded inverse system of $Gal(K)^k/Ad(Gal(K))$ in imaginary sorts of $K$, viewed as an $\mathcal{L}$-structure. Note that $Gal(K)^k/Ad(Gal(K))$ is also a probability space with the measure inherited from the Haar measure.  In the same sense as above, $K$ \emph{interprets $L^2(Gal(K)^k/Ad(Gal(K))$}.

Note  that when $k = 1$, $L^2(Gal(K)^k/Ad(Gal(K))$   is   the Hilbert space of class functions on $Gal(K)$. By the construction of \cite{ChevalierHru2022},  $L^2(Gal(K)/Ad(Gal(K))$ can be viewed as a strictly interpretable subspace of $L^2(Gal(K))$ in the $\mathcal{L}_P$-structure $(M, K)$. 

\medskip

The two interpretable Hilbert spaces defined above are strictly interpretable in the sense of Definition \ref{definition: strictly interpretable Hilbert space}. We will show that these two interpretable Hilbert spaces decompose into interpretable Hilbert spaces generated by asymptotically free sets in a natural and explicit way. This will improve on Theorem \ref{theorem: summary of general scattered  interpretable hilbert space} and Proposition \ref{proposition: asymptotically free generating set for general strictly interpretable hilbert space}.

In order to make precise the underlying structure which we will use, we introduce the abstract notion of a $2$-regular $I$-graded inverse system of finite probability spaces. Examples include  the $I$-graded inverse systems of $Gal(K)$ and   of $Gal(K)^k/Ad(Gal(K))$ considered above. In Section \ref{subsection: definable measures} we will see that  another source of examples comes from $L^2$-spaces associated to $\omega$-categorical measurable theories.

\begin{definition}\label{definition: inverse system of finite probability spaces}
Let $(I, \leq)$ be a partial order such that for all $i, j \in I$ there is $k \in I$ with $i \leq k$ and $j \leq k$.

Let $\mathcal{L}_{pf, I}$ be the language with sorts $(S_i)_{i \in I}$ and for $i \leq j \in I$, a binary relation $\leq \subseteq S_i \times S_j$ and a  binary  relation $E \subseteq S_i \times S_i$.

An $I$-graded inverse system of finite   sets is an $\mathcal{L}_{pf, I}$-structure satisfying the following:
 \begin{enumerate}
\item for every $i \in I$, there is some $K_i \geq 1$ such that the relation $E$ is an equivalence relation on $S_i$ with equivalence classes of size at most $K_i$.
\item  the system of relations $\leq$ across the sorts  $(S_i)_{i \in I}$ define a partial order such that 
\begin{enumerate}
\item for $x \in S_i$ and $y \in S_j$, if $y \geq x$ then writing $a$ and $b$ for the $E$-classes of $x$ and $y$ respectively, $\geq \cap (b\times a)$ is the graph of a surjection $b \to a$.
\item if $i \leq j$, then for every $E$-class $a$ in $S_i$, there is a unique $E$-class $b$ in $S_j$ such that $\leq \cap( a \times b)$ is the graph of a bijection $a \to b$
\end{enumerate}
Hence $\leq$ factors through the equivalence relations $E$ and we will also write $\leq$ for the induced partial order on $E$-classes.
\item for any $i \leq j \in I$, if there are $E$-classes $a'$ in $S_i$ and $a$ in $S_j$ such that $\leq$ is a bijection $a' \to a$, then for any $E$-class $b$ in $S_j$ such that $b \leq a$, there is an $E$-class $b'$ in $S_i$ such that $\leq$ is a bijection $b' \to b$.


\item the directed system of $E$-classes in the sorts $S_i$ forms a lattice in the following sense:
\begin{enumerate}
\item for any $i \in I$ and $E$-classes $a, b \in S_i$, there is a unique class $c$ in $S_i$ such that $c = \max\{d \subseteq S_i \mid d \leq a, d\leq b\}$. We write $c = a\wedge b$. 
\item for any $i, j \in I$ there is $k \in I$ with $i \leq k$ and $j \leq k$ such that for any $E$-classes $a \subseteq S_i$ and $b \subseteq S_j$, there is $c \subseteq S_k$ such that $c = \min\{d\subseteq S_k \mid a \leq d, b \leq d\}$. We write $c = a \vee b$. 
\end{enumerate}
\end{enumerate}
Let $M$ be an $I$-graded inverse system of finite sets in a language extending $\mathcal{L}_{pf, I}$. We say that $M$ is an $I$-graded inverse system of finite probability spaces if for every $i \in I$ there is a strictly definable function $\mu_i : S_i \to [0, 1]$ satisfying the following:
\begin{enumerate}
\setcounter{enumi}{4}
\item for all $i \in I$ and for any $E$-class $a \in S_i$, $\sum_{x \in a} \mu_i(x) = 1$
\item for any $i \leq j \in I$, if $a$, $b$ are   $E$-classes in $S_i$ and $S_j$ respectively such that $a \leq b$, then for any $x \in a$, $\mu_i(x) = \sum_n \mu_j(y_n)$ where $y_1, \ldots, y_n$ is the pullback of $x$ to $b$
\end{enumerate}
We say additionally that $M$ is a $2$-regular $I$-graded inverse system if $M$ satisfies
\begin{enumerate}
\setcounter{enumi}{6} 
\item For every $i \in I$, take $k \in I$ such that for any $E$-classes $a, b$ in $S_i$, $a\vee b$ is  in $S_k$.   For any $E$-classes $a, b$ in $S_i$ and  $x \in a$,  $y \in b$, write $x \vee y = \{w \in a \vee b \mid w \geq x$ and $w \geq y\}$. Suppose $x \vee y \neq \emptyset$ and find $z \in a \wedge b$ such that $z \leq x$ and $z \leq y$. Then $\mu_k(x \vee y) = \mu_i(x) \mu_i(y)/\mu_i(z)$. \label{equation: modularity in profinite probability space} 
\end{enumerate}
\end{definition}
%

\noindent \textbf{Remarks:} (1) When $M$ is an $I$-graded inverse system of finite probability spaces, the sorts $(S_i)_{i \in I}$ form a  directed system of finite sets. We write $\lim_{\rightarrow}(S_i)_I$ for this direct limit. $\lim_{\rightarrow}(S_i)_I$ is an inverse system of finite probability spaces and we write $\lim_{\leftarrow} M$ for the inverse limit. An easy application of Caratheodory's theorem shows that $\lim_{\leftarrow} M$ is a Borel probability space.  
$M$ interprets $L^2(\lim_{\leftarrow} M)$ in the same sense as in the discussion following Theorem \ref{theorem: galois paper}. 

(2) $2$-regularity in Definition \ref{definition: inverse system of finite probability spaces} can be viewed as a local modularity property. See the remark following Proposition \ref{proposition: structure of profinite measure spaces} connecting this modularity property to Theorem \ref{projections commute}.

 \medskip

It follows from elementary Galois theory (and its generalisation due to \cite{Poizat1983}) that the $I$-graded inverse systems of $Gal(K)$  considered in Theorem \ref{theorem: galois paper} is $2$-regular. It   follows easily that the inverse system of   $Gal^k(K)/Ad(Gal(K))$ is also $2$-regular.

The choice of terminology `$2$-regularity' was chosen to distinguish this notion from the stronger notion (call it `$n$-regularity') which says that for any   $x_1, \ldots, x_n \in S_i$ such that $x_i \wedge x_j = x_{i'}\wedge x_{j'}$ for any $i \neq j$ and $i' \neq j'$,  $\mu(x_1 \vee \ldots \vee x_n) = (\prod_i \mu(x_i))\mu(\bigwedge_j x_j)$. It is easy to see that the $I$-graded inverse system of $Gal(K)$ is not usually  $3$-regular.

In the  next proposition,  we give a  decomposition of $L^2(\lim_{\leftarrow}M)$ into subspaces generated by asymptotically free definable sets. Since $L^2(\lim_{\leftarrow}M)$ is strictly interpretable in $M$,  Proposition \ref{proposition: structure of profinite measure spaces} entails Proposition \ref{proposition: asymptotically free generating set for general strictly interpretable hilbert space} in this particular setting. 


 \begin{proposition}\label{proposition: structure of profinite measure spaces}
Let $I$ be a directed partial order and let $M$ be a $2$-regular $I$-graded inverse system of finite  probability spaces.  
 
Then $L^2(\lim_{\leftarrow}M)$ is the completed  orthogonal sum of a family of strictly interpretable subspaces $(V_i)_{i \in I}$ where each $V_i$ is the completed orthogonal sum of a family of finite dimensional subspaces generated by an asymptotically free definable set $D_i$.

Each $D_i$ is in a classical imaginary sort of $M$ and can be expressed explicitly in terms of  $M$. 

%
%
%
%
%
%
\end{proposition}

\begin{proof}
In this proof, we use the notation of Definition \ref{definition: inverse system of finite probability spaces} and  for any $x \in S_i$, we write $[x]$ for the $E$-class of $x$. We also identify elements $x \in S_i$ with vectors of $L^2(\lim_{\leftarrow}M)$.  More precisely, for any $M \models T$,  we identify $x \in S_i$ with $\I_{\pi_{[x]}^{-1}(x)}$ where  $\pi_{[x]} : (\lim_{\leftarrow}M) \to [x]$ is the inverse limit map to $[x]$.

Fix   $i \in I$. For any $E$-class $a$ in $S_i$, we write $H(a)$ for the finite-dimensional subspace of $L^2(\lim_{\leftarrow}M)$ spanned by the elements of $a$. We also write $W(a)$ for the subspace of $H(a)$ spanned by subspaces $H(b)$ where $b < a$.
 $W(a)$ is a finite-dimensional subspace of $H(a)$
   spanned by  fibres of proper surjections $a \to b$ where $b$ is an $E$-class of $S_i$. 

For any $x \in S_i$,  $P_{W[x]^\perp}x$ is a linear combination of $x$ and of vectors spanning $W[x]$. The coefficients in this linear combination depend only on the various inner products achieved between $x$ and the generators of $W[x]$. It follows that the map $S_i \times S_i\to \Rr, (x, y) \mapsto \langle P_{W[x]^\perp}x, y\rangle$ is strictly definable.

 Let $C_i$ be the classical imaginary sort of $T$ obtained by quotienting $S_i$ so that   $x \mapsto P_{W[x]^\perp}x$ is injective. $C_i$ is a   piece of $\mathcal{H}$ where $z \in C_i$ is identified with $P_{W[x]^\perp}x$ for some $x \in S_i$. The definable inclusions $S_j \to S_i$ for $j \leq i$ induce inclusions $C_j \subseteq C_i$ in $\mathcal{H}$. As remarked above, the inner product map on each $C_i$ is strictly definable.

%

\begin{claim}
Let $z_1, z_2 \in C_i$ and find $x_1, x_2 \in S_i$ such that $z_k = P_{W[x_k]^\perp}x_k$ for $k = 1,2$. If $[x_1] \neq [x_2]$ then $\langle z_1, z_2\rangle = 0$. Hence $C_i$ is asymptotically free.
\end{claim}

\begin{proof}
We can write $z_1 = \sum_{v \in [x_1]}\lambda_v v$ and $z_2 = \sum_{w \in [x_2]} \gamma_w w$ where the $\lambda_v$ and $\gamma_w$ are scalars. We have 
\[
\langle z_1, z_2 \rangle = \langle \sum_{u \in [x_1] \wedge [x_2]} \sum_{v \geq u} \lambda_v v, \sum_{u \in [x_1 ]\wedge [x_2]} \sum_{w \geq u} \gamma_w w\rangle  = \sum_{u \in [x_1] \wedge [x_2]} \langle \sum_{v \geq u} \lambda_v v, \sum_{w \geq u} \gamma_w w\rangle
\]
so it is enough to show that for every $u \in [x_1]\wedge [x_2]$, we have $\langle \sum_{v \geq u} \lambda_v v, \sum_{w \geq u} \gamma_w w\rangle = 0$. Note that $\langle z_1, u \rangle = 0$. Fixing such a $u$ we have
\begin{IEEEeqnarray*}{rCl}
\langle \sum_{v \geq u} \lambda_v v, \sum_{w \geq u} \gamma_w w\rangle & = & \sum_{v \geq u} \sum_{w \geq  u} \lambda_v \gamma_w \mu(v \vee w) \\
& = & \sum_{v \geq u} \sum_{w \geq  u} \lambda_v \gamma_w  \frac{\mu(v) \mu(w)}{\mu(u)} \\ 
& = & \big(\sum_{v \geq u} \lambda_v \frac{\mu(v)}{\mu(u)} \big) \big(\sum_{w \geq u} \gamma_w \mu(w)\big) \\
& = & \frac{1}{\mu(u)} \langle z_1 , u\rangle  \big(\sum_{w \geq u} \gamma_w \mu(w)\big) = 0 .
\end{IEEEeqnarray*}
For any $z_1, z_2 \in C_i$, it follows that if $\langle z_1, z_2 \rangle  \neq 0$, then $z_1, z_2$ are contained in some common $H(a)$. Since $H(a)$ is finite-dimensional,  we have $\acl(z_1) = \acl(z_2)$ and hence $C_i$ is asymptotically free. 
 \end{proof}

For every $i \in I$ we take the quotient of  $C_i$ which identifies the subset $\bigcup_{j < i } C_j$ with $0$. Hence the sets $D_i$ are pairwise orthogonal and asymptotically free. By the claim, $D_i$ generates an orthogonal family of finite dimensional subspaces of $L^2(\lim_{\leftarrow}M)$.  The proposition follows. 
\end{proof}

\noindent \textbf{Remarks:} (1) The decomposition procedure of Proposition \ref{proposition: structure of profinite measure spaces} is analogous to the decomposition procedure of Proposition \ref{proposition: asymptotically free generating set for general strictly interpretable hilbert space}. More precisely, recall that in Proposition \ref{proposition: asymptotically free generating set for general strictly interpretable hilbert space}, for $x \in S_i$, we wrote $V(x)$ for the finite dimensional space spanned by $\{w \in \mathcal{P}(\tp(x)) \mid w < x\}$. Then $V(x) \subseteq W[x]$ where $W[x]$ is defined as in the proof of Proposition \ref{proposition: structure of profinite measure spaces}.

To see this, fix $x \in S_i$ in $M$. It follows from Lemma \ref{lemma: weak closure is set of limit points} and Theorem \ref{projections commute}   that $\{y \in \mathcal{P}(\tp(x)) \mid y  < x\} $ is the set of weak limit points of infinite indiscernible sequences in $\tp(x)$ which begin at $x$. 
Let $(x_n)$ be such an indiscernible sequence. We leave it as an exercise to show that $(x_n)$ converges weakly to $\frac{\mu(x_0 \vee x_1)}{\mu((x_0 )}(x_0 \wedge x_1)$. Therefore 
\[
\mathcal{P}(\tp(x)) = \{\frac{\mu(x_0 \vee x_1)}{\mu(x_0)}(x_0 \wedge x_1) \mid (x_0, x_1)\text{ extends to an infinite indiscernible sequence}\}
\]

%
%

The above observation gives an interesting interpretation of the one-basedness phenomenon proved in Theorem \ref{projections commute}: in the present setting, one-basedness in the $L^2$-space  follows directly from axiom scheme (\ref{equation: modularity in profinite probability space}) in Definition \ref{definition: inverse system of finite probability spaces}, which is a modularity axiom. One-basedness in the $L^2$-space holds much more generally, since Theorem \ref{projections commute} holds in all $L^2$-spaces of inverse systems of finite probability spaces. Nonetheless, we have proved that when the inverse system is $2$-regular,  Hilbert space one-basedness interacts well with the underlying inverse limit structure and has a natural interpretation. 

Because the decomposition constructed in the proof of Proposition \ref{proposition: structure of profinite measure spaces} can be given explicitly in terms of the underlying inverse limit structure, we view Proposition \ref{proposition: structure of profinite measure spaces} as an explicit and meaningful strengthening of Proposition \ref{proposition: asymptotically free generating set for general strictly interpretable hilbert space}.

(2) Proposition \ref{proposition: structure of profinite measure spaces} can also be viewed as a clear illustration of some of the limitations of the techniques developed in this paper. We remarked after  Definition \ref{definition: inverse system of finite probability spaces} that $2$-regularity is strictly weaker than $3$-regularity, but we showed in Proposition \ref{proposition: structure of profinite measure spaces} that $2$-regularity already entails asymptotic freedom.  This shows clearly that Hilbert spaces are not the right setting for detecting interactions of arity $>2$ in $L^2$-spaces. 

(3) Note that the definable sets $D_i$ in Proposition \ref{proposition: structure of profinite measure spaces} are not complete types, so Proposition \ref{proposition: structure of profinite measure spaces} is not an instance of Theorem \ref{theorem: summary of general scattered  interpretable hilbert space}. However, applying the proof of Theorem \ref{theorem: summary of general scattered  interpretable hilbert space} to the sets $D_i$ in Proposition \ref{proposition: structure of profinite measure spaces} yields the following statement: for some indexing set $J$, $L^2(\lim_{\leftarrow} M)$ is the completed orthogonal sum of $\bigwedge$-interpretable subspaces $(W_j)_{j \in J}$ such that each $W_j$ is the completed orthogonal sum of a family of finite dimensional subspaces generated by an asymptotically free complete type $p_j$.

(4) Let $(M, K) \models T_P$ as in Theorem \ref{theorem: galois paper}.  
 The regular representation of $Gal(K)$ on $L^2(Gal(K))$ is the unitary representation generated by the group action $g : \I_{hN} \mapsto \I_{ghN}$  where $N \unlhd G$ is open. It is clear that every finite quotient of $Gal(K)$ acts definably on each sort $S_i$  of $Gal(K)$ and this induces a definable action on the interpretable Hilbert space $L^2(Gal(K))$. Hence we say that the regular representation of $Gal(K)$ is definable.

It is easy to check that $Gal(K)$ also acts on the definable sets $(D_i)_I$ constructed in the proof of Proposition \ref{proposition: structure of profinite measure spaces}. In fact, the action by $Gal(K)$ preserves each finite-dimensional subspace entering in the orthogonal sum in each $V_i$, with $V_i$ defined in the statement of Proposition \ref{proposition: structure of profinite measure spaces}.
 Hence we have recovered the well-known fact that the regular representation of a profinite group is the orthogonal sum of finite dimensional representations and we have shown that each of these finite dimensional representations is definable.

(5) Let $(M, K) \models T_P$ as in Theorem \ref{theorem: galois paper}. It is well-known that  the irreducible characters of finite quotients of $Gal(K) $ form an orthonormal set which generates  $L^2(Gal(K)/Ad(Gal(K)) )$. Unfortunately, the vectors in the sets $D_i$ constructed in the proof of Proposition \ref{proposition: structure of profinite measure spaces} are not usually irreducible characters of finite quotients of $Gal(K)$. Nonetheless,  it is easy to check that each finite-dimensional subspace of each $V_i$ in Proposition \ref{proposition: structure of profinite measure spaces}  is  spanned by a finite set of irreducible characters of finite quotients of $Gal(K)$. 
Therefore Proposition \ref{proposition: structure of profinite measure spaces} naturally leads us to character theory in this setting.

\subsection{Definable Measures}\label{subsection: definable measures}

\emph{In this section, we discuss   classical logic theories $T$ with a strictly definable measure $\mu$.  Examples include pseudofinite fields and more generally MS-measurable theories. For any definable set $X$, $L^2(X)$ is strictly interpretable in $T$.}

\emph{We introduce the strong germ property for $L^2(X)$ and show in Proposition \ref{proposition: strong germ property} that it holds in several cases of interest . In Corollary \ref{corollary: omega-categorical free profinite probability space}, we show that for any $\omega$-categorical measurable structure $M$, type-spaces over $M$ are, up to measure $0$,  $2$-regular inverse systems of finite probability spaces, as defined in Section \ref{subsection: galois groups}. This connection to Section \ref{subsection: galois groups} comes as a surprise.}

\medskip
 Suppose $T$ is a  classical logic theory with elimination of imaginaries and with a Keisler measure $\mu$ on a sort $X$ of $T$. This means that for all $M \models T$, $\mu$ is a finitely additive probability measure on the Boolean algebra $Def_x(M)$ of $M$-definable subsets in the variable $x$, where $x$ ranges in $X$. We view $Def_x(M)$ as an algebra of subsets of the type space $S_x(M)$. Suppose in addition that $\mu$ is \emph{strictly definable}, meaning that for any formula $\phi(x, y)$ and any $\lambda \geq 0$, the set of $a$ in $M$  such that $\mu(\phi(x, a)) = \lambda$ is a definable set. One important example is the theory $T$ of pseudofinite fields  with the counting measure $\mu$. $\mu$ was first shown to be definable in  \cite{Chatzidakis1992}. MS-measurable classes introduced  in \cite{Macpherson2008} generalise the case of pseudofinite fields and offer a rich source of examples.

Given $M \models T$, the measure $\mu$ on the algebra $Def_x(M)$ extends to a countably additive probability measure on the $\sigma$-algebra $\mathcal{D}_x(M)$ generated by $Def_x(M)$. We view $\mathcal{D}_x(M)$ as a $\sigma$-algebra over the type space $S_x(M)$ but when $M$ is $\omega_1$-saturated  we can also view $\mathcal{D}_x(M)$ as a $\sigma$-algebra of   subsets of $X$ itself. Write $L^2(X(M))$ for the  space of square-integrable functions on   $S_x(M)$  with respect to $\mathcal{D}$ and $\mu$.  $L^2(X(M))$ is densely generated by     functions of the form $\I_{\phi(x, a)}$.

For any formula $\phi(x, y)$, let $S_\phi$ be an imaginary sort of $T$ which is a Cartesian product of sorts of $T$ corresponding to the tuple $y$. If $\phi(x, y)$ and $\psi(x, y)$ are different formulas, we take the sorts $S_\phi$ and $S_\psi$ to be distinct, although they are copies of each other.

For $M \models T$ and for any formula $\phi(x, y)$, define the map $h_\phi : S_\phi \to L^2( X(M))$ by $h_\phi(a) = \I_{\phi(x, a)}$. By definability of the measure, the collection of maps $(h_\phi)$ gives an interpretable Hilbert space $\mathcal{H}$ such that $H(M)$ is canonically isomorphic to $L^2(X(M))$. We say that $T$ \emph{interprets $L^2(X)$}.  $L^2(X)$ is strictly interpretable, in the sense of Definition \ref{definition: strictly interpretable Hilbert space}. Therefore   the decomposition Theorem \ref{theorem: summary of general scattered  interpretable hilbert space} applies and we know that $L^2(X, \mu)$ is an orthogonal sum of $\bigwedge$-interpretable subspaces generated by asymptotically free complete types.

 We are interested in finding an explicit interpretation of Theorem \ref{theorem: summary of general scattered  interpretable hilbert space} for $L^2(X)$. It is not clear if it is possible to do this in full generality, but we will show in Corollary \ref{corollary: omega-categorical free profinite probability space} that when $T$ is an $\omega$-categorical measurable structure satisfying Fubini, $X(M)$ can be interpreted as a $2$-regular inverse system of finite probability spaces, and hence the explicit decomposition procedure of Proposition \ref{proposition: structure of profinite measure spaces} applies. 

Our main tool for Corollary \ref{corollary: omega-categorical free profinite probability space} will be the strong germ property and Proposition \ref{proposition: strong germ property} which asserts a certain form of elimination of hyperimaginaries for certain measurable theories. 

In the following, we write $\acl(0)$ for the algebraic closure of the emptyset, in the sense of classical logic. This is not to be confused with $\bdd(0)$, which requires us to view $T$ as a continuous logic theory. From the point of view of classical logic, $\bdd(0)$ can be identified with certain  hyperimaginaries of $T$. See Sections \ref{subsection: imaginaries in CL} and \ref{subsectionL CL and classical logic} for a discussion.
 

\begin{definition}
Let $T$ be a  classical logic theory  with a strictly definable measure. For $M \models T$, write $V_0$ for the subspace of $L^2(X(M))$ generated by vectors $\I_{\phi}$ where $\phi$ is a definable subset of $X$ over $\acl(0)$. 

We say that $T$ has the strong germ property if for some (any) $M \models T$, $V_0 = \bdd(0)$ in $L^2(X(M))$.
\end{definition}

 The strong germ property is of general interest in light of the one-basedness property of Theorem \ref{projections commute}. This theorem finds strong structural properties to hold between $\bdd$-closed subspaces of $L^2(X)$, so it is interesting to identify those subspaces explicitly. Adding parameters, the strong germ property says that $\bdd(a) \cap L^2(X)$ is the subspace we would expect. 

 Note that in any $M \models T$,  $V_0$ and $\bdd(0)$ correspond to subspaces of $L^2(X(M))$ of measurable functions with respect to various $\sigma$-algebras. If $\phi(x, a)$ is  a definable subset of $X$, $P_{\bdd(0)} \I_{\phi(x, a)}$ and $P_{V_0}\I_{\phi(x, a)}$ are the Radon-Nikodym derivatives of $\I_{\phi(x, a)}$ with respect to the appropriate $\sigma$-algebras. While we do not pursue this further, the next lemma asserts a form of probabilistic independence with respect to a certain disintegration of $\mu$ along $\bdd(0)$. See \cite{HrushovskiApproxEq} for a discussion.

\begin{lemma}\label{lemma: first independence for strong germ}
Let $T$ and $L^2(X)$ be as above. Let $\phi(x, y)$ and $\psi(x, z)$ be formulas of $T$ where $x$ ranges in $X$. Let $M \models T$ and suppose $a, b\in M$ are independent over $\bdd(0)$ with respect to the inner product formulas of $L^2(X)$. Then $P_{\bdd(0)}\I_{\phi(x, a) \wedge \psi(x, b)} = (P_{\bdd(0)}\I_{\phi(x, a)})( P_{\bdd(0)}\I_{\psi(x, b)})$ almost everywhere, viewed as functions on the probability space $(S(M), \mu)$.
\end{lemma}

\begin{proof}
Let $g \in L^2(X(M))$ be a function $S(M) \to \Rr$ which is in $\bdd(0)$. To make notation lighter, we write $H_0 = \bdd(0) \cap H(M)$. It is enough  to show 
\[
\langle P_{H_0}\I_{\phi(x, a) \wedge \psi(x, b)}, g\rangle = \langle (P_{H_0}\I_{\phi(x, a)})( P_{H_0}\I_{\psi(x, b)}), g\rangle.
\]
Note that 
$
\langle P_{H_0}\I_{\phi(x, a) \wedge \psi(x, b)}, g\rangle = \langle \I_{\phi(x, a) \wedge \psi(x, b)}, g\rangle = \langle \I_{\phi(x, a)} , g\I_{\psi(x, b)}\rangle.
$
Since $g \I_{\psi(x, b)} \in \bdd(b)$ and $\bdd(a), \bdd(b)$ are independent over $\bdd(0)$ in $L^2(X(M), \mu)$, we have
\begin{IEEEeqnarray*}{rCl}
\langle \I_{\phi(x, a)} , g\I_{\psi(x, b)}\rangle  
& = &\langle  P_{H_0}\I_{\phi(x, a)}, gP_{H_0}\I_{\psi(x, b)}\rangle\\
& = & \langle (P_{H_0}\I_{\phi(x, a)})(P_{H_0}\I_{\psi(x, b)}), g\rangle.
\end{IEEEeqnarray*}
\end{proof}

The argument for the next proposition is adapted from the independence theorem for probability logic of \cite{HrushovskiApproxEq}. We assume that $T$ carries a definable measure on all definable sets and we assume that the measure satisfies Fubini (see Definition 3.1 in \cite{ElwesMacpherson2008}).

\begin{proposition}\label{proposition: strong germ property}
Let $T$ be a classical logic theory with a strictly definable measure satisfying Fubini. Suppose $T$ eliminates imaginaries.

  Let $M \models T$ and suppose that for any formula $\phi(x, a) \subseteq X$ with $\phi(x, y) \subseteq X \times Y$ there is a positive-measure $\acl(0)$-definable set $Y' \subseteq Y$ containing $a$ 
such that
\begin{enumerate}
\item for every $\acl(0)$-definable set $Z \subseteq X$,   for all $a' \in Y'$, $\mu(\phi(x, a') \wedge Z(x)) = \mu(\phi(x, a) \wedge Z(x)) $
\item for every $\acl(0)$-definable $Z \subseteq X$  and any two   pairs $(a_1, a_2), (a_1', a_2')$ in $(Y')^2$   with $a_1 $ independent from $a_2$ over $\acl(0)$ with respect to the stable definable function $\mu_x(\phi(x, y_1) \wedge \phi(x, y_2) \wedge Z(x))$ and similarly for  $a_1', a_2'$, we have  
\[
\mu(\phi(x, a_1) \wedge \phi(x, a_2) \wedge Z(x)) = \mu(\phi(x, a_1') \wedge \phi(x, a_2') \wedge Z(x)).
\]
\end{enumerate}
Then $T$ has the strong germ property.
\end{proposition}

\begin{proof}
We   work inside an $\omega_1$-saturated model $M$. Write $V_0$ for the subspace of $L^2(X(M))$ generated by $\acl(0)$ definable sets and $H_0 = \bdd(0)$, so that $V_0 \leq H_0$. In this proof, we also write $\phi(x, a)$ instead of the indicator function $\I_{\phi(x, a)}$. 

 We show that $P_{V_0}(P_{H_0} \phi(x, a))^2 = (P_{V_0}\phi(x, a))^2$. Given this identity, we will have  
 \[
 \|P_{H_0}\phi(x, a)\|^2  = \int_X (P_{H_0} \phi(x, a))^2 d\mu = \int_X P_{V_0}(P_{H_0} \phi(x, a))^2 d\mu = \|P_{V_0} \phi(x, a)\|^2
 \] and the result will follow.
It is enough to show that for any $\acl(0)$-definable set $Z(x)$, we have 
\begin{equation}\label{equation: strong germ 1}
\int_X P_{V_0}(P_{H_0} \phi(x, a))^2 Z(x) d\mu(x) = \int_X  (P_{V_0} \phi(x, a))^2 Z(x) d\mu(x).
\end{equation}
Take $a'$ in $M$ independent from $a$ over $\bdd(0)$ realising $\tp(a/\bdd(0))$. Note that $P_{H_0} \phi(x, a) = P_{H_0} \phi(x, a')$.  By Lemma \ref{lemma: first independence for strong germ}, we have 
\begin{IEEEeqnarray*}{rCl}
\int_X P_{V_0}(P_{H_0} \phi(x, a))^2 Z(x) d\mu(x) &  = &  \int_X (P_{H_0} \phi(x, a))^2 Z(x) d\mu(x) \\
& = & \int_X P_{H_0} (\phi(x, a)  \phi(x, a'))Z(x)d\mu(x) \\
& = & \int_X \phi(x, a) \phi(x, a') Z(x)d\mu(x)\\
& = & \frac{1}{\mu(Y')^2} \int_{y \in Y'}\int_{y' \in Y'}\int_X \phi(x, y) \phi(x, y') Z(x) d\mu(y') d\mu(y)d\mu(x) \\
& = &  \frac{1}{\mu(Y')^2} \int_X (\int_{Y'} \phi(x, y) d\mu(y))^2 Z(x)d\mu(x)
\end{IEEEeqnarray*}
To deduce (\ref{equation: strong germ 1}), it is enough to show that $\int_{Y'}\phi(x, y) d\mu(y) = \mu(Y') P_{V_0}(\phi(x, a))$ on $X$. Since $\int_{Y'} \phi(x, y )d\mu(y) = \mu_y(Y'(y) \wedge \phi(x, y))$, which is a function in $V_0$, this follows by the   computation:
 \begin{IEEEeqnarray*}{rCl}
\mu(Y')\int_XP_{V_0} \phi(x, a) Z(x) d\mu(x) & = &  \mu(Y')\int_X \phi(x, a) Z(x) d\mu(x) \\
&=  & \int_X (\int_{Y'} \phi(x, y)   d\mu(y))Z(x)d\mu(x)
\end{IEEEeqnarray*}
\end{proof}

\noindent \textbf{Remark:  } The conditions of Proposition \ref{proposition: strong germ property} always hold when $T$ is an  $\omega$-categorical  measurable theory and $\mu$ satisfies Fubini. The stationarity theorem of \cite{ChevalierLevi2022} shows that the conditions of Proposition \ref{proposition: strong germ property} also hold when $T$ is the theory of pseudofinite fields or more generally when $T = ACFA$ and $\mu$ is the definable measure on sets of finite total dimension.

\medskip 

Let $T$ be a measurable theory, let $M \models T$ and let $a \subseteq M$. Let $x$ be a tuple of variables in the sort $X$. Let $\mathcal{A}(a)$ be the algebra of $M$-definable subsets of $X$ obtained by identifying  $\phi, \psi \in Def_x(a)$ if $\phi$ and $\psi$ agree almost everywhere. In the following corollary, we write $S'_x(a)$ for the Stone space of $\mathcal{A}(a)$. It is clear that $S'_x(a)$ is a probability space and $L^2(X(a)) \cong L^2(S'_x(a))$ canonically.

 
\begin{corollary}\label{corollary: omega-categorical free profinite probability space}
Let $T$ be an $\omega$-categorical classical logic theory with a strictly definable measure satisfying Fubini. 
Let $X$ be a Cartesian product of sorts of $T$ and $M \models T$. 

Then $S'_x(M)$ is interpretable in $M$ as a $2$-regular inverse system of finite probability spaces.

\end{corollary}

\begin{proof}
We can assume that $T$ admits elimination of imaginaries. Fix $M\models T$. We consider finite tuples $a$ in $M$ such that $S'_x(a) = S'_x(\acl(a))$. Let $(D_i)_{i \in I}$ be the collection of definable sets which give the complete types of such tuples. If $a \in D_i$ and $b \in D_j$, we define $a \leq b$ if every positive-measure $a$-definable set is a union of positive-measure $b$-definable sets, up to measure $0$. For $i, j \in I$, define $i \leq j$ if there are $a \in D_i$ and $b \in D_j$ with $a \leq b$. 
As defined, $I$ is a preorder but we can move to a subset of $I$ so that $I$ becomes a partial order.

For every $i \in I$, we can construct  a piece  $S_i$  of $L^2(S'_x(M))$ in a classical imaginary sort of $T$  consisting of the vectors $\I_p$ where $p \in S'_x(a)$ for some $a \in D_i$. $S_i$ carries an equivalence relation $E$ such that the $E$-classes of $S_i$ are the sets $\{\I_p \mid p \in S'_x(a)\}$ for $a \in D_i$. For $i \leq j \in I$, we can define $\I_p \leq \I_q$ if the unique $a \in D_i$ and $b \in D_j$ coding the $E$-classes of $\I_p$ and $\I_q$ satisfy $a \leq b$ and if $q \subseteq p$ up to measure $0$. 

It is clear that the sets $(S_i)_I$ define an inverse system of finite probability spaces. We  show that it is free. Fix $i \in I$ and $a, b\in D_i$. Take $p \in S'_x(a)$, $q \in S'_x(b)$ and $r \in S'_x(a \wedge b)$ such that $p \subseteq r$ and $q \subseteq r$.   By Theorem \ref{projections commute}, the subspaces of $L^2(S'_x(M))$ given by $\bdd(\I_p)$ and $\bdd(\I_q)$ are orthogonal over $\bdd(\I_p) \cap \bdd(\I_q) = \bdd(a) \cap \bdd(b)$.  By Proposition \ref{proposition: strong germ property}, the subspace of $L^2(S_x'(M))$ given by $\bdd(a) \cap \bdd(b)$ is the subspace generated by $\acl(a) \cap \acl(b)$-definable sets.

Observe that $S'_x(\acl(a) \cap \acl(b)) = S'_x(a \wedge b)$. Therefore $\I_p$ and $\I_q$ are orthogonal over the linear span of $\{\I_r\}$. It follows that the $I$-graded  inverse system $(S_i)_I$ is free. 
\end{proof}

\noindent \textbf{Remark:} Let $T$ be the theory of an infinite dimensional vector space over a finite field with a bilinear nondegenerate symmetric or symplectic form. $T$ is $\omega$-categorical and MS-measurable (see \cite{Macpherson2008} for more details). Then for $M \models T$, $S'_x(M)$ \emph{cannot} be expressed as the inverse limit of a $3$-regular inverse system of finite probability spaces in a definable way. $3$-regularity in $S'_x(M)$ fails in a bad way, since the events $\langle x, a\rangle = 0$, $\langle x, b \rangle = 0$ and $\langle x, a+ b\rangle= 0$ are usually pairwise independent but not independent all together. Therefore Corollary \ref{corollary: omega-categorical free profinite probability space} is the strongest possible result.

This example shows how Hilbert space independence differs from independence in probability algebras, and it shows that the techniques of this paper are Hilbert space specific and do not generalise to arbitrary stably embedded structures. 

\section{Unitary Representations}\label{section: unitary representations}

\emph{  Throughout the paper so far, we were interested in interpretable Hilbert spaces. These are defined at the level of the theory rather than   individual models.  In this section, we will look at the connection with  representation theory, which requires fixing a group. In our context, this amounts to fixing a sufficiently homogeneous model. } 

\emph{In Section \ref{subsection: representations of automorphism groups}, we show that the notion of irreducibility for representations on   interpretable Hilbert spaces does not depend on the choice of model, and in fact is quite local in nature, in the sense that it is witnessed by the representation of $\Aut(\bdd(a))$ where $a$ is a finite tuple. See Propositions  \ref{proposition: representation on bdd-closed subspace}. We also show that interpretable Hilbert spaces provide an interesting framework for studying arbitrary unitary group representations.}

\emph{ In Section \ref{subsection: application to scattered representations}, we study the notion of asymptotic freedom from the point of view of representation theory. We show that asymptotically free complete types give rise to induced representations and we prove a new Mackey-style  irreducibility criterion for these induced representations. See Proposition \ref{proposition: irreducibility of induced representation}. Using a theorem of Howe and Moore, we also show that interpretable Hilbert spaces generated by asymptotically free complete types capture all representations of algebraic groups over local fields of characteristic $0$. }
 
 \emph{ In Section \ref{subsection: omega-categorical structures} we review the special case of $\omega$-categorical structures and we recover the   theorem of \cite{Tsankov2012} which classifies all unitary representations of oligomorphic groups.}

\subsection{Unitary representations of automorphism groups}\label{subsection: representations of automorphism groups}
\emph{In this section, we show how interpretable Hilbert spaces give rise to unitary representations of automorphism groups. We show that irreducibility is a property of the interpretable Hilbert space $\mathcal{H}$. Using a general continuous logic construction and a theorem of Howe and Moore, we also show that interpretable Hilbert spaces provide an interesting framework for studying unitary representations of arbitrary groups.}

\medskip

 We start by recalling some basic definitions. See the appendix of \cite{Bekka2008}  for more background on unitary representations.
 
\begin{definition}
Let $G$ be a group and $H$ a Hilbert space, real or complex. A unitary representation $\sigma$ of $G$ on $H$ is a group action $G \times H \to H$ such that for every $g\in G$  $\sigma(g)$ is a unitary map if $H$ is a complex Hilbert space and $\sigma(g)$ is an orthogonal map if $H$ is a real Hilbert space.

When $G$ is a topological group, we say that  $\sigma$ is continuous if $\sigma : G \times H \to H$ is continuous. This is equivalent to $\sigma(\cdot, v)$ being continuous for every $v \in V$. 
\end{definition}

\noindent
\textbf{Convention:} In this paper, we only consider continuous unitary representations of topological groups, so we just say `representation' instead of `continuous unitary representation'.

\begin{definition}
Let $G$ be a group and let $\sigma, \sigma'$ be two representations of $G$ on the Hilbert spaces $H$ and $H'$ respectively.   $\sigma$ and $\sigma'$ are equivalent if there is a surjective isometry $U : H \to H'$ such that for all $g \in G$ and $v\in H$, $U (\sigma(g) v) = \sigma'(g) U(v)$. We say that $U$ intertwines $\sigma$ and $\sigma'$.
\end{definition}

Let $T$ be an arbitrary continuous logic theory and let $M \models T$. Then $\Aut(M)$ is a topological group with the topology of pointwise convergence. $\Aut(M)$ has a basis     at the identity consisting of subsets of the form $\{g \in \Aut(M) \mid d(g A, A) < \e, A \subseteq M \text{ finite}, 
\e > 0\}$. When $T$ is a classical logic theory, this is a basis of subgroups. Note that we can add imaginary sorts to $M$ without changing the topology on $\Aut(M)$.

Suppose $\mathcal{H}$ is an interpretable Hilbert space in $T$. For any $M \models T$, $\Aut(M)$ has a \emph{canonical unitary representation} $\pi$  on $H(M)$ given by $\pi(g)h x = h( gx)$ where $h$ is the direct limit map on any piece of $\mathcal{H}$. Since we can add the pieces of $\mathcal{H}$ to $M$ without affecting the topology on $\Aut(M)$,  $\pi$ is a continuous representation.   

  The following lemma is an easy definition chase:
   
  \begin{lemma}
   If $\mathcal{H}$, $\mathcal{H'}$ are isomorphic interpretable Hilbert spaces in $T$, then for any $M\models T$, the representations of $\Aut(M)$ on $H(M)$ and $H'(M)$ are equivalent.
  \end{lemma}

We turn to a discussion of irreducibility for interpretable Hilbert spaces.

\begin{definition}
Let $\mathcal{H}_0$ be an $\bigwedge$-interpretable subspace of $\mathcal{H}$.  
We say that $\mathcal{H}_0$   is irreducible if there do not exist complete types $q, q'$ in  $\mathcal{H}_0$ such that $q(x) \cup q'(y)$ implies $\langle  x,  y \rangle = 0$.
 \end{definition}
 
We will study irreducibility for $\bigwedge$-interpretable Hilbert spaces   in relation with the notion of $\omega$-near-homogeneity See Definition \ref{definition: atomic and omega-homogeneous}. As in the discussion following Definition \ref{definition: atomic and omega-homogeneous}, when discussing a $\omega$-near-homogeneous structure with an interpretable Hilbert spaces $\mathcal{H}$, we will always assume that the $\omega$-near-homogeneity property applies to the pieces of $\mathcal{H}$.

\begin{lemma}\label{lemma: irreducible hilbert space}
Let $\mathcal{H}_0$ be a $\bigwedge$-interpretable subspace of $\mathcal{H}$.  
\begin{enumerate}
\item if  $\mathcal{H}_0$ is irreducible then for any   $\omega$-near-homogeneous $M \models T$  the canonical representation $\pi$ of $\Aut(M)$ on $H_0(N)$ is irreducible.
\item if there is some  $M \models T$ realising all types of $T$ such that the canonical representation $\pi$ of $\Aut(M)$ on $H_0(N)$ is irreducible, then $\mathcal{H}_0$ is irreducible
\end{enumerate}
 
\end{lemma}

\begin{proof}
(1)  Let  $ M \models T$ be $\omega$-near-homogeneous. Let $v, w\in H(M)$ be two nonzero vectors. Write $q_1, q_2$ for their respective types.  By irreducibility of $\mathcal{H}_0$ and the assumption following Definition \ref{definition: atomic and omega-homogeneous},  $d(q_1, q_2)^2  < \|v\|^2 + \|w\|^2$. By $\omega$-near-homogeneity, we can find $g \in \Aut(M)$ such that $\|gv - w\|$ is arbitrarily close to $d(q_1, q_2)$. Then $gv$ and $w$ are not orthogonal and the representation of $\Aut(M)$ on $H_0(M)$ is irreducible.

(2) Take $M \models T$ as in the statement. Let $p, q$ be types of $\mathcal{H}_0$ and let $v, w$ be realisations in $H_0(M)$. There is $g \in \Aut(M)$ such that $\langle gv, w\rangle \neq 0$ so $\mathcal{H}_0$ is irreducible. 
\end{proof}

The following proposition shows that irreducibility is a local notion, when the interpretable Hilbert space is generated by a complete type. See also Proposition \ref{proposition: irreducibility of induced representation} for a qualitatively similar statement.

 \begin{proposition}\label{proposition: representation on bdd-closed subspace}
Let $p$ be a complete type in   $\mathcal{H}$. Take $M \models T$   $\omega$-near-homogeneous and realising all types. 
 
Let $a \models p$ in $H(M)$, let $K$ be the subgroup of $\Aut(M)$ fixing $\bdd(a)$ setwise and let $A = H_p(M) \cap \bdd(a)$. If the canonical representation of $K$ on $A$ is irreducible, then $\mathcal{H}_p$ is irreducible.
 \end{proposition}
 
 \begin{proof}
 Let $q, q'$ be two types in $\mathcal{H}_p$. Let $v, w$ be realisations of $q,q'$ in $H(M)$. Choosing $\Aut(M)$-conjugates of $v$ and $w$ if necessary, we can assume that $P_Av $ and $P_Aw$ are nonzero. Conjugating $v$ by an element of $K$, we can assume that $\langle P_A v, P_A w \rangle \neq 0$.
 
Consider the nonforking extension $r$ of $\tp(v/A)$ to $\bdd(Aw)$ with respect to the inner product maps of $\mathcal{H}$. If $z$ realises $r$ in an $\omega_1$-saturated elementary extension $N$ of $M$, we have $\langle z, w\rangle = \langle P_Av, P_Aw\rangle \neq 0$. Let $r' = \tp(z, w)$.
  
  By our assumption on $M$, $r'$ is realised in $M$ by some pair $(z', w')$. By $\omega$-near-homogeneity, we can  assume that  $w'$ is arbitrarily close to $ w$.  Now $\tp(z') =  \tp(v)$ so we can find $g \in \Aut(M)$ taking $v$ arbitrarily close to $z'$. Now we have $  \langle gv, w\rangle \approx \langle z', w'\rangle  = \langle z, w \rangle \neq 0$.
  \end{proof}

 We now introduce a construction which allows us to connect the study of interpretable Hilbert spaces with the study of unitary representations of arbitrary groups. Given a faithful representation of a group $G$ on a Hilbert space $H$, we will construct a continuous logic structure with universe generated by the orbit of an arbitrary vector in $H$ under the action of $G$. We use the following general continuous logic construction:

\begin{definition}
Let $M$ be a continuous logic structure in a language $\mathcal{L}$   and let $G$ be a subgroup of $\Aut(M)$. We define an expansion  $M_G$ of $M$ in a language $\mathcal{L}_G$ as follows. 

For every $n \geq 1$, for every finite Cartesian product $P$ of sorts of $M$ and for every orbit $O$ of $G$ on $P$, we add a  function symbol $r_O : P \to [0, \infty)$. $M_G$ is the structure obtained from $M$ by interpreting each function $r_O(x)$ as $d(x, \overline{O})$, the distance between $x$   and the metric closure of $O$.
 \end{definition}

We could equivalently define $M_G$ as the maximal $G$-invariant expansion of $M$. The following lemma is straightforward: 

\begin{lemma}
Let $M$ be a continuous logic structure, $G$ a subgroup of $\Aut(M)$. Then $\Aut(M_G)$ is the closure of $G$ in $\Aut(M)$ with the topology of pointwise convergence. 
\end{lemma}

 We apply the above construction to unitary representations:

\begin{lemma}\label{lemma: elementary facts about G-expansion of representation}
Let $G$ be an arbitrary group and let $H$ be a Hilbert space with a cyclic faithful representation $\sigma$ of $G$. Let $v$ be a cyclic vector and let $X$ be the closure of the orbit of $v$ in $H$.  
Let $M$ be the continuous logic structure with universe  $X$ and with the inner product map on $X$ induced from $H$ as   only function symbol. The following hold:
\begin{enumerate}
\item $M_G$ is an atomic $\omega$-near-homogeneous structure (see Definition \ref{definition: atomic and omega-homogeneous}) and $X$ is a complete type in $M_G$ 
\item there is an interpretable Hilbert space $\mathcal{H}$ in $Th(M)$ generated by the sort $X$ such that, writing $\pi \upharpoonright G$ for the restriction of the canonical representation of $\Aut(M_G)$ to $G$, we have $\pi \upharpoonright G \simeq \sigma$.
\item $\Aut(M_G)$ is isomorphic to the closure of $G$ in $U(H)$, the unitary group of $H$, in the weak operator topology.
\end{enumerate}
\end{lemma}

\begin{proof}
(1) and (2) are trivial. For (3), it is enough to note that the weak operator topology on $U(H)$ coincides with the pointwise convergence topology on $U(H)$ with respect to the set $X$. 
\end{proof}

 It is also clear that the structure $M_G$ above remains atomic and  $\omega$-near-homogeneous under expansion by the pieces of $\mathcal{H}$ so that the assumption following Definition \ref{definition: atomic and omega-homogeneous} holds.

A classical theorem of \cite{HoweMoore1979} shows that when $G$ is a connected Lie group, $\Aut(M_G)$ is not significantly larger than $G$. We recall this theorem here:

\begin{theorem}[\cite{HoweMoore1979}, Theorem 2.1]\label{theorem: howe moore for closure in unitary group}
Let $G$ be a connected Lie group with a faithful continuous  unitary representation   on $H$ and write $\overline{G}$ for the closure of $G$ in $U(H)$, the unitary group of $H$, with the weak operator topology. Then 
\begin{enumerate}
\item The normaliser of $G$ in $U(H)$ is closed and hence $G$ is normal in $\overline{G}$
\item $\overline{G}/G$ is at most 2-step nilpotent
\item $\overline{G}/Z(\overline{G})$ is a Lie group and $\overline{G}/G\cdot Z(\overline{G})$ is Abelian
\item If the adjoint group of $G$ is closed, then $\overline{G}  =G  \cdot Z(\overline{G})$. 
\end{enumerate}
\end{theorem}

See also Theorem 3.1 in \cite{HoweMoore1979} for an extension of the above result to connected groups $G$.

\medskip

In the next section, we will show that asymptotic freedom is a natural notion from the point of view of representation theory. We remark here that it is more difficult to interpret directly our notion of scatteredness in representation theoretic terms. The main obstacle is that we require $\omega$-saturation to verify scatteredness. 
 While $\omega$-saturation holds in ultraproducts of structures $M_G$, it   does not usually hold for a single structure. 
 
Nevertheless, since every structure $M_G$ is atomic and $\omega$-near-homogeneous, Proposition \ref{proposition: decomposition for atomic omega-near-homogeneous} applies when the weak closure of the orbit $X$ is locally compact. Proposition \ref{proposition: decomposition for atomic omega-near-homogeneous} can be easily reformulated in purely representation theoretic terms:

\begin{proposition}\label{proposition: modified proof of theorem for locally compact orbit}
Let $G$ be a group and let   $\sigma$ be a faithful cyclic unitary representation of   $G$ on $H$. Let $v$ be a cyclic vector and let $X$ be the metric closure of the orbit of $v$ in $H$. 

Suppose that the weak closure of $X$ in $H$ is locally compact in the metric topology. Then $\sigma$ is equivalent to an orthogonal sum of representations $(\tau_i)$ such that each $\tau_i$ has a cyclic vector $v_i$ satisfying the following: if $w$ is a conjugate of $v_i$ such that $\langle w, v_i\rangle \neq 0$, then the orbit of $w$ under the stabilizer of  $v_i$ is precompact.
\end{proposition}

\subsection{Unitary representations with asymptotically free orbits} \label{subsection: application to scattered representations}

\emph{In this section we study interpretable Hilbert spaces generated by asymptotically free complete types from the point of view of representation theory. We show that asymptotically free complete types give rise to induced representations, and we prove a Mackey-style irreducibility criterion for these induced representations. }

\emph{We also relate asymptotic freedom to the notion of disintegrated strongly minimal sets  in continuous  logic.}

\emph{Finally, using a theorem of Howe and Moore, we show that interpretable Hilbert spaces generated by an asymptotically free complete type capture all irreducible unitary representations of   algebraic groups over a local field of characteristic $0$.}
 \medskip

  We begin by recalling the notion of induced representation in the special case where we induce from an open subgroup. See \cite{Bekka2008} for more details. Let $G$ be a topological group and take $K$ an open subgroup of $G$. Let $\sigma$ be a representation of $K$ on a Hilbert space $V$. We suppose here that $V$ is a real Hilbert space (the case of complex Hilbert spaces is similar). Write $\Rr G$ for the free vector space on $G$. We define $G \otimes_\sigma V$, the $\sigma$-tensor of $G$ and $V$, to be the vector space $\Rr G \otimes V$ quotiented by a suitable subspace so that $gk \otimes_\sigma v = g \otimes\sigma(k) v$ for all $g \in G$ and $k \in K$.

We define an inner product on $G \otimes_\sigma V$ as follows. For any  $g , g' \in G$ and $v,v' \in V$,  if $gK \neq g'K$, then $\langle g \otimes_\sigma v, g' \otimes_\sigma v'\rangle = 0$. If $gK = g'K$, then find $k \in K$ such that $g' = gk$ and define $\langle g \otimes_\sigma v, g' \otimes_\sigma v'\rangle =\langle v, \sigma(k) v'\rangle$. Observe that if we choose a set of coset representatives for $G /K$, we can identify $G \otimes_\sigma V$ with the orthogonal sum of copies of $V$ indexed by $G/K$. 

We will always work with the Hilbert space completion of $G \otimes_\sigma V$. We also write $G \otimes_\sigma V$ for this completion. We  define the induced representation of $G$ from $\sigma$, denoted $\Induce_K^G(\sigma)$, as the unitary  representation of $G$ on $G \otimes_\sigma V$ given by $g \cdot( g' \otimes_\sigma v) = gg'\otimes_\sigma v$. Since $K$ is open in $G$, the induced representation is continuous.

\begin{proposition}\label{proposition: representation theory of asymptotically free type}
Let $T$ be a continuous logic theory with an interpretable Hilbert space $\mathcal{H}$.  Suppose $p$ is an asymptotically free  complete  type  in   $\mathcal{H}$. For $x, y \models p$, write $x \sim y $ if $\bdd(x) = \bdd(y)$ and write $[x]$ for the equivalence class of $x$ under $\sim$.

Let   $M\models T$, write $G = \Aut(M)$ and suppose that for some (any) $a \models p$, the orbit of $a$ under $G$ is metrically dense in $p$. Fix $a \models p$ in $H(M)$ and write $K$ for the open subgroup of elements of $G$ which fix $[a]$ setwise. Then the canonical representation $\pi$ of $G$ on $H_p(M)$  is equivalent to $\Induce_K^G(\sigma)$ where $\sigma$ is the restriction of $\pi$ to $K$ on  the Hilbert space $V$ spanned by $ [a]$. 

  \end{proposition}

\begin{proof}
Since $p$ is asymptotically free, $p$ is metrically locally compact and hence there is $\e > 0$ such that for any $x, y \models p$, if $d(x, y) < \e$ then $[x] = [y]$. Therefore $K$ is open in $G$. Let $A = [a]$ and write  $\{A_i\mid i \in I\}$ for the orbit of $A$ under $G$ setwise (we ignore permutations of $A$).  For every $i\in I$ pick $g_i\in \Aut(M)$ which maps $A$ to $A_i$. Then $\{g_i\mid i \in I\}$ is a list of representatives for the left cosets of $K$ in $\Aut(M)$.  Since $p$ is asymptotically free, the sets $A_i$ are pairwise orthogonal in $H_p(M)$.

Write $\pi$ for the canonical representation of $\Aut(M)$ on $H_p(M)$ and let $\sigma$ be the  restriction of $\pi$ to $K$ on $V$, the vector space spanned by $A$. Let $\iota = \mathrm{Ind}_{K}^G(\sigma)$ and write $W = G \otimes_\sigma V$. We show that $\iota$ and $\pi$ are equivalent. Write also $g_i\otimes_K V   $ for the subspace of $W$ given by $\{g_i \otimes_\sigma v \mid v \in V\}$. Take $w \in W$ and write $w = \sum w_i$ where $w_i \in g_i \otimes_\sigma V$. Let $P_i : g_i  \otimes_\sigma V \to V$ be the  Hilbert space isomorphism taking $g_i   \otimes_\sigma v$ to $v$. We define
\[
U(w) = \sum \pi(g_i) P_i(w_i).
\]
  
Since $A_i \perp A_j$ for $i\neq j$, we have $\pi(g_i)P_i(w_i)\perp \pi(g_j) P_j(w_j)$, so $U$ is well-defined, and it is easy to check that $U$ is in fact a surjective     isometry.   $U$ intertwines $\iota$ and $\pi$: take $g  \in G, i\in I, v\in V$. Then, writing $gg_i = g_jk$, we have 
\[
U(\iota(g)(g_i\otimes_K v)) = U(g_j \otimes_K \sigma(k)v) = \pi(g_j)(\sigma(k)v) = \pi(g_j k)v = \pi(g)U(g_i \otimes_K v).
\]
\end{proof}

\noindent \textbf{Remark:} Taking $a$ and $K$ as in Proposition \ref{proposition: representation theory of asymptotically free type}, we note that $\Aut(M/[a])$ is a normal subgroup of $K$ contained in the kernel of $\sigma$. Write $G_1$ for the group of automorphisms of the set $[a]$. Then  $\sigma$ factors through $Aut(M/[a])$ to a faithful representation  of a subgroup of $G_1$.   Since $[a]$ is a separable locally compact metric space, the closure of $K / \Aut(M/[a])$ in $G_1$ is  locally compact with respect to the topology of pointwise convergence. We say that the canonical representation of $G$ is obtained from the representation of $K / \Aut(M/[a])$ by inflation.

\medskip

We prove a version of Mackey's irreducibility criterion  for the induced representations arising in Proposition \ref{proposition: representation theory of asymptotically free type}. General irreducibility criteria for induced representations have been studied in \cite{Mackey1951}, \cite{Godement1948}, \cite{Corwin1975}.  See also Proposition 4.1 in \cite{Tsankov2012} for an irreducibility criterion in the case of oligomorphic groups. In all cases, irreducibility of the induced representation is verified at the level of the commensurator (also called the quasi-normalizer) of the subgroup $K$. See \cite{Corwin1975} or \cite{Tsankov2012} for a discussion. Our criterion also takes place at the level of the commensurator of the subgroup we are inducing from.

In all cases cited above, the irreducibility criterion for induced representations holds only when we induce from \emph{a finite dimensional representation}. It is shown in \cite{BekkaCurtis2003} that the naive generalisation of the criterion of \cite{Corwin1975} does not hold for infinite dimensional representations. Our Proposition \ref{proposition: irreducibility of induced representation} has the interesting feature that it does not require finite-dimensionality of the representations we are inducing from.

\begin{proposition}\label{proposition: irreducibility of induced representation}
  Suppose $p$ is an asymptotically free complete type in $\mathcal{H}$. 
 For $x, y \models p$, write $x \sim y $ if $\bdd(x) = \bdd(y)$ and write $[x]$ for the equivalence class of $x$ under $\sim$.

Let $M \models T$ and fix $a \models p$ in $M$. Write $K$ for the subgroup of elements of $G$ which fix $[a]$ setwise. Let $\sigma$ be the restriction of the canonical representation $\pi$ of $\Aut(M)$ to $K$ on the Hilbert space $V$ spanned by $[a]$. Then 
\begin{enumerate}
\item If $\sigma$ is irreducible, then $\mathcal{H}_p$ is irreducible
\item If $M$ is $\omega$-near-homogeneous and $\sigma$ is reducible, then $\mathcal{H}_p$ is reducible.
\end{enumerate}
 

\end{proposition}

\begin{proof}

To see (2), note that by Proposition \ref{proposition: representation theory of asymptotically free type} and easy facts about induced representations, if $\sigma = \sigma_1 \oplus \sigma_2$, then $\pi = \Induce_K^G(\sigma_1) \oplus \Induce_K^G(\sigma_2)$, so $\pi$ is reducible, and this is witnessed by two complete types realised in $M$. Hence  $\mathcal{H}_p$ is reducible.

We prove (1). Suppose that $\sigma$ is irreducible. If we move to an $|M|^+$-strongly homogeneous and $|M|^+$-saturated elementary extension $M'$ of $M$, then the representation of the subgroup of $\Aut(M')$ which fixes $[a]$ setwise is also irreducible on $V$. Hence by Lemma \ref{lemma: irreducible hilbert space}(2), we can assume that $M$ is $\omega_1$-strongly homogeneous.

 Recall that in Proposition \ref{proposition: representation theory of asymptotically free type} we expressed $H_p(M)$ as the orthogonal sum of subspaces $g_i \otimes_{\sigma} V$ where $(g_i)_{i \in I}$ is a set of coset representatives of $G/K$. Suppose $0$ is an indexing element in $I$ with 
  $g_0 = e$ so that $V = g_0 \otimes_{\sigma} V$. 
 
 Suppose that we   have  a $G$-invariant subspace $Z$ of $H_p(M)$.  Then the orthogonal projection $P_Z$ commutes with $G$. Fix  a nonzero $v \in V$. We can write $P_Zv = \sum_{j \in J} u_j$ where $J \subseteq I$ is countable and $u_j \in g_j \otimes_{\sigma} V$ is nonzero. Write $u_0$ for the element of $\{u_j \mid j\in J\}$ which lies in $V$. Since the $u_j$ are pairwise orthogonal, $u_0 =0$ would imply that $P_Zv = 0$. Switching if necessary to $Z^\perp$, we can assume that $P_Z v \neq 0$. 
 
Write $J_0  = \{j \in J \mid g_ja \notin \bdd(a)\}$ and $J_1 = \{j \in J \mid a \notin \bdd(g_ja)\}$. Then $J = \{0\} \cup J_0 \cup J_1$. We show that $J_0 = J_1 = \emptyset$. By $\omega_1$-strong homogeneity  and saturation of $M$, we can find a sequence $(\alpha_n)$ in $\Aut(M/\bdd(a))$ such for any $n \neq m$ 
\[
 \alpha_n\big\{[ g_ja] \mid j \in J_0\big\}  \cap \alpha_m\big\{[ g_ja] \mid j \in J_0\big\} = \emptyset.
\]
Hence for all $n\neq m$, the sets $\alpha_n\{  u_j \mid j \in J_0\}$ and $\alpha_m\{ u_j \mid j \in J_0\}$ are orthogonal. For all $n$ and $j \in J \setminus J_0$, we also have $\alpha_n u_j = u_j$. Therefore $(\alpha_n P_Z(v))$ converges weakly to $ \sum_{j \in J \setminus J_0} u_j$.
 Since $Z$ is $G$-invariant, we have $\sum_{j \in J \setminus J_0} u_j \in Z$. Since $\sum_{j \in J \setminus J_0} u_j$ is at least as close to $v$ as $\sum_{j \in J} u_j$, we conclude that  $J_0 = \emptyset$.

Suppose for a contradiction that we have $j_1 \in J_1$. Let $J_2 = \{j \in J \mid g_{j}a \notin \bdd(g_{j_1}a)\}$. Note that $0 \in J_2$. By the same argument as above, we have $\sum_{j \in J \setminus J_2}u_j \in Z$. Therefore $\sum_{j \in J_2} u_j = \sum_{j \in J} u_j - \sum_{j \in J \setminus J_2}u_j \in Z$. Since $j_1 \notin J_2$, $\sum_{j \in J_2} u_j$ is an element of $Z$ closer to $v$ than $P_Zv$ and this is a contradiction. Hence $J_1 = \emptyset$ and $P_Zv \in V$. 

Therefore, $Z \cap V$ is a nonempty $K$-invariant subspace of $V$. Since $\sigma$ is irreducible, we have $Z\cap V = V$. By $G$-invariance of $Z$, we have $g_i \otimes_{\sigma} V \subseteq Z$ for all $g_i$ and hence $Z = H_p(M)$. This proves that $\pi$ is irreducible. 
\end{proof}

\noindent \textbf{Remark:} We leave it to the reader to compare Proposition \ref{proposition: irreducibility of induced representation} with the failure of the naive criterion considered in \cite{BekkaCurtis2003}. In particular, we note that our criterion in Proposition \ref{proposition: irreducibility of induced representation} is very much tied to a specific equivalence relation   on the   type $p$.

\medskip

  We now relate asymptotic freedom and strongly minimality in continuous logic. Strong minimality has been defined in  \cite{Hanson2020} and we rephrase the definition here:
  
 \begin{definition}
  We say that a continuous logic theory $T$ with a sort $X$ is strongly minimal if for every $M \models T$, every $M$-definable function $f(x)$ in one variable into $\Rr$ has a unique generic value $\alpha$ on $X$, meaning that for any $\beta \neq \alpha$, the set $\{f(x) = \beta\}$ is compact.
 \end{definition}

 \begin{lemma}
Suppose $T$ is a strongly minimal continuous logic theory. Then for any $M \models T$, $\bdd$ is a pregeometry.
 \end{lemma}
 
 \begin{proof}
We check the exchange property, i.e. for $A \subseteq M$ and $a,b \in M$, if $a \in \bdd(Ab)$ and $a \notin \bdd(A)$, then $b \in \bdd(Aa)$. 

Take $a, b $ in $M$ such that $a \in \bdd(Ab)$ and $a \notin \bdd(A)$. Recall that a complete type $q$ over $Ab$ is uniquely determined by the values $f(q)$ where $f$ ranges over $Ab$-definable functions into $\Rr$. By strong minimality, there is an $Ab$-definable function $f$ with generic value $\alpha$ such that the value $f(a)$ is not generic. By a standard approximation argument, we can find an $A$-definable function $g(x, y)$ such that $g(a, b)$ is not the generic value of $g(x, b)$. 

 Observe that there is some $\alpha_y$ such that for any $c \notin \bdd(A)$, $\alpha_y$ is the generic value of $g(c, y)$. There is a corresponding value $\alpha_x$ for $g(x, c)$, and it follows that $\alpha_x = \alpha_y$. Therefore $g(a, b)$ is not the generic value of $g(a, y)$ and $b \in \bdd(Aa)$. 
 \end{proof}

 \begin{definition}
 Let $T$ be strongly minimal. We say that $T$ is disintegrated if the pregeometry is trivial on $T$, in the sense that  for any $M \models T$ and $A , B \subseteq M$, $\bdd(A \cup B) = \bdd(A)   \cup \bdd(B)$  
 \end{definition}

\begin{proposition}\label{proposition: asymptotically free type is strongly minimal disintegrated}

Let   $T$ be a continuous logic theory with a single sort $X$, and a single $\Cc$-valued binary function $b$ on $X$.
Assume $b$ is positive definite, and that $X$ is an asymptotically free complete type.   Then  
$T$ is   strongly minimal  and disintegrated.
\end{proposition}

\begin{proof}
 Let $M \models T$. We consider the equivalence relation $\sim$ on $X$ defined as the transitive closure of the relation $\langle x, y \rangle \neq 0$. For $x \in X$, write $[x]$ for the $\sim$-class of $x$.   
Since $X$ is a complete asymptotically free type, if $a \in X(M)$ then $[a] \subset bdd(a)$; within $bdd(a)$, $[a]$
is the complement of a $\bigwedge$-definable set over $a$.  If $b \in X(N)$ for some $N \models T$, there exists
an elementary isomorphism $bdd(a) \to bdd(b)$; clearly it carries $[a]$ to $[b]$;  
so any two $\sim$-classes are isomorphic and the isomorphism type of these classes does not depend on the choice of model $M$.   

Let $N$ be a model of $T$,  and let $\{x_i \mid i \in I\}$ be a set of representatives of $\sim$-classes in $X(N)$. Then any bijection $I \to I$ can be extended to an automorphism of $M$. Hence   for $A \subseteq X$, $\bdd(A) = \bigcup_{a \in A}[a]$. It follows   that $T^-$ is strongly minimal and disintegrated.
\end{proof} 

Finally, we show how asymptotic freedom is a natural representation theoretic notion. Let $\sigma$ be a representation of a group $G$ on $H$ and take $v \in H$. We say that the orbit of $v$ is \emph{asymptotically free} if for every $\alpha > 0$, the set of conjugates $w$ of $v$ such that $|\langle w, v\rangle | \geq \alpha$ is precompact. 

  We have already seen that representations with asymptotically free orbits capture   representations of automorphism groups of measurable structures on the associated $L^2$-spaces, and representations of absolute Galois groups on various associated $L^2$-spaces. In Section \ref{subsection: omega-categorical structures}, we show that representatiions with asymptotically free orbits also capture all representations of automorphism groups. Finally, thanks to Theorem 6.1 in \cite{HoweMoore1979}, we show that irreducible unitary representations of certain algebraic groups also have asymptotically free orbits:
   
\begin{proposition}[\cite{HoweMoore1979}, Theorem 6.1]\label{proposition: representations of algebraic groups are asymptotically free}
Let $\sigma$ be an   irreducible representations of an algebraic group over a local field of characteristic $0$. Then every orbit of $\sigma$ is asymptotically free.
\end{proposition}
   
\begin{proof}

Let $G$ be a connected algebraic  group over a local field of characteristic $0$ with a representation $\sigma$ on $H$. Let $P \leq G$ be the preimage under $\sigma$ of the circle group in $U(H)$.  Then  Theorem 6.1 in \cite{HoweMoore1979} shows that for any $v \in H$, the map $g \mapsto |\langle \sigma(g)v, v\rangle|$ tends to $0$ on $G/P$. Since the action by $P$   does not affect compactness, we deduce that the orbit of $v$ is asymptotically free.

If $G$ is not connected,  we   find a connected normal algebraic subgroup $G_0 $ such that $G/G_0$ is finite. 
Then $\sigma$ splits as a finite orthogonal sum of   irreducible representations of $G_0$. 
For any $v \in H$, the $G$-orbit of $v$ is a finite union of $G_0$ orbits in the irreducible subrepresentations and hence the $G$-orbit of $v$ is asymptotically free.
\end{proof}

  See \cite{Bekka2000} for an overview of the Howe-Moore result and its extension to various additional cases.

Combining Proposition \ref{proposition: representations of algebraic groups are asymptotically free} and Proposition \ref{proposition: asymptotically free type is strongly minimal disintegrated}, we find:

\begin{corollary}
Let $G$ be an algebraic group over a local field of characteristic $0$. Then any irreducible representation of $G$ can be understood as a continuous homomorphism $\rho : G \to \Aut(X)$ for some strongly minimal disintegrated $X$ possessing a definable positive definite function. 

$D$ and $\rho$ depend only on the choice of $G$ orbit in the representation of $G$.
\end{corollary}

\subsection{Unitary representations of automorphism groups of $\omega$-categorical structures} \label{subsection: omega-categorical structures}

\emph{In this section, we recall some results  of \cite{Tsankov2012} and \cite{Ibarlucia2021} about unitary representations of automorphisms groups of $\omega$-categorical structures and we show that Corollary \ref{corollary: orthogonal structure for omega-categorical theories} combined with Proposition \ref{proposition: representation theory of asymptotically free type} recovers the classification theorem in \cite{Tsankov2012}.}

\medskip

We refer the reader to \cite{BenYaacov2008} for standard facts concerning $\omega$-categorical structures in continuous logic.   Recall   the Ryll-Nardzewski theorem in continuous logic which says that $T$ is $\omega$-categorical if and only if every complete type is principal (see \cite{BenYaacov2008} 12.10).  See also \cite{BenYaacov2008} 12.11 whcih shows that if $M$ is the separable model of an $\omega$-categorical theory, then $M$ is $\omega$-near-homogeneous.

For an $\omega$-categorical continuous logic theory $T$,  we define the expansion $T^{princ}$ as in Definition  \ref{definition: expansion by a distance-definable set} by adding $p$ as a new sort to $T$, where $p$ is any principal type of $T$.     The following result is Lemma 1.1 in \cite{Ibarlucia2021}  rephrased in the language of this paper.

\begin{lemma}
Let $T$ be an $\omega$-categorical continuous logic theory and let $M$ be an $\omega$-near-homogeneous model of $T$.
Let $\sigma$ be a representation of $\Aut(M)$ on a Hilbert space $H$. Then there is an interpretable Hilbert space $\mathcal{H}$ in $T^{princ}$ such that $\sigma$ is equivalent to the canonical representation of   $\Aut(M)$ on $H(M)$.
\end{lemma}

If $T$ is a classical logic theory, then the $T^{princ}$ construction is not needed, since a principal type is a definable set $D$ and we can  define the interpretation map outside of $D$ to be the trivial $0$ map.  Now the following lemma is a combination of Lemma 3.1 in \cite{Tsankov2012} and Lemma 1.1 in \cite{Ibarlucia2021}:

\begin{lemma}\label{lemma: unitary representations strictly interpretable in classical omega categorical theory}
Let $T$ be a classical logic $\omega$-categorical theory and let $M$ be an   $\omega$-homogeneous model of $T$. Let $\sigma$ be a unitary representation of $\Aut(M)$ on a Hilbert space $H$. Then there is a strictly interpretable Hilbert space $\mathcal{H}$ in  $T$ such that $\sigma$ is equivalent to the canonical representation of $\Aut(M)$ on $H(M)$.
\end{lemma}

Let $T$ be a classical logic $\omega$-categorical theory and let $M$ be an $\omega$-homogeneous model of $T$. Note that if $p$ is a type in a classical imaginary sort of $M$, then the relation $x \in \acl(y)$ is symmetric  and transitive on $p$. This is because $\acl(x)\cap p$ is a finite set with fixed cardinality.  Applying Corollary \ref{lemma: unitary representations strictly interpretable in classical omega categorical theory}, Corollary \ref{corollary: orthogonal structure for omega-categorical theories} and Proposition \ref{proposition: representation theory of asymptotically free type}, we deduce directly that every unitary representation of $\Aut(M)$ is an orthogonal sum of  representations obtained by inflation from representations of   groups of partial automorphisms of finite sets of the form $\acl(a) \cap \tp(a)$, where $a$ is a classical imaginary element of $M$.  This is precisely the classification theorem 5.2 in \cite{Tsankov2012}.
 
It remains to be seen if it is possible to build on the techniques developed in this paper in order to find a classification of the unitary representations of continuous logic $\omega$-categorical structures. This is an open question for future research.

\appendix
\section{Continuous logic}

In \emph{any} version of continuous logic, one has the notion of a {\em type-definable set}.  The collection of type-definable sets is closed under positive Boolean combinations and under projections.   As in the discrete logic case, the projection $(\exists x)p(x,y)$ of a partial type $p(x,y)$ is  the partial type $q(y)$ such that, {\em in a sufficiently saturated model}, $q(b)$ holds iff there exists $a$ with $p(a,b)$. We can thus freely write formulas or sentences involving positive first order operations.  These describe partial types in any formulation of continuous logic, regardless of a specific calculus.

With this in mind, we give a presentation of continuous logic based on the approach of \cite{HensonIovino2002}. A similar approach was recently used in  \cite{Gomez2021}. This approach uses the syntax of classical logic and uses classical results of model theory to deduce corresponding results in continuous logic. This formalism  is completely equivalent to the approach of \cite{BenYaacov2008} and this paper is written so that it is always easy to translate any argument into the formalism of \cite{BenYaacov2008} if one wishes to do so.

In \ref{appendix: continuous logic}, we define syntax and type-spaces and we discuss briefly the relation to  \cite{HensonIovino2002}. In \ref{appendix: usual model theoretic notions}, we recall some   model theoretic notions which are used in this paper. This section is mainly based on \cite{BenYaacov2008}. In \ref{appendix: Hilbert spaces in continuous logic}, we recall some classical facts about the model theory of Hilbert spaces.

\subsection{Continuous logic and type-spaces}\label{appendix: continuous logic}
\emph{In this section we introduce the basic concepts of continuous logic. Since we take classical discrete logic as our starting point, we will use the qualifiers `classical logic...' and  `continuous logic...' to highlight the differences between the two logics. We do not use this terminology elsewhere in this paper.}


 \medskip

In continuous logic, we will work with  positive formulas (and their negations):
\begin{definition}
Let $\mathcal{L}$ be an arbitrary language for classical first-order logic. We say that an $\mathcal{L}$-formula $\phi(x)$ is positive if $\phi(x)$ is logically equivalent to a formula which uses only the logical connectives $\wedge$ and $\vee$ and the usual quantifiers $\forall$ and $\exists$.
\end{definition}
Note that if $\phi$ and $\psi$ are positive $\mathcal{L}$-formulas, then $(\neg \phi \rightarrow \psi)$ is positive. In general, positive formulas can have very weak expressive power but we will always work in certain languages and theories where they have strong expressive power.  The restrictions we impose on $\mathcal{L}$ are the following.

We always fix a multi-sorted language $\mathcal{L}$ with sorts $(S_i)_{i\in I}$ and $(I_n)_{n \geq 1}$. We refer to the sorts $(S_i)$ as the \emph{metric sorts} and to the sorts $I_n$ as the \emph{value sorts}.   $\mathcal{L}$ has an equality relation on every value sort but not on any metric sort.   Each sort $I_n$ will be identified with the interval $[-n, n]$, so we add functions $i_{nm}: I_n \to I_m$ for $n\leq m$ which will play the role of inclusion functions. The value sorts are also equipped with the following structure of the real numbers: functions $+,- ,\times,  \max$ between the appropriate value sorts and predicates $= $ and $\leq$ (note that we choose $\leq$ and not $<$, for reasons which become clear below). In each value sort $I_n$ we also add a constant symbol for each rational number in $[-n, n]$.  Each metric sort is equipped with a function $d_i : S_i \times S_i \to I_n$ for some $n$. $\mathcal{L}$ may contain more function symbols but no  other relation symbols.

Whenever we fix $\mathcal{L}$, we also fix a \emph{minimal $\mathcal{L}$-theory} $T_\mathcal{L}$. This is a collection of $\mathcal{L}$-sentences $T_\mathcal{L} = T_0 \cup T_\Rr$ satisfying the following conditions:
\begin{enumerate}
\item   $T_\Rr$ says that the value sorts satisfy the full first-order theory of the real numbers (where $I_n$ is identified with $[-n, n]$)   
 
\item $T_0$ says that every $d_i$ is a pseudometric on $ S_i$ and that $d_i$ is bounded by some rational $c_i$. 

\item $T_0$ says that every function symbol $f$ in $\mathcal{L}$ is a uniformly continuous function in the following way: for every rational $\e > 0$ there is a  rational $\delta > 0$ such that $T_\mathcal{L}$ contains the positive sentence
\[
  \forall x, y( d(x, y) < \delta \rightarrow d_i(f(x), f(y)) \leq \e)
\]
where $x, y$ are finite tuples of variables appropriate for $f$ and $d$ is the max-metric on the sorts corresponding to    the tuple $x$ and $d_i$ is the metric on the sort of $f(x)$.
\end{enumerate}
When working with $\mathcal{L}$, we only ever consider models of $T_\mathcal{L}$.
\begin{definition}
A continuous logic $\mathcal{L}$-structure $M$ is an $\mathcal{L}$-structure in the usual sense of classical first order logic such that $M \models T_\mathcal{L}$, the value sorts of $M$ are the standard real numbers, and each metric sort of $M$ is a complete metric space.
\end{definition}

A useful way of constructing continuous logic $\mathcal{L}$-structures is to quotient sufficiently saturated classical $\mathcal{L}$-structures: given any model $M$ of $T_\mathcal{L}$, we define $\tilde{M}$ to be the structure obtained by quotienting every sort of $M$ by the $\bigwedge$-definable equivalence  relation $E(x, y)$ which says that $x$ and $y$ are within distance $\leq 1/n$ for every $n\geq 0$. It is easy to check that for every function symbol $f\in \mathcal{L}$, the interpretation $f^M$ of $f$ in $M$ induces a uniformly continuous function $f^{\tilde{M}}$ on $\tilde{M}$. Therefore $\tilde{M}$ is a classical $\mathcal{L}$-structure and $\tilde{M} \models T_\mathcal{L}$ which respects the uniform continuity conditions of $T_\mathcal{L}$.

Therefore, if $M \models T_\mathcal{L}$ then $\tilde{M} \models T_\mathcal{L}$. If $M$ is $\omega_1$-saturated, then the value sorts of $\tilde{M}$ are the standard real numbers and each sort of $\tilde{M}$ is  a complete metric space. Therefore $\tilde{M}$ is a continuous logic $\mathcal{L}$-structure. We say that $\tilde{M}$ is the \emph{standardisation} of $M$.

 The key property of positive formulas is the following: if $\phi(x)$ is any positive formula  and $M \models \phi(a)$ where $a \in M$ and $M$ is a classical $\mathcal{L}$-structure, then $\tilde{M} \models \phi(\tilde{a})$ where $\tilde{a}$ is the equivalence class of $a$ in $\tilde{M}$.

\medskip 
 
For a continuous logic language $\mathcal{L}$, a positive formula $\phi(x)$ in $\mathcal{L}$, and a rational $\e > 0$, we   define the \emph{$\e$-approximation of $\phi(x)$}, written $\phi^\e(x)$, as the formula obtained from $\phi$ by weakening all the bounds mentioned in $\phi$ by $\e$. This construction is adapted from Section 5 in \cite{HensonIovino2002}. Explicitly, $\phi^\e$ is defined up to logical equivalence inductively as follows:
\begin{enumerate}
\item if $\phi(x)$ is of the form $f(x) = r$ where $r$ is a rational and $f(x)$ is an $\mathcal{L}$-term, then $\phi^\e(x)$ is $|f(x) - r| \leq \e$
\item if $\phi(x)$ is of the form $f(x) \leq r$, then $\phi^\e(x)$ is $f(x) \leq r + \e$
\item if $\phi(x) = \psi(x) \wedge \chi(x)$, then $\phi^\e(x) = \psi^\e(x) \wedge \chi^\e(x)$, and similarly if $\phi(x) = \psi(x) \vee \chi(x)$
\item if $\phi(x) = (\exists y) \psi(x, y)$, then $\phi^\e(x)  = (\exists y)\psi^\e(x, y)$, and similarly if $\phi(x) = (\forall y) \psi(x, y)$.
\end{enumerate}
Note that in any   $M \models T_\mathcal{L}$, we always have  $\phi(M) \subseteq \phi^\e(M)$ for any $\e > 0$.  

In continuous logic, we are usually interested in knowing if a formula $\phi(x)$ is \emph{approximately satisfied} by a point $a$ in $M$. This means that for all $\e > 0$, $M \models \phi^\e(a)$, where $\models$ is meant in the usual sense of classical logic. The difference between satisfaction and approximate satisfaction is the same as the difference between saying $(\exists y) f(a, y) = 0$ and $\inf_y |f(a, y) |= 0$. 
\cite{HensonIovino2002} take the route of defining a new relation of approximate satisfaction inside a continuous logic structure (see section 5 in \cite{HensonIovino2002}). In this paper, we prefer to keep the usual notion of satisfaction but we work with approximations of formulas.

In Section 5 of \cite{HensonIovino2002}, the authors define the \emph{weak negation} $neg(\phi)(x)$ of a positive $\mathcal{L}$-formula $\phi(x)$ as follows: 
\begin{enumerate}
\item if $\phi(x)$ is of the form $f(x) = r$, then $neg(\phi)(x)$ is $|f(x) - r| \geq 0$
\item if $\phi(x)$ is of the form $f(x) \leq r$, then $neg(\phi)$ is $f(x) \geq r$
\item if $\phi(x) = \psi(x) \wedge \chi(x)$ then $neg(\phi)(x) = neg(\psi)(x) \vee neg(\chi)(x)$, and similarly if $\phi(x) = \psi(x) \vee \chi(x)$. 
\item if $\phi(x) = (\exists y) \psi(x, y)$, then $neg(\phi) = (\forall y) neg(\psi)(x, y)$, and similarly if $\phi(x) = (\forall y) \psi(x, y)$.
\end{enumerate}

Note that $neg(\phi)$ is a positive formula and in any model of $T_{\mathcal{L}}$, $\neg \phi \subseteq neg(\phi)$. It is easy to check that for any continuous logic structure $M$ and $a \in M$, for any positive formula, for all $\e > 0$, either $M \models \phi^\e(a)$ or $M \models neg(\phi^\e(a))$ and that for $\e < \delta$, $\phi^\e$ and $neg(\phi^\delta)$ are inconsistent. This shows that continuous logic languages have strong expressive power.


\begin{definition}
Let $T$ be a consistent collection of $\mathcal{L}$-sentences. We say that $T$ is a continuous logic $\mathcal{L}$-theory if $T 
\supseteq T_\mathcal{L}$ and for every $\phi\in T$, either $\phi \in T_\mathcal{L}$ or there is some positive sentence $\psi$ and $\e > 0$ such that $\phi = \psi^\e$

We say that $T$ is a complete continuous logic $\mathcal{L}$-theory if for every positive $\mathcal{L}$-sentence $\phi$, for every $\e > 0$, $T$ contains either $\phi^\e$ or $neg(\phi^\e)$.

When $M$ is a continuous logic $\mathcal{L}$-structure, we write $Th(M) = \{\phi^\e \mid  M \models \phi, \e > 0\}$

\end{definition}

Therefore, a complete continuous logic theory $T$ is not maximal consistent with respect to positive sentences. Nevertheless, if $M \models T$ is a continuous logic structure and $N$ is a nonprincipal ultrapower of $M$, then the set of positive sentences true in $\tilde{N}$ is maximal consistent and does not depend on $M$.

Let $M, N$ be continuous logic  structures.  We say that   $N$   is an \emph{elementary extension} of $M$ (and we write $M \prec N$) if $M$ is a substructure of $N$ in the usual sense and for any tuple $a$ in $M$ and any positive formula $\phi(x)$, if $M \models \phi(a)$ then $N \models \phi(a)$.  Note that $Th(M) = Th(N)$ although more positive formulas and sentences may be true in $N$ than in $M$.

\begin{definition}
Let $T$ be a continuous logic theory. A continuous logic type-definable set $p(x)$ is a collection of positive formulas consistent with $T$ such that for every $\phi(x) \in p(x)$ there is some positive $\psi(x)$ and $\e > 0$ and $\phi = \psi^\e$.

A continuous logic complete type $p(x)$ is a continuous logic type-definable set such that for any positive formula $\phi(x)$ and $\e > 0$,  $p$ contains either $\phi^\e$ or $neg(\phi^\e)$.

When $M \models T$ is a continuous logic structure, $a \in M$ and $ A \subseteq M$, we write $\tp(a/A)$ for the continuous logic type of $a$ over $A$. This is the set of formulas $\phi^\e(x)$ such that $\phi(x)$ is a positive formula over $A$ satisfied by $a$.

We write $S_x(T)$ for the space of continuous logic complete types in $x$. 
\end{definition}

Equivalently, we could define a  continuous logic   type $p(x)$ as a continuous logic type-definable set such that there is a maximal consistent set $q(x)$ of positive $\mathcal{L}$-formulas  with $p  = \{\phi^\e \mid \phi  \in q \}$.

Saturation is defined as expected for continuous logic structures, with respect to continuous logic types. Note that the existence of saturated continuous logic  structures follows from the existence of saturated models in classical first order logic: if $M \models T$ is  $\kappa$-saturated in the usual sense of classical logic, it is easy to show that the standardisation  $\tilde{M}$ is $\kappa$-saturated as a continuous logic structure. Homogeneous continuous logic structures are constructed in the same way.

We know from classical logic that $S_x(T)$ is a compact topological space. A basis of  closed sets of $S_x(T)$ is given by the formulas contained in the   continuous logic types of $T$. Moreover, our discussion of the weak negation $neg(\phi)$ shows that $S_x(T)$ is Hausdorff.
%
%
%


For the purpose of this appendix, write $S_{pos}(T)$ for the set of maximal consistent types of positive formulas. For every $p(x) \in S_{pos}(T)$, define $\tilde{p} = \{\phi^\e \mid \phi \in p\} \in S_x(T)$. 
When $M \models T$ is an $\omega_1$-saturated continuous logic structure, every tuple $a \in M$ realises some $p \in S_{pos}(T)$. Therefore, $S_{pos}(T)$ and $S_x(T)$ are homeomorphic topological spaces via the map $p \mapsto \tilde{p}$. 

In this paper we    often  work in $\omega_1$-saturated continuous logic structures. In this context, continuous logic is equivalent to working with the fragment of positive formulas  in classical logic, as is clear from the definitions we have laid down. Arguments about the type space $S_x(T)$ can be streamlined by discussing the space $S_{pos}(T)$  and working with arbitrary positive formulas. Whenever deducing results about non-saturated continuous logic structures, we are careful to introduce approximations of formulas to transfer the results appropriately.


  \medskip
\begin{changemargin}{0.7cm}{0.7cm}     
\emph{
In the remainder of this appendix, we work in continuous logic, so we say `structure' instead of `continuous logic structure', `theory' instead of `continuous logic theory',  `  type' instead of ` continuous logic type', etc. }
 \end{changemargin}

\subsection{Standard facts and definitions  in continuous logic}\label{appendix: usual model theoretic notions}
 Many basic results in classical logic go through to continuous logic unchanged.   We record some  definitions and results which are used in this paper.   In this section, $T$ always denotes a complete continuous logic theory.

\subsubsection{Bounded and definable closure} 
 
\begin{definition}\label{definition: bdd}
Let $M \models T$  and take $A\subseteq M$.   We say that a tuple $b\in M$ is in the bounded closure of $A$ if for every     elementary extension $N$ of $M$, there is no infinite indiscernible sequence realising $\tp (b /A)$ in $N$. 
We write $\bdd(A)$ for the bounded closure of $A$ in $M$. 

We say that a tuple $c \in M$ is in the definable closure of $A$ if for every elementary extension $N$ of $M$, $c$ is the only realisation of $\tp(c/A)$ in $N$. We write $\dcl(A)$ for the definable closure of $A$ in $M$.
\end{definition}

See 10.7 and 10.8 in \cite{BenYaacov2008} for standard results about definable and  bounded closure. Note in particular that if $M \prec N$ and $A \subseteq M$ then $\bdd(A)$ is the same set in $M$ and in $N$, and similarly for $\dcl(A)$.

\subsubsection{Definable functions}
\begin{definition}\label{definition: definable function}
Let $M$ be a    $\kappa$-saturated model of $T$ and take $A\subseteq M$    of size $< \kappa$. Let $X$ and $Y$ be finite Cartesian products of sorts of $T$ and let $p$ be a type-definable subset of $X$ over $A$.
 We say that a function $f : p \to Y$   is definable over $A$ if the set $\{(x, f(x)) \mid x \in p\}$ is    type-definable   over $A$.

When $M$ contains $A$ but isn't sufficiently saturated, we say that $f$ is  definable over $A$ if there is a sufficiently saturated elementary extension $N$ of $M$ and a definable function on $N$ which restricts to $f$.

\end{definition}

When $f : p \to Y$ is definable over $A$, we often identify $f$ with its graph in the type space $S_{xy}(A)$.   $f$ is definable on $p$ over $A$ if and only if the function $p \times Y \to \Rr, (x, y) \mapsto d(f(x), y)$ is definable on $p$ over $A$. See 9.24 in \cite{BenYaacov2008}.

In the special case where $Y$ is  the value sort of $T$ corresponding to the interval $[-n, n]$, the  definable functions $p \to Y$ over $A$  are exactly the continuous functions   $S_x(A) \cap p \to[-n, n]$. This is because the type-space of a $Y$ can always be identified with  the interval of real numbers $[-n, n]$. 

Since type-spaces are compact Hausdorff topological spaces, any complete type $q$ in $S_x(A)$ is uniquely determined by the values $f(q)$ where $f$ ranges over the $A$-definable functions $X \to \Rr$.
Urysohn's lemma also entails that any definable $f : p \to \Rr$ extends to a continuous function $S_x(A) \to \Rr$ so the local definition of $f$ on $p$ is not usually relevant. This is a significant difference with general definable functions $p \to Y$.

A useful technical fact is that any $A$-definable   $f : X \to \Rr$  is the uniform limit of a sequence of $A$-definable functions $(f_n)$ on $X$ such that for all $n$, there is a finite tuple $a_n$ in $A$ and a  function $g_n(x, y)$ definable by a term in the language $\mathcal{L}$ such that $f_n(x) = g_n(x, a_n)$.  To see this, note that such $g_n(x, a_n)$ form a lattice of functions on $S_x(A)$ which separate points. Therefore the Stone-Weierstrass theorem applies. One consequence of this is that  any $A$-definable $f : X \to \Rr$ is definable over a countable subset of $A$.

When working with  definable functions, we will often write down   formulas which contain symbols for these functions, e.g $M \models f(a) = 0$. This is a slight abuse of notation, especially when the functions are only defined on a type-definable set $p$. \emph{These expressions are meant as shorthand for type-definable sets in $\mathcal{L}$.}


As a final comment on definable functions, let $M \models T$, $A \subseteq N$, let $p$ be a type-definable set over $A$, and let $f : p \to Y$ be $A$-definable, where $Y$ is any finite product of sorts of $T$. We have seen that $f$ is defined at the level of the type-space or equivalently at the level of an elementary extension of $M$. Nevertheless, if $p(M) \neq \emptyset$ then $f : p(M) \to Y(M)$ is a total function. This is because $f(a) \in \dcl(Aa)$ for all $a \in M$ and we have seen that $\dcl(Aa)$ does not depend on $M$.

\subsubsection{Canonical parameters and imaginaries}\label{subsection: imaginaries in CL}

In this paper, we make a slightly non-standard use of the notion of canonical parameter. 
\begin{definition}\label{definition: imaginary element}
Let $M \models T$, let $A\subseteq M$ and let $f : X \to Y$ be   $A$-definable. We say that a single element $c \in M$ is a canonical parameter for $f$ if    for any   elementary extension  $N$ of $M$,  any automorphism of $N$ preserves $f$ if and only if it fixes $c$. 
 \end{definition}

This definition is slightly non-standard for the reason that we do not allow tuples of elements as canonical parameters. This is because canonical parameters consisting of exactly one element   play an important role in this paper. We obtain canonical parameters  by adding imaginary sorts whenever we need them. Imaginary sorts in continuous logic are a special case of hyperimaginary sorts in classical logic. See  \cite{BenYaacov2010} Section 5 for a clear comparison. 

\begin{definition}\label{definition: imaginary sort}

  An imaginary sort $S$ of $T$ is a Cartesian product of at most countably many sorts $(S_n)$ of $T$ endowed with a pseudo-metric $d$ such that there is an increasing sequence $(n_k)$ in $\Nn$ and definable pseudo-metrics $d_k$ on $\prod_{i =0}^{n_k} S_i$ such that the pseudo-metrics $(d_k)$ converge uniformly to $d$ on $S$. 
  
  This means that for any $\e > 0$, there is $K \geq 0$ such that for all $k \geq K$, in any $M \models T$, for any $(a_n), (b_n) \in \prod_{n \geq 0} S_n$, $|d((a_n), (b_n)) - d_k((a_n)_{n \leq n_k}, (b_n)_{n \leq n_k})| < \e$.
  

\end{definition}

If $S$ is an imaginary sort of $M$ with metric $d$, expressed as the product of the sorts $(M_n)$, we can add the sort $S$ to the language with the metric $d$ and projection maps $\pi_n : S \to M_n$. This construction is carried out in detail in \cite{BenYaacov2010}.

We   use  imaginary sorts   in two ways. Firstly, we can   use an imaginary sort to add a countable Cartesian product of metric spaces to the language: if $d_n$ is a metric on $M_n$ with diameter $1$ and $S$ is the product of the sorts $(M_n)$, then we can define $d(x, y) = \sum d_n(x_n, y_n)/2^n$. 

 Secondly, we can use imaginary sorts to add canonical parameters for arbitrary definable functions. If $f$ is a definable function over a countable set $A \subseteq M$, we have seen that we can express $f$ as the uniform limit of functions $g_n(x, a_n)$ where $a_n \subseteq A$ is finite and $g$ is a term in the language. For every $n$, let $S_n$ be the product of the sorts corresponding to the tuple $a_n$ and let $d_n$ be the definable pseudometric $d_n(y, z) = \sup_x d'(g_n(x, y), g_n(x, z))$ where $d'$ is the metric on the sort of $x$. The notion of forced limit in \cite{BenYaacov2010} shows how to ensure that the metrics $d_n$ are uniformly convergent. Let $S$ be the Cartesian product of the sorts $S_n$ and let $d$ be the limit of the pseudo-metrics $d_n$. Quotienting out $S_n$ and $S$ to obtain metric spaces, $S$ is an imaginary sort of canonical parameters for $f$. See \cite{BenYaacov2010} for details.
  
 
\subsubsection{Stability and definable types}

\begin{definition}
Let $M \models T$. Take $A\subseteq M$    and let $f : X \times Y \to \Rr$ be an $A$-definable function where $X$ and $Y$ are finite Cartesian products of sorts of $M$ and $f$ takes values in $\Rr$. We say $f$ is unstable if there is some elementary extension $N$ of  $M $,  indiscernible sequences $(a_n), (b_n)$ in $N$ and $\e >0$ such that $|f(a_n, b_m) - f(a_m, b_n)| \geq \e$ for all $n\neq m$. We say $f$ is stable if $f$ is not unstable. 
\end{definition}

\begin{definition}
Let $M \models T$  and let $A, B \subseteq M$. Let $p\in S_x(A)$ and let $f(x, y)$ be an $A$-definable function into $\Rr$.  
We say that $p$ is definable over $B$ with respect to $f$ if there is a $B$-definable function $g(y)$ such that for any tuple $a$ in $A$ in the sort of $y$, $p(x)$ entails that $f(x, a) = g(a)$. In that case we write $g(y) = d_pf(y)$.


%
%
\end{definition}

In this paper we   work with   local stable independence. First developed   in \cite{Shelah1978a}   and \cite{Pillay1986}    for classical logic, local stable independence for continuous logic has roots in \cite{Shelah1975}  and was studied in \cite{BenYaacov2010}. We only recall the main definition:

\begin{definition}\label{definition: forking}
 Let $M \models T$, let $C\subseteq B \subseteq M$ and $p(x)\in S(B)$. Let $\Delta$ be a set of stable functions definable over $A$. 
 We say that $p(x)$ does not fork over $C$ with respect to $\Delta$ if we can add imaginary sorts to $M$ and extend  $p$ to $\bdd(B)$ so that $p$  is definable over $\bdd(C)$ with respect to $\Delta$.

Let $A \subseteq M$. We  write $A \ind_C^\Delta B$ if $\tp(A/BC)$ does not $\Delta$-fork over $C$.

When $\Delta$ is the set of all stable functions, we simply say that $p$ does not fork over $C$ and we write $A \ind_C B$. 
\end{definition}

We refer the reader to \cite{Pillay1986} and \cite{BenYaacov2010} for an exposition of the theory of stable independence.   In this paper we   make use of Morley sequences in the context of local stability.  

\begin{definition}\label{definition: Morley sequence}
Let $M \models T$  and let $\Delta$ be a set of stable formulas. We say that a sequence $(a_n)$ in $M$ is a Morley sequence   over $A$ with respect to $\Delta$ if for all $n$ we have $a_{n+1} \ind^\Delta_{A}a_0\ldots a_n$ and $\tp(a_n / \bdd(A)) = \tp(a_0 / \bdd(A))$.
\end{definition}
 
\subsubsection{Continuous logic and classical logic}\label{subsectionL CL and classical logic}
  
   Continuous logic is a direct generalisation of classical logic, and there is a canonical way of taking a classical logic theory $T$ and viewing it as a continuous logic theory $T^{cont}$. The construction of $T^{cont}$ is as follows.
   
    Every sort of $T$ is now viewed as a metric space with the discrete metric with diameter $1$. We remove the equality symbol from the sorts of $T$. Observe that there is no loss of information in doing this, since $x = y$ is equivalent to $d(x, y)\leq 1/2$ and $x \neq y$ is equivalent to $d(x, y) \geq 1/2$, both of which are approximations of positive formulas. We add the usual value sorts to $T$.   Function symbols in the language of $T$ are unchanged. For each relation symbol $R$ in the language of $T$, we substitute   a function symbol $f_R$ which we view as the indicator function of $R$ in the corresponding continuous logic sort. It is then clear how to axiomatise a continuous logic theory $T^{cont}$ so that there is an exact correspondence between models of $T$ and   models of $T^{cont}$. 

In this paper, when we construct a theory $T^{cont}$ from a classical logic theory $T$, we  distinguish two kinds of imaginary sorts in $T^{cont}$. Firstly we have the \emph{classical imaginary sorts of  $T^{cont}$} which come from the imaginary sorts of $T$ defined in the usual way (see \cite{TentZiegler}). 
Secondly we have the \emph{continuous logic imaginary sorts of  $T^{cont}$} which are imaginary sorts obtained by the constructions sketched in  \ref{definition: imaginary sort}. These correspond to certain hyperimaginary sorts of $T$. Unless specified otherwise, `imaginary sort' and `imaginary element' always refer to imaginaries in continuous logic, as defined in  \ref{definition: imaginary sort}.

\subsubsection{Stable embeddedness  }

We show that the classical notion of stable embeddedness adapts easily from classical logic to  continuous logic. An account of stable embeddedness in the context of classical logic can be found in  the appendix of \cite{Chatzidakis1999}. Here we take $T$ to be a complete continuous logic theory.

\begin{definition}\label{definition: stable embededness}
Let $\mathcal{D}$ be a collection of type-definable sets of $T$. We say that $\bigcup \mathcal{D}$ is stably embedded in $T$ if for any $\kappa$-saturated $M \models T$ with $|M| = \kappa$, any $M$-definable function  from a  finite Cartesian products of sets in $\mathcal{D}$ to $\Rr$ is definable over $\bigcup \mathcal{D}(M)$. 
\end{definition}

In the above definition, saturation is not essential, but it is convenient to include it. Saturation can be eliminated by considering imaginaries as in  \cite{Chatzidakis1999}. We will replicate that argument in Proposition \ref{proposition: T is stably embedded} in a more restricted setting so we only consider saturated structures for now.
  
\begin{lemma}\label{lemma: stable embedded conditions}
Let $\mathcal{D}$ be a collection of distance-definable sets in $T$ (see Definition \ref{definition: distance-definable}). The following are equivalent:
\begin{enumerate}
\item $ \bigcup \mathcal{D}$ is stably embedded in $T$.\
\item  For any $\kappa$-saturated $M  \models T$ with $|M| = \kappa > |\mathcal{L}|$, for every  finite tuple $a\in M$, $\tp(a/\bigcup\mathcal{D}(M))$ is definable over a subset  of $\bigcup\mathcal{D}(M)$ of size at most $|\mathcal{L}|$.
\item For any $\kappa$-saturated $M  \models T$ with $|M| = \kappa > |\mathcal{L}|$, for every finite tuple $a\in M$, there is a subset $C$ of $\bigcup\mathcal{D}(M)$ of size at most $|\mathcal{L}|$ such that $\tp(a/C)$ extends uniquely to $\bigcup\mathcal{D}(M)$.\
\item For any $\kappa$-saturated  $M \models T$ with $|M| = \kappa > |\mathcal{L}|$, every automorphism of $\bigcup\mathcal{D}(M)$ extends to an automorphism of $M$.
\end{enumerate}
\end{lemma}

\begin{proof}
To make notation lighter we can assume that $\mathcal{D}$ is closed under finite Cartesian products. We fix $M \models T$ $\kappa$-saturated with $|M| = \kappa$.

\medskip
$(1) \Rightarrow (2)$ Let $a\in M$. Let $f(x,y)$ be a definable function into $\Rr$ with  $y$ in the sort of $ D \in \mathcal{D}$ and $x$ in the  sort of $a$. By $(1)$ there is a $\bigcup\mathcal{D}(M)$-definable function $g(y)$ such that $g(y)  = f(a, y)$ on $D$. $g$ defines $\tp(a/D(M))$ for $f$. Moreover, $g$ is definable over a  countable $A\subseteq D(M)$, so $\tp(a/\bigcup\mathcal{D}(M))$ is definable over a subset $B$ of $\bigcup\mathcal{D}(M)$ with $|B| \leq |\mathcal{L}|$.

\medskip 
 
  $(2) \Rightarrow (1)$: Let $f(x)$ be an $M$-definable function into $\Rr$ where $x$ is a finite tuple in the sort of $D \in \mathcal{D}$. We can assume that $f(x) = g(x, b)$ where $b$ is a finite tuple in $M$ and $g$ is $0$-definable. $q = \tp(b/\bigcup \mathcal{D}(M))$ is definable over some small $C \subseteq \mathcal{D}(M)$ so we have a $C$-definable  function $d_qg(x) = f(x)$.

\medskip

$(2) \Rightarrow (3)$: Let $p(x) = \tp(a /\bigcup \mathcal{D}(M))$, $a\in M$. Suppose that $p $ is definable over $C\subseteq \bigcup \mathcal{D}(M)$. 
Let $f(x)$ be a $\bigcup\mathcal{D}(M)$-definable function into $\Rr$. We show that the restriction of $p(x)$ to $C$ determines the value of $f(x)$. 
We can assume that $f(x) = g(x, b)$ where $b$ is a finite tuple in $\bigcup \mathcal{D}(M)$ and $g(x, y)$ is $0$-definable.

By (2), $p$ is definable   over $C$ with respect to $g$  and we write $d_pg(y)$ for its definition. Write $d(y, D)$ for the definable function which gives the distance to $D$. An easy compactness argument shows that for every $\e > 0$ there is $\delta > 0$ such that $p(x)$ contains the positive formula over $C$:
\[
\forall y(d(y, D) < \delta \rightarrow |g(x, y) - d_p(g(y))| \leq \e)
\]
Therefore $p \upharpoonright C$ has a unique extension to $\bigcup \mathcal{D}(M)$.

\medskip

$(3) \Rightarrow (2)$: Let $a \in M$, write $p(x) = \tp(a / \bigcup \mathcal{D}(M))$ and let $f(x , y)$ be a definable function with $y$ in the sort of  $D \in \mathcal{D}$ and $x$ a tuple in the sort of $a$. By $(3)$, there is   $C \subseteq \bigcup \mathcal{D}(M)$ such that $p(x) $ is the unique extension of $p \upharpoonright C$ to $\bigcup \mathcal{D}(M)$. Let $N$ be an elementary extension of $M$. Suppose there are $c, c' \in N$ both realising $p \upharpoonright C$ and $b\in D(N)$ such that $|f(c, b)  - f(c', b)| \geq  \delta$ for some $\delta > 0$. Then 
\[
N \models \exists x, x',y \big((p\upharpoonright C)(x) \wedge (p \upharpoonright C)(x') \wedge D(y) \wedge |f(x, y) - f(x', y)| \geq \delta \big)
\]
The above is a type-definable set. Since $M \prec N$ and  $M$ is saturated over $C$, $M$ satisfies the same type-definable set. This contradicts (3) and this proves that $p$ has a unique extension to $\bigcup \mathcal{D}(N)$. 
Hence for any $M \prec N$, $p(x)$ is $C$-invariant in $N$.   It follows that $p(x)$ is $C$-definable.

 \medskip
 
 $(3) \Rightarrow (4):$  Let $\sigma$ be an automorphism of $\bigcup \mathcal{D}(M)$ and suppose that we have extended it to an automorphism $\sigma : \bigcup \mathcal{D}(M) \cup A \to \bigcup \mathcal{D}(M) \cup B$ where $|A| < \kappa$. Let $a \in M$. There is $C \subseteq \bigcup \mathcal{D}(M)$ with $|C| < \kappa$ such that $\tp(a / A, C)$ extends uniquely to $A \cup \bigcup \mathcal{D}(M)$. By saturation, we can find $b \in M$ such that $\sigma$ extends to an automorphism $\bigcup \mathcal{D}(M)\cup Aa \to \bigcup \mathcal{D}(M) \cup Bb$. The result follows by a back-and-forth argument.
 
 \medskip
 
 $(4) \Rightarrow (3)$: Suppose $(3)$ fails for $M  \models T$, where $M$ is $\kappa$-saturated with cardinality  $\kappa$. Fix $a \in M$ which witnesses the failure of $(3)$. Let $(a_i)_{i < \kappa}$ be an enumeration of the realisations of $\tp(a)$ in $M$. Suppose we have constructed an isomorphism $\sigma : C \to \sigma(C)$ where $C$ is a subset of $\bigcup \mathcal{D}(M)$ such that for some $\alpha < \kappa$ and all $i < \alpha$, the maps $\sigma: aC \to a_i \sigma(C)$ are not isomorphisms. 
 
 Suppose that $\sigma_1 : aC \to a_\alpha \sigma(C)$ is an isomorphism. By the failure of $(3)$ there is $a'$ in $M$ and $b\in \bigcup \mathcal{D}(M)$ such that $\tp(a/C)  = \tp(a'/C)$ and $\tp(a/bC) \neq \tp(a'/bC)$. Then we can find $b'$ such that $\sigma_1 : a'bC \to a_\alpha b'\sigma(C)$ is an isomorphism. Then we extend $\sigma$ by putting $\sigma(b) = b'$. Note that now we cannot extend $\sigma$ to $a$ by sending $a$ to $a_\alpha$. By enumerating $\bigcup \mathcal{D}(M)$, we can also make sure that after $\kappa$ iterations of this procedure $\sigma$ is defined on all of $\bigcup \mathcal{D}(M)$. This contradicts $(4)$.
 \end{proof}

\subsection{Hilbert spaces in continuous logic }\label{appendix: Hilbert spaces in continuous logic}
   We recall here   basic facts about the model theory of Hilbert spaces which we   use in this paper. We refer the reader to \cite{BenYaacov2008} for a more complete summary. In this paper we    work with  Hilbert spaces over $\Rr$, but all results can be transposed to complex Hilbert spaces without modification. 

On one presentation of the model theory of Hilbert spaces,  the language of Hilbert spaces in continuous logic consists of countably many metric sorts, which stand for balls with radius $n$ around $0$. We add appropriate inclusion maps between the metric sorts. The language consists of the usual vector space structure over $\Rr$ and a function  $\langle \cdot  , \cdot \rangle $ on each metric sort into an appropriate value sort which stands for the inner product. The axiomatisation of the theory of infinite dimensional  Hilbert spaces $T^{Hilb}$ is as expected. 
In this paper, we usually don't disinguish between a Hilbert space and a  model of $T^{Hilb}$.  

$T^{Hilb}$  is complete, has quantifier-elimination, is stable, and is totally categorical.
The theory of Hilbert spaces does not have elimination of imaginaries, but it has weak elimination of imaginaries:

\begin{lemma}[\cite{BenYaacov2004} 1.2]\label{elimination of imaginaries in Hilbert spaces}
Let $M \models T^{Hilb}$. Let $\alpha$ be a canonical parameter for an $M$-definable function $f$ in an arbitrary imaginary sort of $T^{Hilb}$. Then there is a closed subspace $H$ of $M$ such that each point of $H$ is in $\bdd(\alpha)$ and $\alpha$ is definable over $H$.
\end{lemma}

\begin{definition}\label{definition: orthogonal projection to subspace}
If $H$ is a Hilbert space and   $V \leq H$ is a closed subspace, write $P_{V}$ for the orthogonal projection onto $V$.
\end{definition}

 We will make much use of the characterisation of forking independence in Hilbert spaces. See \cite{BenYaacov2008} for a proof:

\begin{lemma}\label{lemma: forking independence in Hilbert spaces}
Let $M \models T^{Hilb}$. Let $A,B, C\subseteq M$ with $B\subseteq C$. Then $\bdd(B)$ is the closed subspace of $M$ generated by $B$ and $A\ind_B C$ if and only if for all $a\in A$, $P_{\bdd(C)}(a) = P_{\bdd(B)}(a)$. 
\end{lemma}

Finally, we recall some elementary facts about the weak topology in Hilbert spaces which we use in our proofs. If $H$ is a Hilbert space, recall that the weak topology has a sub-basis consisting of the sets 
\[
\{v \in H \mid \langle v, w\rangle\in U, w \in H, U \subseteq \Rr\text{ open}\}.
\]
Recall also that the unit ball is compact in the weak topology and that every bounded sequence in $H$ has a weakly convergent subsequence. 

\begin{lemma}\label{lemma: weak closure is set of limit points}
Let $H$ be a Hilbert space and $A \subseteq H$ a bounded subset of $H$. Then the closure of $A$ in the weak topology is the set of limit points of sequences in $A$ in the weak topology. 
\end{lemma}

\begin{proof}
Suppose $w$ is in the closure of $A$ in the weak topology. Let $v_0 \in A$ be any vector such and construct $(v_n)$ inductively as follows: given $v_0, \ldots, v_n$, find $v_{n+1} \in A$ such that $|\langle v_{n+1}, v_i\rangle - \langle w, v_i\rangle| < 1/(n+1)$ and $|\langle v_{n+1}, w \rangle - \langle w, w\rangle| < 1/(n+1)$. Since $A$ is bounded, we can assume that $(v_n)$ converges weakly to some $z \in H$. We check that $z = w$. Firstly we have
 \[
\langle z, w\rangle = \lim_n \langle v_n, w\rangle =\lim_n \lim_m \langle v_n, v_m\rangle = \lim_n\langle   v_n, z\rangle = \langle z, z\rangle.
\]
Similarly, 
\[
\langle w, w \rangle = \lim_n \lim_m \langle v_n, v_m\rangle = \langle z, z\rangle.
\]
Hence $z = w$. We also deduce that the original   sequence $(v_n)$ converges weakly to $w$ and that it is not necessary to take a subsequence.
\end{proof}

The next lemma is proved in a similar way to Lemma \ref{lemma: weak closure is set of limit points}:
\begin{lemma}\label{indiscernible sequences in Hilbert spaces}
Let $M \models T^{Hilb}$ and $A \subseteq M$ and $v \in M$. Let $(v_n)$ be an indiscernible sequence in  $\tp(v/A)$ in $M$. Then there is an orthogonal sequence $(w_n)$ in $M$ and $w\in M$ such that for all $n$, $v_n = w_n + w$, $w_n \perp A$ and  $w_n \perp w$. It follows that $(v_n)$ converges weakly to $w$ (we write $v_n \rightharpoonup w$).

Moreover, $w$ is the unique element of $M$ such that $\langle w, w\rangle = \langle v_1, v_0\rangle = \langle w, v_n\rangle$ for all $n$.
\end{lemma}

\begin{lemma}\label{lemma: convergence of morley sequences in H spaces}
Let $M\models T^{Hilb}$. Let $(v_n)$ be a Morley sequence over $A\subseteq M$. Then $(v_n)$ converges weakly to $P_{\bdd{A}}(v_0)$.
\end{lemma}

\begin{proof}
It is enough to check that $\lim_n\langle v_n, v_m\rangle = \langle P_{\bdd(A)}(v_0), v_m\rangle$ for all $m$.
For $n > m$ we have $\langle v_n, v_m \rangle = \langle P_{\bdd(A v_0 \ldots v_m)}(v_n) , v_m\rangle  = \langle P_{\bdd(A)}(v_n), v_m\rangle = \langle P_{\bdd(A)}(v_0), v_m\rangle$. 
\end{proof}

\section{More theory of interpretable Hilbert spaces}\label{appendix: more theory}

\emph{In this appendix, we elaborate on some technical notions which we only briefly touched upon in Section \ref{subsection: definition of interpretable Hilbert spaces}.  }

\medskip

%
%

Let $T$ be a continuous logic theory with an interpretable Hilbert space $\mathcal{H}$. Let $M \models T$.

\begin{definition}\label{definition: interpretable Hilbert space morphism}
Let   $H(M), H'(M)$ be   interpretable Hilbert spaces in $M$. Suppose $H(M)$ is the direct limit of $(M_j)_{j\in J}$ and $H'(M)$ is the direct limit of $(M_{j'})_{J'}$. Write $h_j :M_j \to H(M)$ and $h_j': M_{j'} \to H'(M)$ for the direct limit maps.

 We say that a map $F : H(M) \to H'(M)$ is an embedding of piecewise interpretable Hilbert spaces if $F$ is a unitary map and for all $j\in J$ and $j' \in J'$ and all $\e \geq 0$, the set $\{(x,y) \in M_j \times M'_{j'} \mid \|F( h_j x ) -  h_{j'}' y\| \leq \e\}$ is type-definable. 

If $F$ is also surjective, we say that $F$ is an isomorphism of piecewise interpretable Hilbert spaces. 
\end{definition}

\begin{lemma}\label{lemma: reformulation of morphism definition}
Let $H(M), H'(M)$ be   interpretable Hilbert spaces in $M$. Let $(M_i)_{i \in I}$ and $(M_{i'})_{i' \in I'}$ be pieces of $H(M)$ and $H'(M)$ respectively which generate $H(M)$ and $H'(M)$. Write $h_i : M_i \to H(M)$ and $h_{i'}'  : M_{i'} \to H'(M)$ for the direct limit maps.
  

A unitary map $F : H(M) \to H'(M)$ is an embedding of   interpretable Hilbert spaces if and only if for all $i\in I$ and $i' \in I'$ the    map $M_i \times M_{i'} \to \Rr, (x, y) \mapsto \langle F(h_i x), h_{i'}'y \rangle$ is definable. 
\end{lemma}

\begin{proof}
Suppose that $F$ is an embedding  of piecewise interpretable Hilbert spaces. Fix $i \in I$, $i' \in I'$ and $D$ a closed bounded subset of $\Rr$. For every $\e >0$, let $D_\e$ be the closed set $\{x\in \Rr \mid \exists y\in D, |y-x| \leq \e\}$. Let $B$ be an upper bound on $\{\|h_{i'}' y\| \mid y \in M_{i'}\}$.

By compactness, for every $\e >0$ we can find $i'_\e \in I'$ such that for every $x \in M_i$, there is $x' \in M_{i'_\e}$ such that $\|h_{i'_\e}' x' - F(h_i x)\| < \e/B$. The set $\{(x, y) \in M_i \times M_{i'} \mid \langle F(h_i x), h_{i'}' y\rangle \in D \}$ is equal to the intersection over $\e > 0$ of the sets 
\[
  \{(x, y) \in M_i \times M_{i'} \mid \exists x'\in M_{i_\e}, \|F(h_ix) - h_{i'_\e}' x'\| \leq \e/B \text{ and } \langle h_{i'_\e}' x', h_{i'}  y\rangle \in D_\e\}.
\]
Hence this set  is type-definable.

%

Conversely, suppose that $H(M)$ and $H'(M)$ respectively are the direct limits of the sorts $(M_j)_J$ and $(M_{j'})_{J'}$. We   write $(h_j)$ and $(h_{j'}')$ for the direct limit maps. Fix $j \in J$,  $j' \in J'$. For any $x \in M_j$, the element $F(h_j x)$ is uniquely determined by the collection of maps $M_{i'} \to \Rr, y \mapsto \langle F(h_j x), h_{i'}' y\rangle$, for $i' \in I'$. 

Moreover, by a standard compactness argument, for arbitrary $\e > 0$ there is $n_\e \geq 1$ such that for any $x \in M_j$, $h_jx$ is within distance $\e$ of a vector of the form $\sum_{k = 1}^{n_\e} \lambda_k h_{i_k} x_{i_k}$ where $i_k \in I$ for all $k \leq n_\e$, $x_{i_k} \in M_{i_k}$   and $|\lambda_k| \leq n_\e$.  By considering small enough $\delta$ and quantifying over elements $\sum_{k = 1}^{n_\delta} \lambda_k h_{i_k} x_{i_k}$, we find that $\{(x, y) \in M_j \times M_{j'} \mid \|F(h_j x) - h_{j'}' y\| \leq \e \}$ is type-definable if the maps $\langle F(h_ix), h'_{i'}y\rangle$ are definable.

\end{proof}

The following sharpening of Proposition \ref{proposition: equivalent definition of piecewise interpretable hilbert space} follows easily:

\begin{lemma}\label{lemma: isomorphic hilbert spaces}
Let $M$ be a continuous logic structure. As in Proposition \ref{proposition: equivalent definition of piecewise interpretable hilbert space}, suppose  we are given a Hilbert space $H$,  a collection of distinct imaginary sorts $(M_i)_{i \in I}$ of $M$ and maps $F_i : M_i \to H$  such that 
 for all $i,j \in I$, the map $M_i \times M_j \to \Rr, (x, y)\mapsto \langle F_ix, F_j y\rangle$ is definable.  

Then there is an interpretable Hilbert  space $H(M)$ which is unique up to isomorphism such that the sorts $(M_i)_{I}$ are pieces of $H(M)$,   $\bigcup h_i(M_i)$ has dense span in $H(M)$, and for $ x \in M_i$, $y\in M_j$, $\langle h_i x, h_j y\rangle_{H(M)} = \langle F_ix, F_j y\rangle_H$.
\end{lemma}

\begin{proof}
%

Let $H(M)$, $H'(M)$ be two piecewise interpretable Hilbert spaces satisfying the existence claim. Since the inner product maps on the pieces $M_i$ are identical, the uniqueness statement of the GNS theorem applies and we find an isomorphism of Hilbert spaces $F: H(M) \to  H'(M)$. Lemma \ref{lemma: reformulation of morphism definition}   applies directly  so $F$ is an isomorphism of  interpretable Hilbert spaces. 
\end{proof}

Recall that we did not require direct limits of sorts of $M$ to have injective transition maps. We now show that this does not present any significant disadvantage.

\begin{lemma}\label{lemma: transition maps can be assumed isometric}
Let $H(M)$ be an interpretable Hilbert space in $M$. Then $H(M)$ is isomorphic to an interpretable Hilbert space $H'(M)$ which is a direct limit of imaginary sorts $(M_{j'})_{J'}$ with isometric direct limit maps. 
\end{lemma}

\begin{proof}
Suppose that $H(M)$ is the direct limit of the sorts $(M_j)$ with inner product maps $(f_{ij})$. Write $h_{ij} : M_i \to M_j$ for the transition maps between the sorts $M_i, M_j$ when $i\leq j$ and $h_j : M_j \to H(M)$ for the direct limit maps. For every $j \in J$, let $M'_j$ be the imaginary sort of canonical parameters of the inner product map $f_{jj}$, where we view $f_{jj}(x, y)$ as a function in $y$ with a parameter $x$. For $i\leq j \in J$, define the transition map $h'_{ij}: M'_i \to M'_j$ as the map which takes the canonical parameter for the map $f_{ii}(a, .)$ to the canonical parameter for the map $f_{jj}(h_{ij}a, .)$. Observe $h'_{ij}$ is well-defined because $f_{ii}(a, .) = f_{ii}(b, .)$ if and only if $h_ia = h_ib$. It is clear that the direct limit $H'(M)$ of $(M'_j)_J$ satisfies the lemma.
\end{proof}

As a direct application of the construction in Proposition \ref{proposition: equivalent definition of piecewise interpretable hilbert space} and Lemma \ref{lemma: transition maps can be assumed isometric}, we have:

\begin{lemma}\label{lemma: all operations on interpretable Hilbert space are definable}
Let $H(M)$ be an interpretable Hilbert space in $M$. Then $H(M)$ is isomorphic to an interpretable Hilbert space $H'(M)$ in $M$ which is a direct limit of $(M_{j'})_{ j' \in J'}$ with isometric direct limit maps $h'_{j'}$ such that the Hilbert space operations on $H'(M)$ are piecewise bounded\footnote{\label{footnote: piecewise bounded terminology} Thanks to Arturo Rodr\'iguez Fanlo for this terminology}. This means that
\begin{enumerate}
\item for every $i, j \in J'$, there is $k \in J'$ such that $i,j \leq k$ and $h_i'(M_i) + h_j'(M_j) \subseteq h_k'(M_k)$ and the map $M_i \times M_j \to M_k, (x, y) \mapsto (h_k')^{-1}(x+y)$ is definable
\item for every $i \in J'$ and $n\geq 0$, there is $k$ such that for $x \in M_i$ and $\lambda \in [-n, n]$, $\lambda h_i' x \in h_i' M_k$ and the map $[-n,n ] \times M_i \to M_k$, $(\lambda , x) \mapsto  (h_i')^{-1} (\lambda h_i' x)$ is definable.
\end{enumerate}
\end{lemma}

\begin{proof}
 Apply the construction of Proposition \ref{proposition: equivalent definition of piecewise interpretable hilbert space} to all pieces $(M_j)_J$ of $H(M)$ with the direct limit maps $M_j \to H(M)$ to obtain $H_1(M)$ isomorphic to $H(M)$. Observe that the construction of Proposition \ref{proposition: equivalent definition of piecewise interpretable hilbert space}  implies that the operations on $H_1(M)$ are piecewise bounded. Now apply Lemma \ref{lemma: transition maps can be assumed isometric} to obtain $H'(M)$ with isometric direct limit maps.
 \end{proof}

Finally we show that an isomorphism of interpretable Hilbert spaces in an $\omega$-saturated structure induces an isomorphism at the level of the theory $T$.

\begin{lemma}
Let $T$ be a complete theory and let $\mathcal{H}$, $\mathcal{H}'$ be two interpretable Hilbert spaces in $T$. Suppose that for some $\omega$-saturated $M \models T$, $H(M)$ and $H'(M)$ are isomorphic as interpretable Hilbert spaces. Then for every  $N\models T$, $H(N)$ and $H'(N)$ are isomorphic as interpretable Hilbert spaces. 
\end{lemma}

\begin{proof}
Suppose that  $H(M)$ and $H'(M)$ are the direct limits of $(M_j)_J$ and $(M_{j'})_{J'}$ respectively. We write $h_j$ and $h_{j'}'$ for the direct limit  maps $  M_j \to H(M)$, $  M_{j'} \to H'(M)$ respectively. Write $F_M : H(M) \to H'(M)$ for the  isomorphism of interpretable Hilbert spaces. Fix $N \models T$, $j \in J$,  and take $a \in N_j$. Find $x \in M_j$ with $\tp(x) = \tp(a)$ and fix $j' \in J'$ such that $F_M h_j x \in h'_{j'} M_{j'}$.

 $F_M h_j x$ is uniquely determined by  the value $\lambda = \langle F_M h_j x, F_M h_j x\rangle$  and   the $x$-definable function $f_x: M_{j'}\to \Rr,  y\mapsto \langle F_M h_j x, h'_{j'}y\rangle$. By elementarity, for every $n$, there is 
 $b_n \in N_{j'}$ such $|\langle h'_{j'} b_n, h'_{j'} b_n\rangle - \lambda| < 2^{-n}$ and for all $y \in N_{j'}$ $|f_a(y) - \langle h'_{j'} b_n, h'_{j'} y\rangle | < 2^{-n}$. Then $(h'_{j'} b_n)$ is Cauchy and there is  $b \in N_{j'}$ such that $h'_{j'} b$ is the limit of $(h'_{j'} b_n)$. Define $F_Nh_j a = h'_{j'} b$. It is straightforward to check that $F_N$ is well-defined, definable, and  is an isomorphism of interpretable Hilbert spaces. 
\end{proof}

Finally, we give a category-theoretic reformulation of Proposition \ref{proposition: equivalent definition of piecewise interpretable hilbert space}:

\begin{lemma}\label{lemma: category theoretic reformulation of GNS}
Let $M \models T$. The categories $\mathcal{A}, \mathcal{B}, \mathcal{C}$ defined below are all equivalent: 

 Let $\mathcal{A}$ be the category of   interpretable Hilbert spaces in $M$ with  embeddings of    interpretable Hilbert spaces. 

Let $\mathcal{B}$ be the category of pairs $(H, (h_i)_{I})$ where $H$ is a Hilbert space and for every $i\in I$, $h_i : M_i \to H$ is a map on an imaginary sort of $M$ such that
 \begin{enumerate}
\item the set $\bigcup_I h_i(M_i)$ has dense span in $H$
\item for all $i_1, i_2 \in I$, the map $M_{i_1} \times M_{i_2} \to \Rr, (x,y) \mapsto \langle h_{i_1} x , h_{i_2} y\rangle$ is definable.
\end{enumerate}
The morphisms between objects $(H, (h_i)_I) $ and $(H', (h_{i'})_{I'})$ of $\mathcal{B}$ are unitary maps $F : H \to H'$ such that for all $i \in I$, $i' \in I'$ the map $M_i \times M_{i'} \to \Rr, (x, y) \mapsto \langle F(h_ix), h_{i'} y \rangle$ is definable.  

Let $\mathcal{C}$ be the category  of pairs $((M_i)_{I}, (f_{ij})_{i,j \in I})$ where $(M_i)_I$ is a set of imaginary sorts of $M$ and the $f_{ij} : M_i \times M_j \to \Rr$ are definable functions such that their concatenation $f : \bigcup_I M_i \times \bigcup_I M_i \to \Rr$ is positive semidefinite. The morphisms between objects $((M _i)_I, (f_{ij}))$ and $((M_{i'})_{I'}, f_{i'j'})$ of $\mathcal{C}$ are piecewise  definable functions $G : \bigcup_{k \in I\cup I'} M_k \times\bigcup_{k \in I\cup I'} M_k \to \Rr$   such that 
\begin{enumerate}
\item $G$ is positive semidefinite and $G$ extends each function $f_{ij}$ where $i,j \in I$ or $i,j \in I'$  

\item writing $H_I$ and $H_{I \cup I'}$ for the Hilbert spaces induced   by $\bigcup_I M_i$ and $\bigcup_{k \in I \cup I'} M_k$ respectively as in the GNS theorem, the resulting Hilbert space embedding $F: H_{I'} \to H_{I \cup I'}$ induced from $G$ is surjective
\footnote{\label{footnote: induced embedding from GNS theorem} The GNS theorem gives Hilbert spaces $H_I$, $H_{I'}$ and $H_{I \cup I'}$  and maps $a : \bigcup_I M_i \to H_I$, $b : \bigcup_{I'} M_{i'} \to H_{I'}$,  $c: \bigcup_{k \in I \cup I'} M_k \to H_{I \cup I'}$ and Hilbert space embeddings $\phi : H_I \to H_{I \cup I'}$ and $\phi' : H_{I'} \to H_{I\cup I'}$ such that the corresponding diagram commutes. So we have $G(x, y) = \langle \phi\circ a(x), \phi'\circ b(y)\rangle$ when $x \in \bigcup_I M_i$ and $y \in \bigcup_{I'} M_{i'}$}
\end{enumerate}
Composition of morphisms in $\mathcal{C}$ is induced by the   GNS theorem\footnote{\label{footnote: category C} It is possible to give a definition of the morphisms of $\mathcal{C}$ and their composition which does not rely explicitly on the GNS theorem by using the Gram-Schmidt orthogonalisation process and Bessel's inequality. The details are left to the   reader}.

\end{lemma}

\bibliographystyle{alpha}
\bibliography{Hilbertspaces-arxivv1}

\end{document}